\documentclass[12pt]{amsart}
\usepackage[margin=1in,includeheadfoot]{geometry}
\usepackage{amsmath}
\usepackage{slashed}
\usepackage{graphicx}

\usepackage[margin=1in,includeheadfoot]{geometry}
\usepackage{amsmath}
\usepackage{slashed}
\usepackage{graphicx}
\usepackage{mathrsfs}
\usepackage{txfonts}
\usepackage{amssymb,dsfont}
\usepackage{enumerate}
\usepackage{slashed}
\usepackage{fancyhdr}
\usepackage{aliascnt}
\usepackage{pdfsync}
\usepackage{hyperref}
\usepackage{color}
\newtheorem{thm}{Theorem}[section]
\newtheorem{cor}[thm]{Corollary}
\newtheorem{lem}[thm]{Lemma}
\newtheorem{prop}[thm]{Proposition}
\theoremstyle{definition}

\theoremstyle{remark}

\numberwithin{equation}{section}

\newcommand{\eps}{\epsilon}

\newcounter{stepnum}

\def\bee{\begin{eqnarray}}
	\def\beee{\begin{eqnarray*}}
		\def\eee{\end{eqnarray}}
	\def\eeee{\end{eqnarray*}}

\def\ba{\begin{array}}
	\def\ea{\end{array}}

\def\S{\mathbb S}
\def\R{\mathbb R}

\begin{document}
	
	\title[Double mean field equation]{On the existence problem for a double mean field equation and its application to prescribed curvature problems}

	\author[]{}
	\address{}
	\email{}

	\author[Li]{Jiayu Li}%
\address{School of Mathematics, Nanjing University, 22 Hankou Road, Nanjing 210093,  People's Republic of China}%
\email{jiayuli@nju.edu.cn}%
	
	\author[Liu]{Lei Liu}
	\address{School of Mathematics and Statistics, Key Laboratory of Nonlinear Analysis and Applications (Ministry of Education), Hubei Key Laboratory of Mathematical Sciences, Central China Normal
		University, Wuhan, 430079, People's Republic of China}%
	\email{leiliu2020@ccnu.edu.cn}

	\author[Liu]{Yanjun Liu}%
	\address{School of Mathematical Sciences, Chongqing Normal University, Chongqing 401331, People's Republic of China}%
	\email{liuyj@mail.nankai.edu.cn}%

\author[Zhu]{Mingyuan Zhu}
	\address{School of Mathematics and Statistics,  Central China Normal
		University, Wuhan, 430079, People's Republic of China}
	\email{2056396052@qq.com}
	
	\thanks{}

	\subjclass[2010]{}
	
	\date{\today}
	\begin{abstract}
		Let $(M,g)$ be a compact Riemann surface with smooth boundary $\partial M$. Let $K(x)>0$, $h(x)> 0$ or $h(x)\equiv 0$ be two smooth functions on $M$ and $\partial M$ respectively. In this paper, we consider the following double mean field type equation with a nonlinear Neumann boundary condition related to a prescribed curvature problem
\begin{equation}\label{equation:01}
		\begin{cases}
			\displaystyle
			-\Delta_g u+\frac{4\pi}{|M|}
			=2\rho\frac{Ke^u}{\int_MKe^u\,dv_g},
			& \text{in }M,\\[1.2em]
			\displaystyle
			\partial_\nu u
			=2\rho'\frac{he^{u/2}}{\int_{\partial M}he^{u/2}\,ds_g},
			& \text{on }\partial M,
		\end{cases}
	\end{equation} where $\rho> 0,\ \rho'\geq 0$ and $\rho+\rho'=2\pi$, $\nu$ is the outward unit normal vector to $\partial M$. When $h\equiv 0$, the second equation changes to $\partial_\nu u=0$ on $\partial M$. By introducing a boundary integration constant \eqref{Def:integration-constant}, we obtain a sufficient condition under which \eqref{equation:01} admits a solution.

As a geometric application, we solve the prescribed curvature equation for both the Gauss curvature and the geodesic curvature in the case that $M$ is topologically a disk.  When $(M,g)$ is the standard hemisphere $\mathbb S^2_+$, let $p\in \partial \mathbb S^2_+$ be a maximum point of $$
	\Phi(q)=h(q)+\sqrt{h(q)^2+K(q)},
	\qquad q\in\partial \mathbb S^2_+.
	$$ Suppose
	\begin{equation*}
		K_\nu(p)+2\Phi(p)(-\Delta_{\partial \mathbb S^2_+})^{1/2}h(p)<0.
	\end{equation*}
Then the prescribed curvature problem
\begin{equation*}
	\begin{cases}
		-\Delta_g u+\frac{4\pi}{|M|}=2Ke^u,
		& \text{in }\mathbb S^2_+,\\
		\partial_\nu u=2he^{u/2},
		& \text{on }\partial \mathbb S^2_+.
	\end{cases}
	\end{equation*}
admits a smooth solution.

When $(M,g)$ is the standard unit disk $\mathbb D$, let $H\in C^\infty(\overline{\mathbb D})$ be the harmonic extension of $h\in C^\infty(\partial \mathbb D)$ and $p\in \partial \mathbb D$ be a maximum point of $$
	\Phi(q)=H(q)+\sqrt{H(q)^2+K(q)},
	\qquad q\in\partial \mathbb D.
	$$ Suppose
	\begin{equation*}
		\Phi_\nu(p)<0.
	\end{equation*}
Then the prescribed curvature problem
\begin{equation*}
		\begin{cases}
			-\Delta u=2Ke^u,
			&\text{in }\mathbb D,\\
			\partial_\nu u+2=2he^{u/2},
			&\text{on }\partial\mathbb D.
		\end{cases}
	\end{equation*}
admits a smooth solution.
	\end{abstract}
	\maketitle
	\section{Introduction}

	\

Let $(M,g)$ be a closed Riemann surface. The classical Kazdan-Warner problem \cite{Kazdan-Warner} asks whether a prescribed function $K$ can be realized as the Gauss curvature of a new metric $\tilde{g}$ conformal to $g$. Let $\Pi$ be the Gauss curvature of $(M,g)$. Then  the Kazdan-Warner problem reduces to solving the following equation $$-\Delta_g u+2\Pi=2Ke^u,\ \
		 \text{in }M.$$ This problem has been well understood in the case of a torus  \cite{Kazdan-Warner}.

When $(M,g)$ is the standard sphere $\mathbb{S}^2$, this problem is particularly  delicate and receives the name of Nirenberg's problem, which is equivalent to solving
\begin{equation}\label{equation:19}
-\Delta_g u+2=2Ke^u,\ \
		 \text{in }\mathbb{S}^2.\end{equation}
This question has been addressed in a large number of papers. See for instance \cite{Aubin,Chang1993,ChangYang1987,ChangYang1988,Chen-Ding1987,Chen-Ding1988,Chen-Li1995,Kazdan-Warner,Struwe2005} and others.

Analytically, it is also interesting to consider \eqref{equation:19} in a general closed Riemann surface $(M,g)$ and to ask under what kind of conditions on $K$, the equation
\begin{equation}\label{equation:20}
-\Delta_g u+\frac{8\pi}{|M|}=2Ke^u,\ \
		 \text{in } M\end{equation} has a solution. By minimizing the functional corresponding to Moser-Trudinger's inequality, Ding-Jost-Li-Wang \cite{DingJostLiWang} proved that: suppose $|M|=1$ and $p$ is a maximum point of $A(x)+2\log K(x)$ in $M$ where $A$ is the constant term in the expansion of Green's function, if $$\frac{\Delta_gK(p)}{K(p)}-\frac{|\nabla_g K|^2(p)}{K^2(p)}>-(8\pi-2\Pi(p)),$$  then \eqref{equation:20} admits a solution.
	
	Now let $(M,g)$ be a compact connected Riemannian surface with smooth
	nonempty boundary. We denote by $\Delta_g$ the Laplace--Beltrami
	operator, by $\nu$ the outward unit normal along $\partial M$, and by
	$dv_g$ and $ds_g$ the area and boundary length elements, respectively.
	The simultaneous prescription of Gaussian curvature in $M$ and
	geodesic curvature on $\partial M$ leads to the boundary value problem
	\begin{equation}\label{equation:02}
	\begin{cases}
		-\Delta_g u+2\Pi=2Ke^u,
		& \text{in }M,\\
		\partial_\nu u+2\kappa=2he^{u/2},
		& \text{on }\partial M,
	\end{cases}
	\end{equation}
	where $\Pi$ and $\kappa$ are the Gaussian and geodesic curvatures of
	the background metric, and $K$ and $h$ are the corresponding
	curvatures of the conformal metric $e^ug$, $\nu$ is the outward normal vector field to $\partial M$. Here the sign convention
	for $\kappa$ is $\nabla_\tau\tau=-\kappa\nu$, with $\tau$ a unit
	tangent vector to $\partial M$. Integrating the equations and using
	the Gauss--Bonnet theorem gives
	$$
	\int_MKe^u\,dv_g+\int_{\partial M}he^{u/2}\,ds_g=\int_M \Pi\,dv_g+\int_{\partial M}\kappa\,ds_g=2\pi\chi(M).
	$$

The solvability of \eqref{equation:02} has been investigated in
some special settings. When $h\equiv0$, the problem consists
of prescribing the Gaussian curvature while requiring the
boundary to be geodesic. See \cite{ChangYang1988,GuoHu1995,LiuLi1993,LopezSorianoRuiz2016,Wang2001}.
The other case $K\equiv0$ concerns flat conformal
metrics with prescribed boundary geodesic curvature.
See \cite{ChangLiu1996,LiuHuang2005,Zhang2015}.
The case of constant curvatures (i.e. $K,h$ are two constants) has been considered in \cite{Brendle}, where solutions are obtained by using a parabolic flow.

Up to our knowledge, very few existence results are available for \eqref{equation:02} when both curvatures are non-constant. One of them is Cruz-Ruiz \cite{Cruz-Ruiz}, where, by introducing the parameter $\rho$, they showed that the problem \eqref{equation:15} is equivalent to the following mean field type equation
\begin{align}\label{equation:16}
		\begin{cases}
			\displaystyle
			-\Delta_g u+\frac{4\pi\chi(M)}{|M|}
			=2\rho\frac{Ke^u}{\int_MKe^u\,dv_g},
			& \text{in }M,\\[1.2em]
			\displaystyle
			\partial_\nu u
			=2(2\pi\chi(M)-\rho)\frac{he^{u/2}}{\int_{\partial M}he^{u/2}\,ds_g},
			& \text{on }\partial M,\\[1.2em]
			\displaystyle
\frac{(2\pi\chi(M)-\rho)^2}{|\rho|}=\frac{\left(\int_{\partial M}he^{u/2}\,ds_g\right)^2}{\left|\int_MKe^u\,dv_g\right|}, & \text{for }0<\rho<2\pi.
		\end{cases}
	\end{align}
In \cite{Cruz-Ruiz}, they got some existence results for the case that $M$ is the unit disk and the functions $K$ and $h$ are nonnegative with certain symmetric properties. When $\chi(M)\leq 0$, the case of surfaces topologically different from the disk,
by  using a variational approach (either by minimization of the Euler-Lagrange energy or via min-max methods) and blow-up analysis, they \cite{Lopez-Soriano-Malchiodi-Ruiz} derived some existence results. In \cite{Lopez-Soriano-Malchiodi-Ruiz}, a new and interesting blow-up phenomenon appears, where the length and area diverge.

Later, Ruiz \cite{Ruiz2024} got some existence results without symmetry assumptions for the unit disk $M$.  Precisely,
on the unit disk $\mathbb D$, let $H$ be the harmonic extension
of $h$. Denote
$$
\Phi=H+\sqrt{H^2+K}.
$$
For smooth data satisfying either $K\geq0$, $h>0$, or
$K>0$, $h\geq0$, by using blow-up analysis and Leray--Schauder degree, the author \cite{Ruiz2024} proved existence under the conditions that
$$
\nabla\Phi\neq0\quad\text{on }\partial\mathbb D,
\qquad
\deg_B(\nabla\Phi,\mathbb D,0)\neq0.
$$

Recently, Struwe \cite{Struwe2024} developed a prescribed
curvature flow for positive $K$ and $h$, based on the joint
variational formulation of Cruz-Ruiz \cite{Cruz-Ruiz}.
By analyzing boundary concentration and the associated
finite-dimensional shadow flow, he obtained further
existence criteria without symmetry assumptions.

By Proposition 3.1 in \cite{Lopez-Soriano-Malchiodi-Ruiz}, using a conformal change of metric, we can always assume $\kappa=0$ and $\Pi=\frac{2\pi\chi(M)}{|M|}$. Then \eqref{equation:02} is equivalent to
\begin{equation}\label{equation:15}
	\begin{cases}
		-\Delta_g u+\frac{4\pi\chi(M)}{|M|}=2Ke^u,
		& \text{in }M,\\
		\partial_\nu u=2he^{u/2},
		& \text{on }\partial M.
	\end{cases}
	\end{equation}

Based on \eqref{equation:15} and \eqref{equation:16},  Battaglia and  L\'opez-Soriano \cite{BattagliaLopezSoriano}  considered the following mean field type equation on a general compact Riemann surface with a smooth boundary
	\begin{equation}\label{equation:mean-field}
		\begin{cases}
			\displaystyle
			-\Delta_g u+\frac{2(\rho+\rho')}{|M|}
			=2\rho\frac{Ke^u}{\int_MKe^u\,dv_g},
			& \text{in }M,\\[1.2em]
			\displaystyle
			\partial_\nu u
			=2\rho'\frac{he^{u/2}}{\int_{\partial M}he^{u/2}\,ds_g},
			& \text{on }\partial M.
		\end{cases}
	\end{equation}
On one hand, it is easy to see that \eqref{equation:mean-field} can be derived from \eqref{equation:15} by setting $\rho=\int_MKe^u\,dv_g$ and $\rho'=\int_{\partial M}he^{u/2}\,ds_g$, which includes also the nongeometrical case $\rho+\rho'\neq 2\pi\chi(M)$. On the other hand, different from \eqref{equation:16}, not all the solutions to \eqref{equation:mean-field} correspond to \eqref{equation:15} since there is an additional third condition in \eqref{equation:16}. By using variational method and Moser-Trudinger inequality, they proved that \eqref{equation:mean-field} has solutions if $\rho<4\pi$ and $\rho+\rho'<2\pi$. They further constructed min-max solutions over a broad range of parameters, under suitable topological assumptions and away from the critical values of $\rho$ and $\rho+\rho'$. Recently, Battaglia and
L\'opez-Soriano \cite{BattagliaLopezSoriano2025} developed
another mean-field formulation of the curvature prescription
problem. On the hemisphere, they obtained existence for
$K\geq0$, $K\not\equiv0$, under a common nontrivial finite
rotational symmetry of $K$ and $h$, without imposing a sign
condition on $h$.

\

 In the present paper, we study \eqref{equation:mean-field} in the critical and geometric case, namely $$\rho+\rho'= 2\pi,\ \ 0<\rho\leq 2\pi,\ 0\leq\rho'<2\pi.$$ Precisely, we focus on the following two issues:
 \begin{itemize}
 \item[(1)] \textbf{Analytical part:} Under what conditions on $K$ and $h$, does \eqref{equation:mean-field} admit a solution?

     \

 \item[(2)] \textbf{Geometric part:} What geometric applications does \eqref{equation:mean-field} have? In particular, can we solve \eqref{equation:02} on a surface $(M,g)$ with $\chi(M)>0$, i.e., when $M$ is topologically a disk?
 \end{itemize}

 \

	Before stating our main result, we give some notations.	For $p\in\partial M$, let $G(x,p)$ be the normalized Neumann Green's
	function satisfying
	\begin{equation}\label{eq:intro-green-function}
		\begin{cases}
			\displaystyle
			\Delta_g G(x,p)=\frac{4\pi}{|M|},
			& \text{in } M,\\[0.8em]
			\displaystyle
			\frac{\partial G(x,p)}{\partial\nu}=4\pi\delta_p,
			& \text{on } \partial M,\\[0.8em]
			\displaystyle
			\int_M G(x,p)\,dv_g=0.
		\end{cases}
	\end{equation}
	In geodesic normal coordinates centered at $p$, with the positive
	$x_2$-direction pointing into $M$, we write
	\begin{equation}\label{eq:intro-green-expansion}
		\begin{aligned}
			G(x,p)
			={}&-4\log r+A(p)+b_1(p)x_1+b_2(p)x_2\\
			&+c_1(p)x_1^2+2c_2(p)x_1x_2+c_3(p)x_2^2+O(r^3),
			\qquad r=d_g(x,p).
		\end{aligned}
	\end{equation}
	In particular, $A(p)$ is continuous for $p\in\partial M$. See Lemma 2.1 in \cite{Hu-Bartsch}.

	Define the following boundary integration constant corresponding to $h$ and Green function $G$ by
	\begin{equation}\label{Def:integration-constant}
		\Lambda(G,h):=\lim_{\delta\downarrow0}
		\left(
		\int_{\partial M\setminus B_\delta(p)}
		\frac{h(x)}{h(p)}e^{G(x,p)/2}\,ds_g
		-\frac{2e^{A(p)/2}}{\delta}
		\right),
	\end{equation}
	where $B_\delta(p)=\{x\in \overline{M}:d_g(x,p)<\delta\}$.	We want to remark  that the limit $\Lambda(G,h)$ exists and is finite. See Lemma \ref{lem:boundary-green-finite-part}.

\

Now we state our first main theorem as follows.
	
	\begin{thm}\label{thm:main-01}
		Suppose $K\in C^\infty(M)$ and
		$h\in C^\infty(\partial M)$ are two positive functions and
\begin{equation}\label{eq:intro-critical-parameters}
		\rho+\rho'=2\pi,\qquad \rho>0,\qquad \rho'>0.
	\end{equation}
Suppose that $p\in\partial M$
		is a maximum point of
		$$
		A+\frac{\rho}{\pi}\log K+\frac{\rho'}{\pi}\log h^2
		\qquad\text{on }\partial M.
		$$
		If
		\begin{equation}\label{eq:intro-existence-condition}
			\begin{aligned}
				&-\frac{K_\nu(p)}{K(p)}-2\kappa(p)+
				\frac{2\rho'}{\pi\rho e^{A(p)/2}}\Lambda(G,h)>0,				
			\end{aligned}
		\end{equation} where $\Lambda(G,h)$ is the boundary integration constant defined in \eqref{Def:integration-constant},
		then $J_{\rho,\rho'}$ attains its minimizer in $W^{1,2}(M)$.
		Consequently, \eqref{equation:mean-field} admits a smooth solution.
	\end{thm}

\

When $h\equiv 0$ and $\rho'=0$, we have
\begin{thm}\label{thm:main-02}
		 Let $K\in C^\infty(M)$ be a positive function,	$h\equiv 0$, $
		\rho=2\pi$, $ \rho'=0.	$
Suppose that $p\in\partial M$
		is a maximum point of
		$$
		A+2\log K
		\qquad\text{on }\partial M.
		$$
		If
		\begin{equation}
			\begin{aligned}
				&\frac{K_\nu(p)}{K(p)}+2\kappa(p)<0,				
			\end{aligned}
		\end{equation}
		then $J_{2\pi,0}$ attains its minimizer in $W^{1,2}(M)$.
		Consequently, the following equation
\begin{equation}\label{equation:17}
		\begin{cases}
			\displaystyle
			-\Delta_g u+\frac{4\pi}{|M|}
			=2 Ke^u,
			& \text{in }M,\\
			\partial_\nu u
			=0,
			& \text{on }\partial M.
		\end{cases}
	\end{equation}
 admits a smooth solution.
	\end{thm}

\
	
	On a spherical cap, the boundary integration constant $\Lambda(G,h)$ has an explicit expression.
	Let
	$$
	M=\mathbb S^2\cap\{x_3\geq\cos\theta_0\},
	\qquad 0<\theta_0<\pi,
	$$
	with the standard spherical metric. Then $A$ is constant on
	$\partial M$, $\kappa=\cot\theta_0$, and
	$$
	\lim_{\delta\downarrow0}
	\left[
	\int_{\partial M\setminus B_\delta(p)}
	h(x)e^{G(x,p)/2}\,ds_g
	-\frac{2e^{A(p)/2}h(p)}{\delta}
	\right]
	=-\pi e^{A(p)/2}(-\Delta_{\partial M})^{1/2}h(p).
	$$
	Here $(-\Delta_{\partial M})^{1/2}$ denotes the positive square root
	of $-\Delta_{\partial M}$. See Section \ref{sec:hemisphere-application}.

\

By Theorem \ref{thm:main-01}, we have the following corollary.

	\begin{thm}\label{thm:main-03}
		Let $M$ be the spherical cap above and $K$ be a  smooth
		positive function. Then we have:
\begin{itemize}

\item[(1)] Assume $h$ is a smooth  positive function and \eqref{eq:intro-critical-parameters} holds.
		Suppose that $p\in\partial M$ is a maximum point of
		$$\rho\log K+2\rho'\log h \ \  on \ \ \partial M.$$ If
		\begin{equation}\label{equation:22}
		\frac{K_\nu(p)}{K(p)}+2\cot\theta_0
		+\frac{2\rho'}{\rho}
		\frac{(-\Delta_{\partial M})^{1/2}h(p)}{h(p)}<0,
		\end{equation}
		then $J_{\rho,\rho'}$ attains its minimizer and
		\eqref{equation:mean-field} has a smooth solution.

\

\item[(2)] Assume $h\equiv 0$, $\rho=2\pi,\  \rho'=0$. Suppose that $p\in\partial M$ is a maximum point of
		$K$ on $\partial M$. If
		$$
		\frac{K_\nu(p)}{K(p)}+2\cot\theta_0
		<0,
		$$
		then $J_{2\pi,0}$ attains its minimizer and
		\eqref{equation:17} has a smooth solution.

\

\item[(3)] In particular, the above conclusions hold for  constant $K>0$ and $h\geq 0$ whenever $\frac{\pi}{2}<\theta_0<\pi$.
\end{itemize}
	\end{thm}

\

In the above theorems, we focus on the existence problem for the  PDE \eqref{equation:mean-field} on a general compact Riemann surface. Now we come back to the prescribed curvature problem \eqref{equation:02} or \eqref{equation:15}. In the model \eqref{equation:15}, we set $$\chi(M)=1,\ \Pi=1,\ \kappa=0,$$ i.e. $(M,g)$ is the standard hemisphere $\mathbb S^2_+:=\mathbb S^2\cap\{x_3\geq 0\}$.

\

\begin{thm}\label{thm:main-05}
	Let $K\in C^\infty(\mathbb S^2_+)$
	and $h\in C^\infty(\partial \mathbb S^2_+)$ where $K>0$, $h> 0$ or $h\equiv 0$. Denote
	$$
	\Phi(q)=h(q)+\sqrt{h(q)^2+K(q)},
	\qquad q\in\partial \mathbb S^2_+.
	$$
	Let $p\in\partial \mathbb S^2_+$ be the maximal point of $\Phi(x)$ on $\partial \mathbb S^2_+$, i.e.
	$
	\Phi(p)=\max_{\partial \mathbb S^2_+}\Phi.
	$
	Suppose
	\begin{equation}\label{eq:joint-boundary-condition}
		K_\nu(p)+2\Phi(p)(-\Delta_{\partial \mathbb S^2_+})^{1/2}h(p)<0.
	\end{equation}
Then the prescribed curvature problem
\eqref{equation:15}
admits a smooth solution.
\end{thm}

\

As the first corollary of  Theorem \ref{thm:main-05}, we have
\begin{cor}\label{cor:01}
Let $K$ be a smooth positive function on $\mathbb S^2_+$, $h\geq 0$ be a constant. Suppose that $p\in\partial \mathbb S^2_+$ is a maximum point of		$K$ on $\partial \mathbb S^2_+$, i.e. $K(p)=\max_{\partial \mathbb S^2_+}K(x)$. If	$$	\partial_{\nu}K(p)<0,$$ then the prescribed curvature problem
\eqref{equation:15}
admits a smooth solution.
\end{cor}

\

As the second corollary of  Theorem \ref{thm:main-05}, we have
\begin{thm}\label{thm:main-06}
	Let $\mathbb D=\{x\in\R^2:|x|<1\}$ be the standard unit disk.
	Let $K\in C^\infty(\overline{\mathbb D})$ and
	$h\in C^\infty(\partial\mathbb D)$, where $K>0$, $h> 0$ or $h\equiv 0$.
	Denote by $H$ the harmonic extension of $h$, i.e.
	$$
	\begin{cases}
		\Delta H=0,
		&\text{in }\mathbb D,\\
		H=h,
		&\text{on }\partial\mathbb D.
	\end{cases}
	$$
	Set
	$$
	\Phi(x)=H(x)+\sqrt{H(x)^2+K(x)},
	\qquad x\in\overline{\mathbb D}.
	$$
	Let $p\in\partial\mathbb D$ be a maximal point of $\Phi(x)$, i.e. $\Phi(p)=\max_{\partial\mathbb D}\Phi$. Suppose
	$$
	\partial_\nu\Phi(p)<0,
	$$
	where $\nu$ is the outward unit normal vector field.
	Then there exists $u\in C^\infty(\overline{\mathbb D})$
	satisfying
	\begin{equation}\label{eq:disk-prescribed-curvature}
		\begin{cases}
			-\Delta u=2Ke^u,
			&\text{in }\mathbb D,\\
			\partial_\nu u+2=2he^{u/2},
			&\text{on }\partial\mathbb D.
		\end{cases}
	\end{equation}
\end{thm}

\

To prove Theorem \ref{thm:main-01}, inspired by the work of Ding-Jost-Li-Wang \cite{DingJostLiWang}, we shall try to minimize the following functional
		\begin{align*}
			J_{\rho,\rho'}(u)
			=&\frac12\int_M|\nabla_g u|^2\,dv_g
			+\frac{2(\rho+\rho')}{|M|}\int_Mu\,dv_g-2\rho\log\int_MKe^u\,dv_g
			-4\rho'\log\int_{\partial M}he^{u/2}\,ds_g.
		\end{align*}
Let $\rho_\epsilon=\rho-\epsilon$ and $u_\epsilon$ be a
	minimizer of $J_{\rho_\epsilon,\rho'}$. If $u_\epsilon$ converges strongly to $u$ in $W^{1,2}(M)$, then it is easy to see that $u$ is a minimizer of $J_{\rho,\rho'}$ which is the desired solution. Next we just need to consider the case that $\{u_\epsilon\}$ blows up as $\epsilon\to 0$.

 On one hand, by using blow-up analysis, we  show that if $\{u_\epsilon\}$ blows up, then the functional has a lower bound that
	\begin{equation}\label{eq:intro-blow-up-level}
		\begin{aligned}			
			\inf_{u\in W^{1,2}(M)}J_{\rho,\rho'}(u)
			\geq{}&-4\pi+2\rho'+4\pi\log4
			-2(2\pi+\rho')\log(2\pi+\rho')\\
			&-2\pi\max_{q\in\partial M}
			\left[
			A(q)+\frac{\rho}{\pi}\log K(q)
			+\frac{\rho'}{\pi}\log h^2(q)
			\right].
		\end{aligned}
	\end{equation}
Here, in addition to the usual blow-up properties such as the classification of bubbles and the convergence to Green's function away from blow-up points, we need to show that for sufficiently small $\epsilon$, the blow-up point must lie on the boundary $\partial M$. Moreover, we shall compute the value of the quantity $$\lim_{\epsilon\to 0}\frac{\sqrt{\int_M Ke^{u_\epsilon}dv_g}}{\int_{\partial M}he^{\frac{u_\epsilon}{2}}ds_g}$$ which determines the center of the bubble and plays an important role in our later construction. See Section \ref{sec:blow-up-analysis}.

On the other hand, we construct a blow-up sequence $\{\phi_{\epsilon,\delta}\}$ such that $J_{\rho,\rho'}(\phi_{\epsilon,\delta})$ is strictly less than the right hand side of \eqref{eq:intro-blow-up-level} if $0<\delta<\delta_0$ and $\epsilon$ is sufficiently small, where $\delta_0$ just depends on $M,\ \|K\|_{C^1(M)},\ \|h\|_{C^1(\partial M)}$. Then we get a contradiction and Theorem \ref{thm:main-01} follows. Unlike the interior case \cite{DingJostLiWang}, where the center of the bubble coincides with the pole of Green's function, in the boundary case these two points do not coincide in general. To address this, we introduce an additional parameter $\delta$ in the construction of the blow-up sequence.

To prove Theorem \ref{thm:main-05}, or equivalently, to solve \eqref{equation:16}, we minimize the functional \begin{equation}
	\begin{aligned}
		\widehat J(u,\rho)
		={}&J_{\rho,\,2\pi-\rho}(u)+2\rho\log\rho
		+4(2\pi-\rho)\log(2\pi-\rho)+2\rho,
	\end{aligned}
\end{equation}
over both the function and the parameter $$(u,\rho)\in W^{1,2}(\mathbb S^2_+)\times[0,2\pi]\ \ with\ \ \int_{\partial \mathbb S^2_+}ud\theta=0.$$ Denote the perturbed functional $$\widehat J_t(u,\rho)=\widehat J(u,\rho)+\frac{t}{2}\int_M|\nabla_g u|^2dv_g,\ \ t>0.$$  By Moser-Trudinger's inequality, one can see that $\widehat J_t(u,\rho)$ is coercive for any fixed $t>0$, which immediately yields that it has a minimizer $(u_t,\rho_t)$. Set $$
	w_t=u_t+
	\log\frac{\rho_t}
	{(1+t)\displaystyle\int_MKe^{u_t}\,dv_g}.
	$$

If $w_t$ converges strongly to $w_*$ in $W^{1,2}(M)$, then it follows from the E-L equation of $(u_t,\rho_t)$ that $w_*$ is  a desired  solution of \eqref{equation:15}. Theorem \ref{thm:main-05} then follows. If $w_t$ blows up,  we will show that there exists a global minimizer $(u_*,\rho_*)\in W^{1,2}(M)\times (0,2\pi)$ such that $$
	m:=\inf_{\substack{u\in W^{1,2}(M)\\0\leq\rho\leq2\pi}}
	\widehat J(u,\rho)=\widehat J(u_*,\rho_*).
	$$ Then $(u_*,\rho_*)$ solves  \eqref{equation:16}, and Theorem \ref{thm:main-05}  follows.

To establish the existence of $(u_*,\rho_*)$, we proceed in the following two steps. Suppose $w_t$ blows up at $q$ (must be a boundary point, i.e. $q\in\partial M$). By using blow-up analysis, we firstly prove that  $$
	m:=\inf_{u\in W^{1,2}(M)}
	\widehat J(u,\rho_*),
	$$ where $$\rho_*=2\pi\left(
	1-\frac{h(q)}{\sqrt{K(q)+h(q)^2}}
	\right)\in(0,2\pi).$$
Secondly, we minimize $\widehat J(u,\rho_*)$, which is similar to the proof of Theorem \ref{thm:main-01}. Let $u_\epsilon$ be a minimizing sequence. We prove that it must converge strongly to $u_*$ in $W^{1,2}(M)$. Otherwise, on one hand, by Proposition \ref{prop:lower-bound} and the definition of $\widehat J(u,\rho)$, noting that $A=4\log 2-2$ for $M=\mathbb S^2_+$, we have
\begin{equation*}
	m=\inf_{u\in W^{1,2}(M)}\widehat J(u,\rho_*)
	\geq\min_{z\in\partial M}L(\rho_*,z),
\end{equation*}
where
$$
\begin{aligned}
	L(\rho,z)
	={}&4\pi-2(4\pi-\rho)\log(4\pi-\rho)
	+2\rho\log\rho\\
	&+4(2\pi-\rho)\log(2\pi-\rho)
	-2\rho\log K(z)-4(2\pi-\rho)\log h(z),
\end{aligned}
$$
for $0<\rho<2\pi$ and $z\in\partial M$. The function $L(\cdot,z)$ has a unique minimum at
$$
\rho_z=2\pi\left(
1-\frac{h(z)}{\sqrt{K(z)+h(z)^2}}
\right).
$$
Therefore, by the choice of $p$ in Theorem \ref{thm:main-05}, we have
\begin{equation}
	m\geq\min_{z\in\partial M}L(\rho_*,z)\geq \min_{z\in\partial M}L(\rho_z,z)=\min_{z\in\partial M}\left(4\pi-8\pi\log\Phi(z)\right)
	\geq4\pi-8\pi\log\Phi(p).
\end{equation}

On the other hand, taking
	$$
	\rho_p=2\pi\left(
	1-\frac{h(p)}{\sqrt{K(p)+h(p)^2}}
	\right)\in(0,2\pi)
	$$
	and the test function $\phi_{\epsilon,\delta}$ centered at
	$p$ in \eqref{eq:test-function}, by Proposition \ref{prop:01}, we obtain
	$$
	\begin{aligned}
		\widehat J(\phi_{\epsilon,\delta},\rho_p)
		={}&4\pi-8\pi\log\Phi(p)-\frac{4\pi}{\Phi(p)\sqrt{K(p)}}
		\left[
		K_2(p)-2\Phi(p)(-\Delta_{\partial M})^{1/2}h(p)
		\right]\sqrt{\epsilon}\\
		&+O(\delta\sqrt{\epsilon})
		+o(\sqrt{\epsilon}).
	\end{aligned}
	$$
	By \eqref{eq:joint-boundary-condition}, first choosing
	$\delta>0$ sufficiently small and then $\epsilon>0$
	sufficiently small, we obtain
	\begin{equation}
		m\leq\widehat J(\phi_{\epsilon,\delta},\rho_p)
		<4\pi-8\pi\log\Phi(p).
	\end{equation}
This is a contradiction.

\

\noindent{\textbf{Notations:}} Let $B_r(p)\subset M$ be a geodesic ball centered at $p$ with radius $r$. Denote $\partial ^+B_r(p)=\partial B_r(p)\cap M$. Denote by $\nabla_g, \Delta_g$ and $D, \Delta$ the derivative and Laplace operator with respect to Riemannian metric $g$ and Euclidean metric.

Denote $f=o_\epsilon(1)$ if $\lim_{\epsilon\to 0}f=0$, $f=o_R(1)$ if $\lim_{R\to +\infty}f=0$, $f=o(\epsilon)$ if $\lim_{\epsilon\to 0}\frac{f}{\epsilon}=0$, $f=O(\epsilon)$ if $|f|\leq C\epsilon$.

\
	
	The paper is organized as follows. In Section \ref{sec:blow-up-analysis}, we study the blow-up analysis for the minimizing sequence $\{u_\epsilon\}$ and derive a lower bound of $J_{\rho,\rho'}(u_\epsilon)$. In Section \ref{sec:construc-blowup-seq}, we construct a blow-up sequence $\{\phi_{\epsilon,\delta}\}$ such that  $J_{\rho,\rho'}(\phi_{\epsilon,\delta})$ is strictly less than the lower bound obtained in Section \ref{sec:blow-up-analysis}. Theorem \ref{thm:main-01} and Theorem \ref{thm:main-02} are proved in this section. In Section \ref{sec:hemisphere-application}, we consider some geometric applications for the case that $M$ is a spherical cap or a unit disk. The proofs of Theorems \ref{thm:main-03}--\ref{thm:main-06} are given in this section.

	\

	\section{Blow-up analysis and a lower bound for the functional}\label{sec:blow-up-analysis}

\

In this section, we will study the blow-up analysis for the minimizing sequence $u_\epsilon$ of the functional $J_{\rho_\epsilon,\ \rho'}$ where $\rho_\epsilon:=\rho-\epsilon$. Specially, we will derive a lower bound of  $\inf_{\epsilon}J_{\rho_\epsilon,\ \rho'}$ if $u_\epsilon$ blows up.

\

	We consider the functional
	\begin{equation}\label{equation:functional}
		J_{\rho,\rho'}(u):=\frac{1}{2}\int_M|\nabla u|^2dv_g+\frac{2(\rho+\rho')}{|M|}\int_Mudv_g-2\rho\log\int_MKe^udv_g-4\rho'\log\int_{\partial M}he^{\frac{u}{2}}ds_g,
	\end{equation} where  $$u\in W^{1,2}(M),\ \ \int_{\partial M}uds_g=0,\ \ \rho+\rho'=2\pi,\ \ 0<\rho\leq 2\pi.$$
	
	For any $\epsilon\in (0,\frac{1}{2}\rho)$, denoting $\rho_\epsilon:=\rho-\epsilon$,  by the Moser-Trudinger inequality \cite[Theorem~3.1]{BattagliaLopezSoriano}, we know the functional $J_{\rho_\epsilon,\ \rho'}$ is coercive which implies that there exists a minimizer $u_\epsilon\in W^{1,2}(M)$ such that $$J_{\rho_\epsilon,\ \rho'}(u_\epsilon)=\inf_{u\in W^{1,2}(M)}J_{\rho_\epsilon,\ \rho'}(u).$$	

	If $u_\epsilon\to u$ strongly in $W^{1,2}(M)$ as $\epsilon\searrow 0$, then it is easy to see that $u$ is a minimizer of $J_{\rho,\ \rho'}$. This gives the desired solution to \eqref{equation:mean-field}, and Theorem \ref{thm:main-01} follows.

Thus, we may assume, after passing to a subsequence, that $u_\epsilon$ does not converge strongly to $u$ in $W^{1,2}(M)$ as $\epsilon\searrow 0$.

In this part, we will show the following proposition.
\begin{prop}\label{prop:lower-bound}
If $\{u_\epsilon\}$ blows up, then we have the following energy bound \begin{align*}
\liminf_{\epsilon\to 0}J_{\rho_\epsilon,
			\rho'}(u_\epsilon)\geq& -4\pi+2\rho'+ 4\pi\log 4-2(2\pi+\rho')\log(2\pi+\rho')\notag\\&-2\pi \max_{p\in \partial M}\left(A(p) +\frac{\rho}{\pi}\log K(p)+\frac{\rho'}{\pi}\log h^2(p)\right).\end{align*}
\end{prop}
	
\

We first assume $\rho'>0$. The case $\rho'=0$ is similar and  treated at the end of this section.

\

	Since $u_\epsilon$ is a minimizer, it satisfies
	\begin{align}\label{equation:mean-field-01}
		\begin{cases}
			-\Delta u_\epsilon+\frac{2(\rho_\epsilon+\rho')}{|M|}=2\rho_\epsilon\frac{Ke^{u_\epsilon}}{\int_MKe^{u_\epsilon}} \ \ &in\ \ M,\\
			\partial_\nu u_\epsilon=2\rho'\frac{he^{\frac{u_\epsilon}{2}}}{\int_{\partial M}he^{\frac{u_\epsilon}{2}}}\ \ &on\ \ \partial M.
	\end{cases}\end{align} By translation invariance, we may assume
	\begin{equation}\label{equation:03}
		\int_MKe^{u_\epsilon}dv_g+(\int_{\partial M}he^{\frac{u_\epsilon}{2}}ds_g)^2=1.
	\end{equation}

	Set $$\mu_\epsilon:=\sqrt{\int_MKe^{u_\epsilon}},\ \ \kappa_\epsilon:=\int_{\partial M}he^{\frac{u_\epsilon}{2}}.$$ Then it is easy to see that  $$\mu_\epsilon,\ \kappa_\epsilon\in (0,1),\ \ \mu_\epsilon^2+\kappa_\epsilon^2=1.$$ By Moser-Trudinger inequality, we have the following lemma.
\begin{lem}\label{lem:normalization-ratio-bounds}
		There exists a constant $C>1$, independent of $\epsilon$, such that
		$$
		C^{-1}\leq\frac{\mu_\epsilon}{\kappa_\epsilon}\leq C
		$$
		for all sufficiently small $\epsilon>0$.
	\end{lem}
	
	\begin{proof}
		By the Moser--Trudinger inequalities on compact surfaces with boundary \cite{ChangYang1988,LiLiu2005,Yang2006} and the boundedness of $K,h$, we have
		$$
		J_{2\pi,0}(u)\geq-C,\ \ J_{0,2\pi}(u)\geq-C,
		\qquad u\in W^{1,2}(M).
		$$
		
		Since $\rho_\epsilon+\rho'=2\pi-\epsilon$,
		a direct computation gives
		$$
		\begin{aligned}
			J_{\rho_\epsilon,\rho'}(u_\epsilon)
			={}&
			\left(1-\frac{\epsilon}{2\pi}\right)
			J_{2\pi,0}(u_\epsilon)
			+\frac{\epsilon}{4\pi}
			\int_M|\nabla u_\epsilon|^2\,dv_g+4\rho'
			\log\frac{\mu_\epsilon}{\kappa_\epsilon}\\
			={}&
			\left(1-\frac{\epsilon}{2\pi}\right)
			J_{0,2\pi}(u_\epsilon)
			+\frac{\epsilon}{4\pi}
			\int_M|\nabla u_\epsilon|^2\,dv_g+4\rho_\epsilon
			\log\frac{\kappa_\epsilon}{\mu_\epsilon}.
		\end{aligned}
		$$
		Noting $u_\epsilon$ is a minimizer, then
		$$
		J_{\rho_\epsilon,\rho'}(u_\epsilon)
		\leq
		J_{\rho_\epsilon,\rho'}(0)
		\leq C,
		$$ which implies
		$$
		4\rho'
		\log\frac{\mu_\epsilon}{\kappa_\epsilon}\leq C,
		\qquad
		4\rho_\epsilon
		\log\frac{\kappa_\epsilon}{\mu_\epsilon}\leq C
		$$ when $\epsilon$ is small.
		Since $\rho,\ \rho'\in (0,2\pi)$, then the conclusion of the lemma follows.
		
	\end{proof}

Denoting $$v_\epsilon:=u_\epsilon-2\log \mu_\epsilon,$$ then $v_\epsilon$ satisfies
	\begin{align}\label{equation:liouville-01}
		\begin{cases}
			-\Delta v_\epsilon+\frac{2(\rho_\epsilon+\rho')}{|M|}=2\rho_\epsilon K e^{v_\epsilon} \ \ &in\ \ M,\\
			\partial_\nu v_\epsilon=2\rho'\widetilde{h}_\epsilon e^{\frac{v_\epsilon}{2}}\ \ &on\ \ \partial M,
	\end{cases}\end{align} where $\widetilde{h}_\epsilon=h\frac{\mu_\epsilon}{\kappa_\epsilon}$.
	Moreover, there holds
	$$ 0\leq  \widetilde{h}_\epsilon\leq C,\ \ \int_M K e^{v_\epsilon}dv_g=1,\ \ \int_{\partial M}\widetilde{h}_\epsilon e^{\frac{v_\epsilon}{2}}ds_g=1$$  and	
	\begin{align}\label{equation:06}
		J_{\rho_\epsilon,\ \rho'}(u_\epsilon)&=J_{\rho_\epsilon,\ \rho'}(v_\epsilon)\notag\\
		&=\frac{1}{2}\int_M|\nabla v_\epsilon|^2dv_g+\frac{2(\rho_\epsilon+\rho')}{|M|}\int_Mv_\epsilon dv_g-2\rho_\epsilon\log\int_MK e^{v_\epsilon}dv_g-4\rho'\log\int_{\partial M}h e^{\frac{v_\epsilon}{2}}ds_g\notag\\
		&	=\frac{1}{2}\int_M|\nabla v_\epsilon|^2dv_g+\frac{2(\rho_\epsilon+\rho')}{|M|}\int_Mv_\epsilon dv_g-4\rho'\log\frac{\kappa_\epsilon}{\mu_\epsilon}.
	\end{align}

	\
	\begin{lem}
		Setting $\lambda_\epsilon=\max_{x\in M}v_\epsilon(x)$, then $\liminf_{\epsilon\searrow 0}\lambda_\epsilon=+\infty$.
	\end{lem}
	\begin{proof}
		If not, there exists a subsequence of $v_{\epsilon}$ (also denoted by itself), such that $\lambda_\epsilon\leq C$. By standard elliptic theory, we get $\|v_\epsilon\|_{L^\infty(M)}\leq C.$

On one hand, by \eqref{equation:liouville-01} and elliptic estimates, there holds $\|v_\epsilon\|_{W^{2,p}(M)}\leq C(p),$ for any $p>1$.	On the other hand, since $2\log \mu_\epsilon$ is a constant, there must be $$-C\leq 2\log \mu_\epsilon \leq 0,$$ which yields $$\|u_\epsilon\|_{L^\infty(M)}\leq C.$$ Otherwise, we have $$u_\epsilon=v_\epsilon+2\log  \mu_\epsilon \to -\infty$$ uniformly in $M$, which is a contradiction to \eqref{equation:03}.
		
		Thus, we obtain $$\|u_\epsilon\|_{W^{2,p}(M)}\leq C(p),$$ for any $p>1$. Then passing to a subsequence, $u_\epsilon\to u$ strongly in $W^{1,2}(M)$, which is also a contradiction.
	\end{proof}
	
	\

	Assume $v_\epsilon(x_\epsilon)=\max_{x\in M}v_\epsilon (x)=\lambda_\epsilon$ and $x_\epsilon\to p$ as $\epsilon\searrow 0$. Since $2\rho_\epsilon\int_M K e^{v_\epsilon}\leq 4\pi$, by the classification theory in \cite{Chen-Li}, we know $p\in\partial M$. We choose a conformal coordinate system in a  neighborhood $B_R(p)$ around $p$, i.e. there exists a conformal map  $F:D^+(0)\to B_R(p)$ with the properties $F(0)=p$ and $$F^*g=e^{\varphi}((dx^1)^2+(dx^2)^2),$$ where $\varphi$ is a smooth function with $\varphi(0)=0$.
	
	For simplicity of notations, we still denote $v_\epsilon\circ F$ by $v_\epsilon$. Setting $$w_\epsilon(x):=v_\epsilon (x_\epsilon+\delta_\epsilon x)+2\log\delta_\epsilon,\ \ where\ \ \delta_\epsilon=e^{-\frac{\lambda_\epsilon}{2}}.$$ It is easy to see that $\lim_{\epsilon\searrow 0}\delta_\epsilon=0$, $w_\epsilon (x)\leq w_\epsilon(0)=0$ and
	\begin{align*}
		\begin{cases}
			-\Delta_g w_\epsilon (x)+\frac{2(\rho_\epsilon+\rho')}{|M|}\delta_\epsilon^2=2\rho_\epsilon K (x_\epsilon+\delta_\epsilon x )e^{w_\epsilon (x)},\ &x\in\{x\in\R^2\ |\ x_\epsilon+\delta_\epsilon x\in D^+\}\\
			\partial_\nu w_\epsilon(x)=2\rho' \widetilde{h}_\epsilon (x_\epsilon+\delta_\epsilon x)e^{\frac{w_\epsilon (x)}{2}}, \ &x\in\{x\in\R^2\ |\ x_\epsilon+\delta_\epsilon x\in \partial^0 D^+\},
		\end{cases}
	\end{align*}
	where $g(x)=e^{\varphi(x_\epsilon+\delta_\epsilon x)}((dx^1)^2+(dx^2)^2)$.
	
	\
	
\begin{lem}
We have $$\limsup_{\epsilon\searrow 0}\frac{dist(x_\epsilon,\ \partial^0D^+)}{\delta_\epsilon}\leq C<+\infty.$$
\end{lem}
\begin{proof}
If not, we may assume  $$\lim_{\epsilon\to 0}\frac{dist(x_\epsilon,\ \partial^0D^+)}{\delta_\epsilon}=+\infty.$$ Then $$\{x\in\R^2\ |\ x_\epsilon+\delta_\epsilon x\in D^+\}\to \R^2,\ \ as\ \ \epsilon\searrow 0.$$
Since $w_\epsilon(x)\leq w_\epsilon (0)=0$, the standard estimates for the Liouville type equation yield $\|w_\epsilon\|_{C^3(D_R)}\leq C$. Passing to a subsequence, we may assume that $$w_\epsilon\rightharpoonup w\ \ in\ \ C^{2}_{loc}(\R^2),$$ where $w$ satisfies
	\begin{align*}
		-\Delta w= 2\rho K(p)e^{w}\ \ in \ \ \R^2.
	\end{align*}
The classification result in \cite{Chen-Li} gives that $$w(x)=-2\log(1+\frac{2\rho K(p)}{8}|x|^2).$$
Thus we have
\begin{align*}
2\rho=\lim_{\epsilon\to 0}2\rho_\epsilon\int_M Ke^{v_\epsilon}dv_g\geq \int_{\R^2}2\rho K(p)e^{w}dx=8\pi
\end{align*} which yields $\rho\geq 4\pi$. This is a contradiction to $\rho\leq 2\pi$.
	
\end{proof}
	
	\

	Passing to a subsequence, we may assume $$\lim_{\epsilon\searrow 0}\frac{dist(x_\epsilon,\ \partial^0D^+)}{\delta_\epsilon}=a.$$
	Then $$\Omega_\epsilon:=\{x\in\R^2\ |\ x_\epsilon+\delta_\epsilon x\in D^+\}\to \R_a^2:=\{x=(x^1,x^2)\in\R^2\ |\ x^2\geq -a\},\ \ as\ \ \epsilon\searrow 0.$$  Since $w_\epsilon(x)\leq w_\epsilon (0)=0$, by the standard estimates of Liouville type equation, for any $k,\ R>1$, we have $$\|w_\epsilon\|_{C^k(D_R\cap \Omega_\epsilon)}\leq C(k,R).$$ Passing to a subsequence, we get $$\lim_{\epsilon\to 0}\|w_\epsilon- w\|_{C^k(D_R\cap\Omega_\epsilon\cap \R^2_a)}=0,$$ where $w$ is a limit solution on $\R^2_a$ satisfying
	\begin{align*}
		\begin{cases}
			-\Delta w=\widetilde{K}(0)e^w,\ \ &in\ \ \R^2_a,\\
			\partial_\nu w=\widetilde{h}(0)e^{\frac{w}{2}},\ \ &on\ \ \partial\R^2_a,
		\end{cases}
	\end{align*}where
	\begin{align*}
			\widetilde{K}(0)=2\rho K(p),\ \
			\widetilde{h}(0)=2\rho'h(p)\lim_{\epsilon\to 0}\frac{\mu_\epsilon}{\kappa_\epsilon}.
	\end{align*}
	By the classification results in \cite{Li-Zhu}, we have $$w(x)=\log\frac{8\lambda ^2}{\widetilde{K}(0)[\lambda^2+|x-(0,\frac{c\lambda}{\sqrt{2}}-a)|^2]^2},\  \ where\ \ c=-\frac{\widetilde{h}(0)}{\sqrt{\widetilde{K}(0)}},\ \lambda=\frac{2 \sqrt{2\widetilde{K}(0)}}{\widetilde{K}(0)+\frac{1}{2}\widetilde{h}^2(0)}.$$
Furthermore, since $w(x)\leq w(0)=0$, we can see that there must be $$a=0.$$

	\

	\begin{lem}
		We have $$\lim_{\epsilon\to 0}\frac{\mu_\epsilon^2}{\kappa_\epsilon^2}=\frac{K(p)}{(4\pi-\rho)h^2(p)}.$$
	\end{lem}
	\begin{proof}
		On the one hand, by \cite{Li-Zhu}, we have $$\int_{\R^2_+}\tilde{K}(0)e^wdx+ \int_{\partial\R^2_+}\tilde{h}(0)e^{\frac{w}{2}}=4\pi.$$ On the other hand, since $$\int_M Ke^{v_\epsilon}dv_g=1,\ \ \int_{\partial M}\widetilde{h}_\epsilon e^{\frac{v_\epsilon}{2}}ds_g=1,\ \ \rho+\rho'=2\pi,$$ then $$\int_{\R^2_+}\tilde{K}(0)e^wdx\leq 2\rho,\ \  \int_{\partial\R^2_+}\tilde{h}(0)e^{\frac{w}{2}}\leq 2\rho',$$ which implies the inequalities in the above estimates must be equalities.
		
		A direct computation yields $$2\rho'=\int_{\partial\R^2_+}\tilde{h}(0)e^{\frac{w}{2}}=\frac{\sqrt{8}\pi\tilde{h}(0)}{\sqrt{\tilde{K}(0) +\frac{1}{2}\tilde{h}^2(0)}}.$$ That is $$\frac{\tilde{h}^2(0)}{\tilde{K}(0)}=\frac{2\rho'^2}{4\pi^2-\rho'^2}.$$ Then one can  easily get the conclusion of the lemma from the definition of $(\tilde{K}(0),\tilde{h}(0))$.
	\end{proof}

	\

	\begin{lem}\label{lem:maximum-point-on-boundary}
		When $\epsilon>0$ is small enough, there must be
		$
		x_\epsilon\in\partial M.
		$
		
	\end{lem}
	
	\begin{proof}
		Suppose the conclusion is false. Then there
		exists a sequence $\epsilon_j\searrow 0$ such that
		$x_{\epsilon_j}\in M\setminus\partial M$. For simplicity of notations, we
		still denote it by $\epsilon$.
		
		In the conformal coordinate system introduced above, identifying $p$
		with the origin, write
		\[
		x_\epsilon=(s_\epsilon,t_\epsilon),
		\qquad
		t_\epsilon>0.
		\]
		Let
		\[
		x_\epsilon^0:=(s_\epsilon,0)\in\partial^0D^+.
		\]
		Since $a=0$, we have
		\[
		\lim_{\epsilon\to 0}\frac{t_\epsilon}{\delta_\epsilon}
		=\lim_{\epsilon\to 0}
		\frac{\operatorname{dist}
			\left(x_\epsilon,\partial^0D^+\right)}
		{\delta_\epsilon}
		= 0.
		\]
		Set
		\[
		y_\epsilon
		:=
		\frac{x_\epsilon^0-x_\epsilon}{\delta_\epsilon}
		=
		\left(0,-\frac{t_\epsilon}{\delta_\epsilon}\right).
		\]
		Then $y_\epsilon$ lies on the flat part of
		$\partial\Omega_\epsilon$ and
		$
		y_\epsilon\longrightarrow (0,0).
		$
		
		Since $x_\epsilon$ is an interior maximum point of $v_\epsilon$, we
		have
		$
		\nabla v_\epsilon(x_\epsilon)=0.
		$
		Recalling the definition of $w_\epsilon$, it follows that
		\[
		w_\epsilon(0)=0,
		\qquad
		\nabla w_\epsilon(0)=0.
		\]
		Moreover, the boundary elliptic estimates obtained above imply that
		$
		\left\|w_\epsilon\right\|_
		{C^2(D_1(0)\cap\overline{\Omega_\epsilon})}
		\leq C.
		$
		Therefore, we have
		$$
		\begin{aligned}
			\left|\nabla w_\epsilon(y_\epsilon)\right|
			&=
			\left|
			\nabla w_\epsilon(y_\epsilon)-\nabla w_\epsilon(0)
			\right| \leq
			\left\|D^2w_\epsilon\right\|_
			{L^\infty(D_1(0)\cap\Omega_\epsilon)}
			|y_\epsilon|
			\leq C|y_\epsilon|
			\longrightarrow 0,
		\end{aligned}
		$$
		and
		$$
		\begin{aligned}
			|w_\epsilon(y_\epsilon)|
			&=
			|w_\epsilon(y_\epsilon)-w_\epsilon(0)| \leq
			\left\|\nabla w_\epsilon\right\|_
			{L^\infty(D_1(0)\cap\Omega_\epsilon)}
			|y_\epsilon|
			\leq C|y_\epsilon|
			\longrightarrow 0.
		\end{aligned}
		$$
		In particular,
		$
		\frac{\partial w_\epsilon}{\partial\nu}(y_\epsilon)
		\longrightarrow 0.
		$
		
		On the other hand, the boundary condition satisfied by
		$w_\epsilon$ gives
		$$
		\frac{\partial w_\epsilon}{\partial\nu}(y_\epsilon)
		=
		2\rho'
		\widetilde{h}_\epsilon(x_\epsilon^0)
		e^{\frac{w_\epsilon(y_\epsilon)}{2}}= 2\rho'
			h(x_\epsilon^0)
			\frac{\mu_\epsilon}{\kappa_\epsilon}e^{\frac{w_\epsilon(y_\epsilon)}{2}} \longrightarrow
			2\rho'h(p)
			\sqrt{
				\frac{K(p)}
				{(4\pi-\rho)h^2(p)}
			}
			\neq 0
		$$
		which contradicts
		\[
		\frac{\partial w_\epsilon}{\partial\nu}(y_\epsilon)
		\longrightarrow 0.
		\]
		Consequently, $x_\epsilon\in\partial M$ for all sufficiently small
		$\epsilon>0$.
	\end{proof}

\

Following the proof of Lemma 5.1 in \cite{Jost-Zhou-Zhu} and applying elliptic regularity, we have
\begin{lem}
	Passing to a subsequence, there holds
	$$
	v_\epsilon-\frac{1}{|M|}\int_Mv_\epsilon\,dv_g
	\longrightarrow G
	$$
	in $C^\infty_{\mathrm{loc}}(M\setminus\{p\})$ and weakly in
	$W^{1,q}(M)$ for every $1<q<2$, where $G$ is defined in \eqref{eq:intro-green-function}.
	
	There exists $R>0$ small enough such that
	$$
	G(x,p)=4\log\frac{1}{|x-p|}+g(x),
	\qquad
	x\in \overline{B_R(p)}\setminus\{p\},
	$$
	where $g\in C^\infty(\overline{B_R(p)})$.
\end{lem}

\

Let $G(x,p)=-4\log r+A(p)+b_1x_1+b_2x_2+c_1x_1^2+2c_2x_1x_2+c_3x_2^2+O(r^3),$ where $r(x)=dist(x, p)$.
	\begin{lem}\label{lem:02}
		Let
		$
		G_\epsilon(x):=G(x,x_\epsilon).
		$
		In $M \backslash B_{R\delta_\epsilon}(x_\epsilon)$, we have
		\begin{equation}
			\left.
			\begin{aligned}[b]
				v_\epsilon \geq{}&
				G_\epsilon
				+2\log\delta_\epsilon
				+\log\frac{8\lambda^2}{\widetilde{K}(0)}
				+2\log\frac{R^2}
				{\lambda^2+\left(R-\frac{c\lambda}{\sqrt{2}}\right)^2}
				-A(p)
				+o_\epsilon(1).
			\end{aligned}
			\right.
		\end{equation}
	\end{lem}
	
	\begin{proof}
		In $M \backslash B_{R\delta_\epsilon}(x_\epsilon)$,
		\begin{equation*}
			\left.
			\begin{aligned}
				\Delta(v_\epsilon-G_\epsilon)
				={}&
				\frac{2(\rho+\rho'-\epsilon)}{|M|}
				-2(\rho-\epsilon)K e^{v_\epsilon}
				-\frac{4\pi}{|M|} \\
				={}&
				-\frac{2\epsilon}{|M|}
				-2(\rho-\epsilon)K e^{v_\epsilon}
				\leq 0.
			\end{aligned}
			\right.
		\end{equation*}
		
		Since $x_\epsilon\to p$ and the regular part $A$ of the Green
		function is continuous on $\partial M$ (c.f. \cite{Hu-Bartsch} Lemma 2.1), we have
		$
		A(x_\epsilon)=A(p)+o_\epsilon(1).
		$
		Therefore,
		\begin{equation*}
			\left.
			\begin{aligned}
				(v_\epsilon-G_\epsilon)
				\big|_{\partial B_{R\delta_\epsilon}(x_\epsilon)}
=&-2\log\delta_\epsilon
				+w\big|_{\partial^+B_R(0)}
				+4\log(R\delta_\epsilon)
				-A(x_\epsilon)+o_{\epsilon}(1)
				 \\
				=&
				-2\log\delta_\epsilon
				+w\big|_{\partial^+B_R(0)}
				+4\log(R\delta_\epsilon)
				-A(p)
				+o_{\epsilon}(1) \\
				\geq &
				2\log\delta_\epsilon
				+\log\frac{8\lambda^2}{\widetilde{K}(0)}
				+2\log\frac{R^2}
				{\lambda^2+\left(R-\frac{c\lambda}{\sqrt{2}}\right)^2}
				-A(p)
				+o_\epsilon(1).
			\end{aligned}
			\right.
		\end{equation*}
		Moreover,
		\[
		\frac{\partial(v_\epsilon-G_\epsilon)}{\partial\nu}
		\bigg|_{\partial M\setminus B_{R\delta_\epsilon}(x_\epsilon)}
		=
		\frac{\partial v_\epsilon}{\partial\nu}
		\bigg|_{\partial M\setminus B_{R\delta_\epsilon}(x_\epsilon)}
		\geq 0.
		\]
		Applying the strong maximum principle, we get the conclusion of the
		lemma.
	\end{proof}
	
	\begin{lem}\label{lem:01}
		There holds
		\begin{align*}
			\lim\limits_{R\rightarrow +\infty}\lim\limits_{\epsilon\rightarrow 0}\int_{M \backslash B_{R\delta_\epsilon}(x_\epsilon)}\rho_\epsilon K e^{v_\epsilon}dv_g
			+\lim\limits_{R\rightarrow +\infty}\lim\limits_{\epsilon\rightarrow 0}\int_{ \partial M \backslash B_{R\delta_\epsilon}(x_\epsilon)}\rho'  \tilde{h}_\epsilon e^{\frac{v_\epsilon}{2}}ds_g=0.
		\end{align*}
	\end{lem}
	\begin{proof}
		On  one hand, we have
		\begin{align*}
			\lim\limits_{R\rightarrow +\infty}\lim\limits_{\epsilon\rightarrow 0}\left(\int_{ B_{R\delta_\epsilon}(x_\epsilon)}2\rho_\epsilon K e^{v_\epsilon}
			+\int_{ \partial M \cap B_{R\delta_\epsilon}(x_\epsilon)}2\rho'  \tilde{h}_\epsilon e^{\frac{v_\epsilon}{2}}\right)= \int_{\R^2_+}\tilde{K}(0)e^wdx+ \int_{\partial\R^2_+}\tilde{h}(0)e^{\frac{w}{2}}=4\pi.
		\end{align*}
		
		On the other hand,
		\begin{align*}
			\lim\limits_{\epsilon\rightarrow 0}\left(\int_{M}2\rho_\epsilon K e^{v_\epsilon}dv_g
			+\int_{ \partial M }2\rho' \tilde{h}_\epsilon e^{\frac{v_\epsilon}{2}}ds_g\right)=\lim_{\epsilon\to 0}2(\rho_\epsilon+\rho')=4\pi,
		\end{align*} which immediately implies the conclusion of the lemma.
		
	\end{proof}

	Now, we prove Proposition \ref{prop:lower-bound}.
\begin{proof}[\textbf{Proof of Proposition \ref{prop:lower-bound} for the case $\rho'>0$}]
	
On one hand, it is easy to see that
	\begin{equation}
		\left.
		\begin{aligned}[b]
			\int_{ B_{R\delta_\epsilon}(x_\epsilon)}|\nabla v_\epsilon|^2&=\int_{ B_{R}^+(0)}|\nabla w|^2+o_\epsilon(1)\\
			&=8\pi \log(\lambda^2+R^2) -8\pi\log\lambda^2-8\pi-8\pi\log\frac{2\pi+\rho'}{2\pi-\rho'}+4\rho'+o_R(1)+o_\epsilon(1).
		\end{aligned}
		\right.
	\end{equation}
	
On the other hand, from the divergence theorem and equation, we have
	\begin{equation}
		\left.
		\begin{aligned}[b]
			\int_{M \backslash B_{R\delta_\epsilon}(x_\epsilon)}|\nabla  v_\epsilon|^2
			=&\int_{M \backslash B_{R\delta_\epsilon}(x_\epsilon)}-v_\epsilon\Delta   v_\epsilon+\int_{ \partial (M \backslash B_{R\delta_\epsilon}(x_\epsilon))}v_\epsilon \frac { \partial v_\epsilon}{\partial \nu}\\
			=&-\frac{2(\rho+\rho'-\epsilon)}{|M|} \int_{M \backslash B_{R\delta_\epsilon}(x_\epsilon)}v_\epsilon+2(\rho-\epsilon)\int_{M \backslash B_{R\delta_\epsilon}(x_\epsilon)} K e^{v_\epsilon} v_\epsilon\\
			&-\int_{ \partial^+ B_{R\delta_\epsilon}(x_\epsilon)}v_\epsilon \frac { \partial v_\epsilon}{\partial r} +2\rho'\int_{ \partial M \backslash B_{R\delta_\epsilon}(x_\epsilon)}  \tilde{h}_\epsilon e^{\frac{v_\epsilon}{2}}v_\epsilon.
		\end{aligned}
		\right.
	\end{equation}
	
	By Lemma \ref{lem:02} and Lemma \ref{lem:01}, we have
	\begin{align*}
		&2\rho_\epsilon\int_{M \setminus B_{R\delta_\epsilon}(x_\epsilon)} K e^{v_\epsilon} v_\epsilon+2\rho'\int_{ \partial M \setminus B_{R\delta_\epsilon}(x_\epsilon)}  \tilde{h}_\epsilon e^{\frac{v_\epsilon}{2}}v_\epsilon\\
		\geq&2\rho_\epsilon\int_{M \setminus B_{R\delta_\epsilon}(x_\epsilon)} K e^{v_\epsilon}G_\epsilon+4\rho_\epsilon\log\delta_\epsilon\int_{M \setminus B_{R\delta_\epsilon}(x_\epsilon)}K e^{v_\epsilon}\\
		&+2\rho_\epsilon\int_{M \setminus B_{R\delta_\epsilon}(x_\epsilon)} K e^{v_\epsilon}\left(\log\frac{8\lambda^2}{\tilde{K}(0)}+2\log\frac{R^2}{\lambda^2+(R-\frac{c\lambda}{\sqrt{2}})^2}- A(p)  + o_\epsilon(1)+o_R(1)\right)\\
		&+2\rho'\int_{ \partial M \setminus B_{R\delta_\epsilon}(x_\epsilon)} \tilde{h}_\epsilon e^{\frac{v_\epsilon}{2}} G_\epsilon+4\rho' \log\delta_\epsilon\int_{ \partial M \setminus B_{R\delta_\epsilon}(x_\epsilon)}  \tilde{h}_\epsilon e^{\frac{v_\epsilon}{2}}\\
		&+2\rho'\int_{ \partial M \setminus B_{R\delta_\epsilon}(x_\epsilon)}  \tilde{h}_\epsilon e^{\frac{v_\epsilon}{2}}\left(\log\frac{8\lambda^2}{\tilde{K}(0)}+2\log\frac{R^2}{\lambda^2+(R-\frac{c\lambda}{\sqrt{2}})^2}- A(p)  + o_\epsilon(1)+o_R(1)\right)\\
		\geq&2\rho_\epsilon\int_{M \setminus B_{R\delta_\epsilon}(x_\epsilon)} K e^{v_\epsilon}G_\epsilon+4\rho_\epsilon \log\delta_\epsilon\int_{M \setminus B_{R\delta_\epsilon}(x_\epsilon)} K e^{v_\epsilon}\\
		&+2\rho'\int_{ \partial M \setminus B_{R\delta_\epsilon}(x_\epsilon)}  \tilde{h}_\epsilon e^{\frac{v_\epsilon}{2}} G_\epsilon+4\rho' \log\delta_\epsilon\int_{ \partial M \setminus B_{R\delta_\epsilon}(x_\epsilon)} \tilde{h}_\epsilon e^{\frac{v_\epsilon}{2}}+ o_\epsilon(1)+o_R(1).
	\end{align*}
	
	By the equation and  divergence theorem, we have
	\begin{align*}
		&2\rho_\epsilon\int_{M \setminus B_{R\delta_\epsilon}(x_\epsilon)} K e^{v_\epsilon}G_\epsilon+2\rho'\int_{ \partial M \setminus  B_{R\delta_\epsilon}(x_\epsilon)}  \tilde{h}_\epsilon e^{\frac{v_\epsilon}{2}} G_\epsilon\\
		=&\int_{M \setminus B_{R\delta_\epsilon}(x_\epsilon)} \left(-\Delta v_\epsilon+\frac{2(\rho_\epsilon+\rho')}{|M|}\right)G_\epsilon+2\rho'\int_{ \partial M \backslash B_{R\delta_\epsilon}(x_\epsilon)}  \tilde{h}_\epsilon e^{\frac{v_\epsilon}{2}} G_\epsilon\\
		=&\int_{M \backslash B_{R\delta_\epsilon}(x_\epsilon)} -\Delta G_\epsilon v_\epsilon
		-\int_{ \partial M \backslash B_{R\delta_\epsilon}(x_\epsilon)} \frac{\partial v_\epsilon}{\partial \nu}G_\epsilon+\int_{ \partial M \backslash B_{R\delta_\epsilon}(x_\epsilon)}  v_\epsilon \frac{\partial G_\epsilon}{\partial \nu}+ \int_{ \partial^+ B_{R\delta_\epsilon}(x_\epsilon)} \frac{\partial v_\epsilon}{\partial r}G_\epsilon \\
		&- \int_{ \partial^+ B_{R\delta_\epsilon}(x_\epsilon)} v_\epsilon \frac{\partial G_\epsilon}{\partial r}+\frac{2(\rho_\epsilon+\rho')}{|M|}\int_{M \setminus B_{R\delta_\epsilon}(x_\epsilon)} G_\epsilon+2\rho'\int_{ \partial M \setminus B_{R\delta_\epsilon}(x_\epsilon)}  \tilde{h}_\epsilon e^{\frac{v_\epsilon}{2}} G_\epsilon\\
		=&-\frac{4\pi}{|M|}\int_{M \backslash B_{R\delta_\epsilon}(x_\epsilon)}  v_\epsilon+ \int_{ \partial^+ B_{R\delta_\epsilon}(x_\epsilon)} \frac{\partial v_\epsilon}{\partial r}G_\epsilon - \int_{ \partial^+ B_{R\delta_\epsilon}(x_\epsilon)} v_\epsilon \frac{\partial G_\epsilon}{\partial r}
		+o_\epsilon(1).
	\end{align*}
	
	Similarly,
	\begin{align*}
		&4\rho_\epsilon \log\delta_\epsilon\int_{M \setminus B_{R\delta_\epsilon}(x_\epsilon)} K e^{v_\epsilon}+4\rho' \log\delta_\epsilon\int_{ \partial M \backslash B_{R\delta_\epsilon}(x_\epsilon)}  \tilde{h}_\epsilon e^{\frac{v_\epsilon}{2}}\\
		=&2\log\delta_\epsilon\int_{M \backslash B_{R\delta_\epsilon}(x_\epsilon)} \left(-\Delta v_\epsilon+\frac{2(\rho_\epsilon+\rho')}{|M|}\right)+4\rho' \log\delta_\epsilon\int_{ \partial M \backslash B_{R\delta_\epsilon}(x_\epsilon)} \tilde{h}_\epsilon e^{\frac{v_\epsilon}{2}} \\
		=& 2\log\delta_\epsilon \int_{\partial^+ B_{R\delta_\epsilon}(x_\epsilon)} \frac{\partial v_\epsilon}{\partial r} +\frac{4(\rho_\epsilon+\rho')}{|M|}\log\delta_\epsilon |M \backslash B_{R\delta_\epsilon}(x_\epsilon)|.
	\end{align*}
	
	Now, we arrive at
	\begin{align*}
		\int_{M \backslash B_{R\delta_\epsilon}(x_\epsilon)}|\nabla  v_\epsilon|^2
		\geq&-\frac{2(\rho_\epsilon+\rho')}{|M|} \int_{M \backslash B_{R\delta_\epsilon}(x_\epsilon)}v_\epsilon-\int_{ \partial^+ B_{R\delta_\epsilon}(x_\epsilon)}v_\epsilon \frac { \partial v_\epsilon}{\partial r} \\
		&-\frac{4\pi}{|M|}\int_{M \backslash B_{R\delta_\epsilon}(x_\epsilon)}  v_\epsilon + \int_{ \partial^+ B_{R\delta_\epsilon}(x_\epsilon)} \frac{\partial v_\epsilon}{\partial r}G_\epsilon - \int_{\partial^+ B_{R\delta_\epsilon}(x_\epsilon)} v_\epsilon \frac{\partial G_\epsilon}{\partial r}\\
		& +2\log\delta_\epsilon \int_{ \partial^+ B_{R\delta_\epsilon}(x_\epsilon)} \frac{\partial v_\epsilon}{\partial r} +\frac{4(\rho_\epsilon+\rho')}{|M|}\log\delta_\epsilon |M \backslash B_{R\delta_\epsilon}(x_\epsilon)|+o_\epsilon(1) \\
		=&-\frac{2(\rho_\epsilon+\rho'+2\pi)}{|M|} \int_{M \backslash B_{R\delta_\epsilon}(x_\epsilon)}v_\epsilon-\int_{\partial^+ B_{R\delta_\epsilon}(x_\epsilon)}\frac{\partial v_\epsilon}{\partial r}\left(v_\epsilon-G_\epsilon-2\log\delta_\epsilon\right)\\
		&-\int_{\partial^+ B_{R\delta_\epsilon}(x_\epsilon)}v_\epsilon \frac{\partial G_\epsilon}{\partial r}+4(\rho_\epsilon+\rho')\log\delta_\epsilon  + o_\epsilon(1)\\
		:=&\mathbf{I}+\mathbf{II}+\mathbf{III}+4(\rho_\epsilon+\rho')\log\delta_\epsilon + o_\epsilon(1).
	\end{align*}
	
For $\mathbf{I}$, since $$|v_\epsilon(x)|\leq -2\log\delta_\epsilon+C,\ \ \forall\ x\in B_{R\delta_\epsilon}(x_\epsilon),$$ then \begin{equation*}\left|\int_{B_{R\delta_\epsilon}(x_\epsilon)}v_\epsilon dv_g\right|\leq C(R)\delta_\epsilon^2|\log\delta_\epsilon|=o_\epsilon(1).\end{equation*}
	Computing directly,
	we have
	\begin{align*}
		\mathbf{I}=&-\frac{2(\rho_\epsilon+\rho'+2\pi)}{|M|} \int_{M\setminus B_{R\delta_\epsilon}(x_\epsilon) }v_\epsilon\\ =&-\frac{2(\rho_\epsilon+\rho'+2\pi)}{|M|} \int_{M}v_\epsilon+o_\epsilon (1)\\
		=&-\frac{2(\rho_\epsilon+\rho'+2\pi)}{|M|} \int_{M }(v_\epsilon-\overline{v}_\epsilon) -\frac{2(\rho_\epsilon+\rho'+2\pi)}{|M|} \int_{M }\overline{v}_\epsilon\\
		=&-\frac{2(\rho_\epsilon+\rho'+2\pi)}{|M|}  \int_{M }G_\epsilon-2(\rho_\epsilon+\rho'+2\pi)  \bar{v}_\epsilon+o_\epsilon(1)\\=&-2(\rho_\epsilon+\rho'+2\pi)  \bar{v}_\epsilon+o_\epsilon(1).
	\end{align*}
	
	By Lemma \ref{lem:02}, we get
	\begin{align*}
		\mathbf{II}=&-\int_{\partial^+ B_{R\delta_\epsilon}(x_\epsilon)}\frac{\partial v_\epsilon}{\partial r}\left(v_\epsilon-G_\epsilon-2\log\delta_\epsilon\right)\\
		\geq& -\int_{\partial^+ B^+_{R}(0)}\frac{\partial w}{\partial r}dS\cdot \left(- A(p)+\log\frac{8\lambda^2}{\tilde{K}(0)} + o_\epsilon(1)+o_R(1)\right) \\
		=& 4\pi \left(- A(p)+\log\frac{8\lambda^2}{\tilde{K}(0)}\right)+ o_\epsilon(1)+o_R(1).
	\end{align*}
	
	For the term $\mathbf{III}$, since \begin{align*}
v_\epsilon\bigg|_{\partial^+ B_{R\delta_\epsilon}(x_\epsilon)}&=w|_{\partial^+B_R(0)}-2\log\delta_\epsilon+o_\epsilon(1)\\
&=-2\log\delta_\epsilon+\log\frac{8\lambda^2}{\tilde{K}(0)}-2\log\big(\lambda^2+(R-\frac{c\lambda}{\sqrt{2}})^2\big) +o_\epsilon(1)+o_R(1), \end{align*} we obtain
	\begin{align*}
		\mathbf{III}&=-\int_{\partial^+ B_{R\delta_\epsilon}(x_\epsilon)}v_\epsilon \frac{\partial G_\epsilon}{\partial r}\\		&=-\left(-2\log\delta_\epsilon+\log\frac{8\lambda^2}{\tilde{K}(0)}-2\log\big(\lambda^2+(R-\frac{c\lambda}{\sqrt{2}})^2\big)\right) \int_{\partial^+ B_{R\delta_\epsilon}(x_\epsilon)} \frac{\partial G_\epsilon}{\partial r}+o_\epsilon (1)+o_R(1)\\		&=\left(-2\log\delta_\epsilon+\log\frac{8\lambda^2}{\tilde{K}(0)}-2\log\big(\lambda^2+(R-\frac{c\lambda}{\sqrt{2}})^2\big)\right)\int_{M\setminus  B_{R\delta_\epsilon}(x_\epsilon)} \Delta G_\epsilon+o_\epsilon (1)+o_R(1)\\
		&=4\pi\left(-2\log\delta_\epsilon+\log\frac{8\lambda^2}{\tilde{K}(0)} -2\log\big(\lambda^2+(R-\frac{c\lambda}{\sqrt{2}})^2\big)\right)+o_\epsilon (1)+o_R(1).
	\end{align*}
	
	Combining these together, we have
	\begin{align*}
		&\int_{M \backslash B_{R\delta_\epsilon}(x_\epsilon)}|\nabla  v_\epsilon|^2\\
		\geq& -2(\rho_\epsilon+\rho'+2\pi)  \bar{v}_\epsilon+ 4\pi\left(-A(p)+2\log\frac{8\lambda^2}{\tilde{K}(0)}-2\log(\lambda^2+R^2)\right)-4\epsilon\log\delta_\epsilon+o_\epsilon(1)+o_R(1).
	\end{align*}
	Therefore, we have
	\begin{align*}
		\int_{ M}|\nabla v_\epsilon|^2=&\int_{M \backslash B_{R\delta_\epsilon}(x_\epsilon)}|\nabla  v_\epsilon|^2+\int_{ B_{R\delta_\epsilon}(x_\epsilon)}|\nabla  v_\epsilon|^2\\
		\geq& -2(\rho_\epsilon+\rho'+2\pi)  \bar{v}_\epsilon-4\epsilon\log\delta_\epsilon+ 4\pi\left(-A(p)+2\log\frac{8\lambda^2}{\tilde{K}(0)}-2\log\lambda^2-2\right)\\ & -8\pi\log\frac{2\pi+\rho'}{2\pi-\rho'}+4\rho'+o_\epsilon(1)+o_R(1),
	\end{align*} which implies
	\begin{align*}
		\frac{1}{2}\int_M|\nabla v_\epsilon|^2+\frac{2(\rho_\epsilon+\rho')}{|M|}\int_M v_\epsilon
		\geq&-2\epsilon\log\delta_\epsilon-\epsilon\bar{v}_\epsilon+2\pi\left(-A(p)+2\log\frac{8}{\tilde{K}(0)}-2\right)\\& -4\pi\log\frac{2\pi+\rho'}{2\pi-\rho'}+2\rho'+o_\epsilon(1)+o_R(1)
	\end{align*}
	
	By \eqref{equation:06}, we have
	\begin{align*}
		J_{\rho_\epsilon,\rho'}(v_\epsilon)\geq& -2\epsilon\log\delta_\epsilon-\epsilon\bar{v}_\epsilon-4\pi-2\pi A(p)+ 4\pi\log 4+2\rho'\\& -2\rho\log K(p)-2\rho'\log h^2(p)-2(2\pi+\rho')\log(2\pi+\rho')+o_\epsilon(1).\end{align*} Hence,
	\begin{align}\label{equation:08}
		\liminf_{\epsilon \to 0}J_{\rho_\epsilon,
			\rho'}(v_\epsilon)\geq& -4\pi+2\rho'+ 4\pi\log 4-2(2\pi+\rho')\log(2\pi+\rho')\notag\\&-2\pi \max_{p\in \partial M}\left(A(p) +\frac{\rho}{\pi}\log K(p)+\frac{\rho'}{\pi}\log h^2(p)\right),\end{align} where $\rho+\rho'=2\pi$.
\end{proof}

\

\begin{proof}[\textbf{Proof of Proposition \ref{prop:lower-bound} for the case $\rho'=0$}]

	Using the notations as before, then we have $\int_MKe^{v_\varepsilon}\,dv_g=1$ and
	$$
	\begin{cases}
		-\Delta_gv_\varepsilon
		+\dfrac{4\pi-2\varepsilon}{|M|}
		=
		(4\pi-2\varepsilon)Ke^{v_\varepsilon},
		&\text{in }M,\\[6pt]
		\partial_\nu v_\varepsilon=0,
		&\text{on }\partial M.
	\end{cases}
	$$

	 A similar argument yields that passing to a subsequence, there holds
	$$
	\lambda_\varepsilon
	=
	v_\varepsilon(x_\varepsilon)
	=
	\max_{\overline M}v_\varepsilon\to +\infty,\ \ x_\epsilon\to p\in\partial M,\ \
	\frac{\operatorname{dist}(x_\varepsilon,\partial M)}
	{\delta_\varepsilon}\to a=0,\ \ as\ \ \epsilon\to 0.
	$$

	In the conformal coordinates used above, write
	$x_\varepsilon=(s_\varepsilon,t_\varepsilon)$ and set
	$$
	w_\varepsilon(x)
	=
	v_\varepsilon(x_\varepsilon+\delta_\varepsilon x)
	+2\log\delta_\varepsilon.
	$$
	Then
	$$
	w_\varepsilon(x)
	\longrightarrow
	w(x)
	=
	-2\log\left(1+\frac{\pi K(p)}2|x|^2\right)
	$$
	in $C^2_{\mathrm{loc}}(\overline{\mathbb R^2_+})$, where $w(x)$ satisfies $$-\Delta w=4\pi K(p)e^w\ \ in\ \ \R^2_+,\ \ \partial_\nu w=0\ \ on\ \ \partial \R^2_+.$$
	In particular,
	$$
	w(x)=-2\log\left(1+\frac{\pi K(p)}{2}|x|^2\right),\ \ \int_{\mathbb R^2_+}4\pi K(p)e^w\,dx=4\pi.
	$$

	Now we claim the conclusion of Lemma~\ref{lem:maximum-point-on-boundary}  still holds.
	
	Otherwise, there exists a sequence
	$\varepsilon\searrow0$ such that $t_\varepsilon>0$. Since
	$$
	\partial_{22}w(0)=-2\pi K(p)<0,
	$$
	by the preceding $C^2$ convergence and the fact
	$t_\varepsilon/\delta_\varepsilon\to0$, we get $$\partial^2_{22}w_\epsilon (0)<-\frac{3}{2}\pi K(p)<0,$$ when $\epsilon$ is small.

Then for any $-\frac{t_\epsilon}{\delta_\epsilon}\leq x_2\leq 0$, we get
\begin{align*}
\partial^2_{22}w_\epsilon(0,x_2)\leq \partial^2_{22}w_\epsilon(0)+\|D^3 w_\epsilon\|_{L^\infty(D_1(0)\cap \overline{\Omega_\epsilon})}\frac{t_\epsilon}{\delta_\epsilon}<-\pi K(p)
\end{align*}when $\epsilon$ is small.

	Therefore, we obtain
	$$
	\begin{aligned}
		0
		&=
		\partial_2w_\varepsilon(0)
		-
		\partial_2w_\varepsilon(0,-\frac{t_\epsilon}{\delta_\epsilon})&=
		\int^0_{-\frac{t_\varepsilon}{\delta_\epsilon}}
		\partial_{22}w_\varepsilon(0,t)\,dt\leq
		-\frac{\pi K(p)t_\varepsilon}{\delta_\varepsilon}
		<0,
	\end{aligned}
	$$
	which is a contradiction. Thus $x_\varepsilon\in\partial M$
	for all sufficiently small $\varepsilon$.
	
	One can check all the remaining proofs in the previous case still work. We finished the proof of Proposition \ref{prop:lower-bound}.
	
\end{proof}

	\

	\section{Construct the blow-up sequence}\label{sec:construc-blowup-seq}

\
	
In this section, we firstly construct the blow-up sequence $\{\phi_{\epsilon,\delta}\}$ such that $$J_{\rho_\epsilon,\rho'}(\phi_{\epsilon,\delta})< -4\pi+2\rho'+ 4\pi\log 4-2(2\pi+\rho')\log(2\pi+\rho')-2\pi \max_{p\in \partial M}\left(A(p) +\frac{\rho}{\pi}\log K(p)+\frac{\rho'}{\pi}\log h^2(p)\right),$$ when $\epsilon$ and $\delta$ are small enough. Combining this with the lower bound obtained in Section 2, we will prove Theorem \ref{thm:main-01} and Theorem \ref{thm:main-02}.

\

We start this section by proving the following lemma.

	\begin{lem}\label{lem:04}
		Let  $G(x, p)$ be the Green function defined by \eqref{eq:intro-green-function}. Let $p\in\partial M$ and $\Pi(p)$, $\kappa(p)$ be the Gaussian curvature and geodesic curvature of $(M, g)$ and $\partial M$ at $p$ respectively. Then the following identity holds
		$$b_2=-2\kappa,\ \ \ \frac{2\pi}{3}\Pi+c_1\pi+c_3\pi=\frac{2\pi^2}{|M|}.$$
	\end{lem}
	\begin{proof}
		In a normal coordinate system near $p$,
		the Riemannian metric can be written as $g = dr^2 + g^2(r,\theta)d\theta^2$ with
		$$g(r, \theta)=r-\frac{\Pi(p)}{6}r^3+O(r^4).$$
		By the divergence theorem, we have
		\begin{align}\label{equation:05}
			\int_{\partial^+B_r(p)}\frac{\partial G}{\partial\nu}=&\int_{M \backslash B_r(p)}-\Delta G +\int_{\partial M \backslash B_r(p)} \frac{\partial G}{\partial\nu}\notag\\
			=&-\frac{4\pi}{|M|}(|M|-|B_r(p)|)\notag\\
			=&-4\pi+\frac{4\pi}{|M|}\frac{\pi r^2}{2}+o(r^2).
		\end{align}
		On the other hand,
		\begin{align*}
			\int_{\partial^+B_r(p)}\frac{\partial G}{\partial\nu}=\int_{\theta_1(r)}^{\pi-\theta_2(r)}&\left(-\frac{4}{r}+b_1\cos\theta+b_2\sin\theta+2c_1r \cos^2\theta +4c_2r\sin\theta\cos\theta+2c_3 r \sin^2 \theta\right)\\
			&\cdot \left(r-\frac{1}{6}\Pi(p)r^3+O(r^4)\right)d\theta .
		\end{align*}
		
		Let $p=0$ and $x_1,x_2$ be rectangular coordinates on $T_0M$ such that $\frac{\partial}{\partial x_2}$ is the inward unit normal vector to $\partial M$ at $0\in \partial M$. By \cite{Glutsyuk} or Lemma \ref{lem:03} in the appendix, in a small neighborhood of $0$, we can write $\partial M$ as a graph $(x_1, f(x_1))$, where
		$$
		f(x_1)=\frac{1}{2}\kappa x_1^2+\frac{1}{6}\kappa'x_1^3+O(x_1^4)
		$$
		and $\kappa=\kappa(p)$ is the geodesic curvature of $\partial M$ at $p\in \partial M$, $\kappa'=\nabla_\tau\kappa (p)$.
		
		The circle $\partial^+B_r(p)$ intersects $\partial M$ at two points
		$$
		(R_+(r),f(R_+(r)))\quad\mbox{and}\quad
		(-R_-(r),f(-R_-(r))),
		$$
		where $R_+(r),R_-(r)>0$. Then
		$$
		r^2=R_+^2(r)+f^2(R_+(r))
		=R_-^2(r)+f^2(-R_-(r))
		=R_\pm^2(r)+O(R_\pm^4(r)),
		$$
		which implies
		$$
		R_+(r)=r+O(r^3),\qquad R_-(r)=r+O(r^3).
		$$
		For the positive intersection point, we have
		$$
		\tan\theta_1(r)=\frac{f(R_+(r))}{R_+(r)}
		=\frac{1}{2}\kappa R_+(r)+\frac{1}{6}\kappa'R_+^2(r)+O(R_+^3(r))
		=\frac{1}{2}\kappa r+\frac{1}{6}\kappa'r^2+O(r^3),
		$$
		which implies
		$$
		\theta_1(r)=\arctan\left(\frac{1}{2}\kappa r+\frac{1}{6}\kappa'r^2+O(r^3)\right)
		=\frac{1}{2}\kappa r+\frac{1}{6}\kappa'r^2+O(r^3).
		$$
		For the negative intersection point, we have
		$$
		f(-R_-(r))
		=\frac{1}{2}\kappa R_-^2(r)-\frac{1}{6}\kappa'R_-^3(r)+O(R_-^4(r)),
		$$
		and hence
		$$
		\tan\theta_2(r)=\frac{f(-R_-(r))}{R_-(r)}
		=\frac{1}{2}\kappa R_-(r)-\frac{1}{6}\kappa'R_-^2(r)+O(R_-^3(r))
		=\frac{1}{2}\kappa r-\frac{1}{6}\kappa'r^2+O(r^3),
		$$
		which implies
		$$
		\theta_2(r)=\frac{1}{2}\kappa r-\frac{1}{6}\kappa'r^2+O(r^3).
		$$
		Therefore,
		$$
		\theta_1(r)+\theta_2(r)=\kappa r+O(r^3).
		$$
		
		A direct computation yields
		\begin{align*}
			\int_{\theta_1(r)}^{\pi-\theta_2(r)}-\frac{4}{r}\left(r-\frac{1}{6}\Pi(p)r^3+O(r^4)\right)d\theta
			&=(-4+\frac{2}{3}\Pi r^2)(\pi-\theta_1(r)-\theta_2(r))+O(r^3)\\
			&=-4\pi+4\kappa r+\frac{2\pi}{3}\Pi r^2+O(r^3),
		\end{align*}
		\begin{align*}
			&\int_{\theta_1(r)}^{\pi-\theta_2(r)}(b_1\cos\theta+b_2\sin\theta)(r-\frac{1}{6}\Pi(p)r^3+O(r^4))d\theta \\
			&=b_1r(\sin(\pi-\theta_2)-\sin\theta_1)+b_2r(\cos\theta_1-\cos(\pi-\theta_2))+O(r^3)\\
			&=2b_2r+O(r^3)
		\end{align*}
		and
		\begin{align*}
			&\int_{\theta_1(r)}^{\pi-\theta_2(r)}\left(2c_1r \cos^2\theta +4c_2r\sin\theta\cos\theta+2c_3 r \sin^2 \theta\right)\left(r-\frac{1}{6}\Pi(p)r^3+O(r^4)\right)d\theta\\&=(c_1+c_3)\pi r^2+O(r^3).
		\end{align*}
		Therefore,
		\begin{align}\label{equation:04}
			\int_{\partial^+B_r(p)}\frac{\partial G}{\partial\nu}
			=-4\pi+(4\kappa+2b_2)r+
			\left(\frac{2\pi}{3}\Pi+c_1\pi+c_3\pi\right)r^2+O(r^3).
		\end{align}
		By comparing \eqref{equation:04} with \eqref{equation:05}, we get
		$$
		b_2=-2\kappa,\qquad
		\frac{2\pi}{3}\Pi+c_1\pi+c_3\pi=\frac{2\pi^2}{|M|}.
		$$
	\end{proof}

	\

	Now we begin to construct the blow-up sequence. Fix $p\in\partial M$. Let $x_1,x_2$ be the normal coordinates at $p$ such that $\frac{\partial}{\partial x_2}$ is the inward unit normal vector at $p=0$. Let
	$
	r=\sqrt{x_1^2+x_2^2}.
	$

	Choose $\alpha=\alpha_\epsilon$ such that
	$$
	\alpha^4\epsilon
	=
	\frac{1}{\log(-\log\epsilon)}.
	$$

Set
	$$
	a
	=
	\frac{\rho'}{\sqrt{4\pi^2-\rho'^2}},
	\ \
	r_\epsilon^2
	=
	x_1^2
	+
	\bigl(x_2+a\sqrt{\epsilon}\bigr)^2,\ \ R_\eps=\alpha\sqrt{\epsilon}.
	$$

	Fix $\delta>0$ sufficiently small and take $\epsilon>0$ such that
	$
	2R_\eps+a\sqrt{\epsilon}<\delta.
	$
	Choose $\eta_\delta\in C^\infty(M)$, $\chi_\epsilon\in C^\infty(M)$ such that
	$$
	0\leq\eta_\delta\leq1,\ \ |\nabla_g\eta_\delta|
	\leq
	\frac{C}{\delta}\ \ \text{and}
	\ \
	\eta_\delta\equiv1
	\ \ \text{for }r\leq\delta,
	\ \
	\eta_\delta\equiv0
	\ \ \text{for }r\geq2\delta.
	$$	
	and
	$$
	0\leq\chi_\epsilon\leq1,
	\ \ |\nabla_g\chi_\epsilon|
	\leq
	\frac{C}{R_\eps}\ \ \text{and}\ \
	\chi_\epsilon\equiv1
	\ \ \text{for }r_\epsilon\leq R_\eps,
	\ \
	\chi_\epsilon\equiv0
	\ \ \text{for }r_\epsilon\geq2R_\eps.
	$$

	Define
	\begin{equation}\label{eq:test-function}
		\phi_{\epsilon,\delta}(x)
		=
		\begin{cases}
			\displaystyle
			-2\log\bigl(\epsilon+r_\epsilon^2\bigr)
			+
			A(p)
			+
			b_1(p)x_1
			+
			b_2(p)x_2+\log\eps,
			& r_\epsilon\leq R_\eps,
			\\[1.2em]
			\displaystyle
			-2\log\bigl(\epsilon+r_\epsilon^2\bigr)
			+
			\bigl(1-\chi_\epsilon(x)\bigr)
			\bigl[G(x,p)+4\log r\bigr]
			\\
			\displaystyle \quad
			+
			\chi_\epsilon(x)
			\bigl[
			A(p)
			+
			b_1(p)x_1
			+
			b_2(p)x_2
			\bigr]+\log\eps,
			& R_\eps<r_\epsilon<2R_\eps,
			\\[1.2em]
			\displaystyle
			G(x,p)
			+
			\eta_\delta(x)
			\left[
			4\log r
			-
			2\log\bigl(\epsilon+r_\epsilon^2\bigr)
			\right]+\log\eps,
			& r_\epsilon\geq2R_\eps.
		\end{cases}
	\end{equation}

\

In this section, we will prove the following proposition.

\begin{prop}\label{prop:01}
		Assume that
		$$
		\rho+\rho'=2\pi,
		\qquad
		0<\rho\leq2\pi,
		\qquad
		0\leq\rho'<2\pi.
		$$
		Then we have
		\begin{align*}
			&J_{\rho,\rho'}(\phi_{\epsilon,\delta})\\& =
			-4\pi
			+
			2\rho'
			+
			4\pi\log4
			-
			2(2\pi+\rho')\log(2\pi+\rho')
			-
			2\pi
			\left[
			A(p)
			+
			\frac{\rho}{\pi}\log K(p)
			+
			\frac{\rho'}{\pi}\log h^2(p)
			\right] \\
			&
			\quad -
			4\pi
			\sqrt{\frac{\rho}{2\pi+\rho'}}
			\Bigg[
			\frac{K_2(p)}{K(p)}
			-
			2\kappa(p)
			+
			\frac{2\rho'}
			{\pi\rho e^{A(p)/2}h(p)}
			\left(
			\int_{\partial M\cap\{r\geq\delta\}}
			h(x)e^{G(x,p)/2}\,ds_g
			-
			\frac{2e^{A(p)/2}h(p)}{\delta}
			\right)
			\Bigg]\sqrt{\epsilon} \\
			&\quad
			+
			O(\delta\sqrt{\epsilon})
			+
			o(\sqrt{\epsilon}).
		\end{align*}
		
	\end{prop}

\

For simplicity of notations, we assume $\kappa(p)>0$ in the proof of Proposition \ref{prop:01}. The case $\kappa(p)\leq 0$ is similar.

\
	
	\begin{lem}\label{lem:bubble-gradient}
		As $\epsilon\to0$, we have
		\begin{align}
			\frac12\int_{M\cap\{r_\epsilon\leq R_\eps\}}
			|\nabla_g\phi_{\epsilon,\delta}|^2\,dv_g
			={}&8\pi\log\alpha-4\pi
			-8\pi\log\bigl(a+\sqrt{1+a^2}\bigr)
			+\frac{4\pi a}{\sqrt{1+a^2}} \notag\\
			&+\frac{16a}{\alpha}+\frac{8\pi}{\alpha^2}
			-8\bigl(b_2(p)+\kappa(p)\bigr)R_\eps \notag\\
			&+\left(
			4\pi\sqrt{1+a^2}\,b_2(p)
			+\frac{2\pi\kappa(p)(2a^2+3)}{\sqrt{1+a^2}}
			\right)\sqrt{\epsilon}
			+o(\sqrt{\epsilon}).
			\label{eq:bubble-gradient}
		\end{align}
	\end{lem}
	
	\begin{proof}
		
		Noting that $r^2\leq C(r_\epsilon^2+\epsilon)$, $$|\nabla_g\phi_{\epsilon,\delta}|^2=|D\phi_{\epsilon,\delta}|^2+(g^{-2}-r^{-2}) |\frac{\partial}{\partial\theta}\phi_{\epsilon,\delta}|^2$$ and
		\begin{align*}
			\int_{\{r_\epsilon\leq R_\eps\}}|g^{-2}-r^{-2}||\frac{\partial}{\partial\theta}\phi_{\epsilon,\delta}|^2 dv_g&= \int_{\{r_\epsilon\leq R_\eps\}}\left|\frac{\Pi}{3}+O(r)\right|\left| \frac{-4a\sqrt{\epsilon}r\cos\theta}{r_\epsilon^2+\epsilon}-b_1r\sin\theta+b_2r\cos\theta\right|^2 dv_g\\&\leq C\int_{\{r_\epsilon\leq R_\eps\}}\left(\left| \frac{4a\sqrt{\epsilon}r\cos\theta}{r_\epsilon^2+\epsilon}\right|^2+r^2 \right)dx\\
			&\leq C\epsilon\int_{\{r_\epsilon\leq R_\eps\}}\frac{1}{r_\epsilon^2+\epsilon}dx+C(\alpha\sqrt{\epsilon})^4 \leq C(\alpha\sqrt{\epsilon})^4+C\epsilon\log\alpha\\&=o(\sqrt{\epsilon}) ,
		\end{align*} and
		
		\begin{align*}
			\int_{\{r_\epsilon\leq R_\eps\}}|D\phi_{\epsilon,\delta}|^2r^2dx&\leq C\int_{\{r_\epsilon\leq R_\eps\}}\left(\frac{r_\epsilon^2}{(\epsilon+r_\epsilon^2)^2}+1\right)r^2dx\\
			&\leq C \int_{\{r_\epsilon\leq R_\eps\}}\frac{r_\epsilon^2}{\epsilon+r_\epsilon^2}dx+ C(\alpha\sqrt{\epsilon})^4\\
			&\leq C(\alpha\sqrt{\epsilon})^2=o(\sqrt{\epsilon}),
		\end{align*}	
		we have
		\begin{align*}
			\int_{\{r_\epsilon\leq R_\eps\}}|\nabla_g\phi_{\epsilon,\delta}|^2dv_g= \int_{\{r_\epsilon\leq R_\eps\}}|D\phi_{\epsilon,\delta}|^2dv_g+ o(\sqrt{\epsilon})= \int_{\{r_\epsilon\leq R_\eps\}\cap M}|D\phi_{\epsilon,\delta}|^2dx+ o(\sqrt{\epsilon}).
		\end{align*}
		
		\
		
		\noindent\textbf{Step 1:} Compute $\int_{\{r_\epsilon\leq R_\eps\}\cap M}|D (2\log (\epsilon+r_\epsilon^2))|^2dx$.
		
		\
		
		It is easy to see that
		\begin{equation}
			\int_{\{r_\epsilon\leq R_\eps\}\cap M}|D (2\log (\epsilon+r_\epsilon^2))|^2dx=16\int_{\{x=(x_1,x_2),r_\epsilon\leq R_\eps,x_2\geq f(x_1)\}}\frac{r_\epsilon^2}{(\epsilon+r_\epsilon^2)^2}dx,
		\end{equation} where $x_2=f(x_1)=\frac{\kappa(p)}{2}x_1^2+O(x_1^3)$ is the local representation of $\partial M$.
		
		Setting
		$$
		y_1=\frac{x_1}{\sqrt{\epsilon}},
		\qquad
		y_2=\frac{x_2+a\sqrt{\epsilon}}{\sqrt{\epsilon}},
		$$ then
		$$
		\int_{\{r_\epsilon\leq R_\eps\}\cap M}|D (2\log (\epsilon+r_\epsilon^2))|^2dx
		=
		16\int_{\Omega}
		\frac{|y|^2}{(1+|y|^2)^2}\,dy,
		$$
		where $$\Omega:=\{y=(y_1,y_2)| \ y\in B_\alpha(0), y_2\geq a+\frac{f(y_1\sqrt{\epsilon})}{\sqrt{\epsilon}}=a+d_\epsilon(y_1)\}$$ and $$d_\epsilon(y_1)=\frac{\kappa(p)}{2}\sqrt{\epsilon}y_1^2+O(\epsilon y_1^3).$$

		Denote  $\Omega=B^+_\alpha(0)\setminus (\Omega_1\cup\Omega_2)$ where
		\begin{align*}
			\Omega_1&:=\{y=(y_1,y_2)\in\R^2\ |\ |y|\leq \alpha,\ 0\leq y_2\leq a\},\\
			\Omega_2&:=\{y=(y_1,y_2)\in\R^2\ |\ |y|\leq \alpha,\ a\leq y_2\leq a+d_\epsilon(y_1)\}.
		\end{align*}		
		Firstly, a direct computation yields
		\begin{align}\label{equation:09}
			16\int_{B^+_\alpha(0)}
			\frac{|y|^2}{(1+|y|^2)^2}\,dy&=8\pi\log (\alpha^2+1)+\frac{8\pi}{\alpha^2+1}-8\pi\notag\\
			&=16\pi\log \alpha+\frac{16\pi}{\alpha^2}-8\pi+O(\frac{1}{\alpha^4}).
		\end{align}		
		Secondly, noting that
		\begin{align*}
			32\int_0^a\int_0^{\sqrt{\alpha^2-y_2^2}}\frac{1}{1+|y|^2}\,dy&=32\int_0^a\frac{1}{\sqrt{1+y_2^2}} \arctan\frac{\sqrt{\alpha^2-y_2^2}}{\sqrt{1+y_2^2}}dy_2\\
			&=32\int_0^a\frac{1}{\sqrt{1+y_2^2}} \left(\frac{\pi}{2}- \frac{\sqrt{1+y_2^2}}{\sqrt{\alpha^2-y_2^2}}+O(\frac{1}{\alpha^3}) \right)dy_2\\
			&=16\pi\log (\sqrt{1+a^2}+a)-32\frac{a}{\alpha}+O(\frac{1}{\alpha^3})
		\end{align*}
		and
		\begin{align}\label{equation:07}
			&-32\int_0^a\int_0^{\sqrt{\alpha^2-y_2^2}}\frac{1}{(1+|y|^2)^2}\,dy\notag\\&=-16\int_0^a\frac{1}{(\sqrt{1+y_2^2})^3} \arctan\frac{\sqrt{\alpha^2-y_2^2}}{\sqrt{1+y_2^2}}dy_2-16\int_0^a\frac{1}{1+y_2^2} \frac{\sqrt{\alpha^2-y_2^2}}{\alpha^2+1}dy_2\notag\\
			&=-16\int_0^a\frac{1}{(\sqrt{1+y_2^2})^3} \left(\frac{\pi}{2}- \frac{\sqrt{1+y_2^2}}{\sqrt{\alpha^2-y_2^2}}+O(\frac{1}{\alpha^3}) \right) dy_2 -\frac{16\alpha}{\alpha^2+1}\int_0^a\frac{1}{1+y_2^2} \left( 1+O(\frac{1}{\alpha^2})\right)dy_2\notag\\
			&=-8\pi\int_0^a\frac{1}{(\sqrt{1+y_2^2})^3} dy_2+\frac{16}{\alpha}\int_0^a\frac{1}{1+y_2^2}(1+O(\frac{1}{\alpha^2})) dy_2 -\frac{16\alpha}{\alpha^2+1}\int_0^a\frac{1}{1+y_2^2} dy_2\notag\\
			&\quad +O(\frac{1}{\alpha^3})\notag\\
			&=-8\pi\frac{a}{\sqrt{1+a^2}}+O(\frac{1}{\alpha^3}),
		\end{align}
		we get
		\begin{align}\label{equation:10}
			16\int_{\Omega_1}
			\frac{|y|^2}{(1+|y|^2)^2}\,dy&=32\int_0^a\int_0^{\sqrt{\alpha^2-y_2^2}}\frac{|y|^2}{(1+|y|^2)^2}\,dy\notag\\
			&=16\pi\log (\sqrt{1+a^2}+a)-8\pi\frac{a}{\sqrt{1+a^2}}-32\frac{a}{\alpha}+O(\frac{1}{\alpha^3}).
		\end{align}
		
		Thirdly, let $t_\epsilon^-<0<t_\epsilon^+$ and $(t_\epsilon^-,a+d_\epsilon(t_\epsilon^-))$, $(t_\epsilon^+,a+d_\epsilon(t_\epsilon^+))$ are the two intersection points of the curve $y_2=\sqrt{\alpha^2-y_1^2}$ and the curve $y_2=a+d_\epsilon(y_1)$. Then we have $$(t_\epsilon^\pm)^2+(a+d_\epsilon(t_\epsilon^\pm))^2=\alpha^2$$ and
		\begin{align*}
			\sqrt{\alpha^2-a^2}-|t_\epsilon^\pm|=\frac{2ad_\epsilon(t_\epsilon^\pm)+ (d_\epsilon(t_\epsilon^\pm))^2}{\sqrt{\alpha^2-a^2}+|t_\epsilon^\pm|} \leq C\frac{d_\epsilon(t_\epsilon^\pm)}{|t_\epsilon^\pm|}\leq C\alpha\sqrt{\epsilon},
		\end{align*} which implies $\frac{1}{2}\alpha<|t_\epsilon^\pm|<\alpha$.
		
		Thus,
		\begin{align*}
			\int_{t_\epsilon^+}^{\sqrt{\alpha^2-a^2}} \int_a^{a+d_\epsilon(y_1)}\frac{|y|^2}{(1+|y|^2)^2}\,dy&\leq  \int_{t_\epsilon^+}^{\sqrt{\alpha^2-a^2}} \int_a^{a+d_\epsilon(\sqrt{\alpha^2-a^2})}\frac{1}{|y|^2}\,dy\\
			&\leq \frac{C}{(t_\epsilon^+)^2}\alpha\sqrt{\epsilon}\alpha^2\sqrt{\epsilon}\leq C\alpha\epsilon=o(\sqrt{\epsilon}).
		\end{align*}
		Similarly, $$\int^{t_\epsilon^-}_{-\sqrt{\alpha^2-a^2}} \int_a^{a+d_\epsilon(y_1)}\frac{|y|^2}{(1+|y|^2)^2}\,dy=o(\sqrt{\epsilon}).$$
		Therefore,
		\begin{align*}
			16\int_{\Omega_2}
			\frac{|y|^2}{(1+|y|^2)^2}\,dy&=16\int_{-\sqrt{\alpha^2-a^2}}^{\sqrt{\alpha^2-a^2}} \int_a^{a+d_\epsilon(y_1)}\frac{|y|^2}{(1+|y|^2)^2}\,dy+ o(\sqrt{\epsilon}).
		\end{align*}
		
		Computing directly, we obtain
		\begin{align*}
			&\int_{-\sqrt{\alpha^2-a^2}}^{\sqrt{\alpha^2-a^2}} \int_a^{a+d_\epsilon(y_1)}\left|\frac{|y|^2}{(1+|y|^2)^2}-\frac{y_1^2+a^2}{(1+y_1^2+a^2)^2}\right| \,dy \\ &\leq C\int_{-\sqrt{\alpha^2-a^2}}^{\sqrt{\alpha^2-a^2}} \int_a^{a+d_\epsilon(y_1)}\frac{(y_1^2+a^2)(1+|y|^2)+(1+|y|^2)^2}{(1+y_1^2+a^2)^2(1+|y|^2)^2}d_\epsilon(y_1)\,dy\\
			&\leq C\int_{-\sqrt{\alpha^2-a^2}}^{\sqrt{\alpha^2-a^2}} \frac{1}{(1+y_1^2+a^2)^2}d^2_\epsilon(y_1)\,dy_1\leq C\alpha\epsilon=o(\sqrt{\epsilon})
		\end{align*}
		and
		\begin{align*}
			&16\int_{-\sqrt{\alpha^2-a^2}}^{\sqrt{\alpha^2-a^2}} \int_a^{a+d_\epsilon(y_1)}\frac{y_1^2+a^2}{(1+y_1^2+a^2)^2} \,dy \\ &=16\int_{-\sqrt{\alpha^2-a^2}}^{\sqrt{\alpha^2-a^2}}\frac{y_1^2+a^2}{(1+y_1^2+a^2)^2} \left(\frac{\kappa(p)}{2}\sqrt{\epsilon}y_1^2+O(\epsilon |y_1|^3)\right)\,dy_1\\
			&=16\kappa(p)\sqrt{\epsilon}\left(\sqrt{\alpha^2-a^2}+\frac{\sqrt{\alpha^2-a^2}}{2(\alpha^2+1)}- \frac{2a^2+3}{2\sqrt{a^2+1}}\arctan\frac{\sqrt{\alpha^2-a^2}}{\sqrt{1+a^2}}\right)+O(\alpha^2\epsilon)\\
			&=16\kappa(p)\alpha\sqrt{\epsilon}-4\pi \kappa(p)\frac{2a^2+3}{\sqrt{a^2+1}}\sqrt{\epsilon}+o(\sqrt{\epsilon}).
		\end{align*}
		
		Then
		\begin{align}\label{equation:11}
			16\int_{\Omega_2}
			\frac{|y|^2}{(1+|y|^2)^2}\,dy&=16\kappa(p)\alpha\sqrt{\epsilon}-4\pi \kappa(p)\frac{2a^2+3}{\sqrt{a^2+1}}\sqrt{\epsilon}+o(\sqrt{\epsilon}).
		\end{align}
		
		Combining \eqref{equation:09}-\eqref{equation:11}, we get
		\begin{align}
			\frac{1}{2}\int_{\{r_\epsilon\leq R_\eps\}\cap M}|D (2\log (\epsilon+r_\epsilon^2))|^2dx
			&=
			8\pi\log\alpha-4\pi
			-8\pi\log\bigl(a+\sqrt{1+a^2}\bigr)
			+\frac{4\pi a}{\sqrt{1+a^2}}
			+\frac{16a}{\alpha}
			+\frac{8\pi}{\alpha^2} \notag\\
			&\quad
			-8\kappa(p)\alpha\sqrt{\epsilon}
			+\frac{2\pi\kappa(p)(2a^2+3)}
			{\sqrt{1+a^2}}\sqrt{\epsilon}
			+o(\sqrt{\epsilon}).
			\label{eq:bubble-radial-complete}
		\end{align}

		\
		
		\noindent\textbf{Step 2:} Compute $\int_{\{r_\epsilon\leq R_\eps\}\cap M}D (-2\log (\epsilon+r_\epsilon^2))D(b_1x_1+b_2x_2)dx$.
		
		\
		
		It is easy to see that
		\begin{align*}
			&\int_{\{r_\epsilon\leq R_\eps\}\cap M}D (-2\log (\epsilon+r_\epsilon^2))D(b_1x_1+b_2x_2)dx\\
			&=\int_{\{x=(x_1,x_2),r_\epsilon\leq R_\eps,x_2\geq f(x_1)\}}\left(-\frac{4b_1 x_1}{\epsilon+r_\epsilon^2}- \frac{4b_2 (x_2+a\sqrt{\epsilon})}{\epsilon+r_\epsilon^2}\right)dx\\
			&=\sqrt{\epsilon}\int_\Omega \left(-\frac{4b_1 y_1}{1+|y|^2}- \frac{4b_2 y_2}{1+|y|^2}\right)dy.
		\end{align*}

		On one hand, by symmetry, we have
		\begin{align*}
			-4b_1\sqrt{\epsilon}\int_\Omega \frac{ y_1}{1+|y|^2}dy&= 4b_1\sqrt{\epsilon}\int_{\Omega_2} \frac{ y_1}{1+|y|^2}dy\\
			&=4b_1\sqrt{\epsilon}\int_{-\sqrt{\alpha^2-a^2}}^{\sqrt{\alpha^2-a^2}} \int_a^{a+d_\epsilon(y_1)}\frac{ y_1}{1+|y|^2}\,dy+ o(\sqrt{\epsilon})\\
			&=4b_1\sqrt{\epsilon}\int_{-\sqrt{\alpha^2-a^2}}^{\sqrt{\alpha^2-a^2}} \int_a^{a+d_\epsilon(y_1)}\frac{ y_1}{1+|y_1|^2+a^2}\,dy+ o(\sqrt{\epsilon})\\
			&=4b_1\sqrt{\epsilon}\int_{-\sqrt{\alpha^2-a^2}}^{\sqrt{\alpha^2-a^2}}\frac{ y_1}{1+|y_1|^2+a^2}\left(\frac{\kappa(p)}{2}\sqrt{\epsilon}y_1^2+O(\epsilon |y_1|^3)\right) \,dy_1+ o(\sqrt{\epsilon})\\
			&=o(\sqrt{\epsilon}).
		\end{align*}
		
		On the other hand, noting that
		\begin{align*}
			\int_{\Omega\cup\Omega_2} \frac{ y_2}{1+|y|^2}dy&= \int_{-\sqrt{\alpha^2-a^2}}^{\sqrt{\alpha^2-a^2}}\int_a^{\sqrt{\alpha^2-y_1^2}} \frac{ y_2}{1+|y|^2}dy\\
			&=2\sqrt{\alpha^2-a^2}-2\sqrt{1+a^2}\arctan\sqrt{\frac{\alpha^2-a^2}{1+a^2}}\\
			&=2\alpha(1-\frac{1}{2}\frac{a^2}{\alpha^2}+O(\alpha^{-4})) -2\sqrt{1+a^2}\left(\frac{\pi}{2}-\sqrt{\frac{1+a^2}{\alpha^2-a^2}}+O(\alpha^{-3})\right)\\
			&=2\alpha-\pi\sqrt{1+a^2}+\frac{2+a^2}{\alpha}+O(\alpha^{-3})
		\end{align*}
		and
		\begin{align*}
			\int_{\Omega_2} \frac{ y_2}{1+|y|^2}dy&= \int_{-\sqrt{\alpha^2-a^2}}^{\sqrt{\alpha^2-a^2}}\int_a^{a+d_\epsilon(y_1)} \frac{ y_2}{1+|y|^2}dy+o(\sqrt{\epsilon})\\
			&= \int_{-\sqrt{\alpha^2-a^2}}^{\sqrt{\alpha^2-a^2}}\int_a^{a+d_\epsilon(y_1)} \frac{ a}{1+a^2+y_1^2}dy+o(\sqrt{\epsilon})+O(\alpha^3\epsilon)=O(\alpha\sqrt\epsilon),
		\end{align*}
		we have
		\begin{align*}
			-4b_2\sqrt{\epsilon}\int_\Omega \frac{ y_2}{1+|y|^2}dy= -8b_2\alpha\sqrt{\epsilon}+4\pi b_2\sqrt{1+a^2}\sqrt{\epsilon}+o(\sqrt{\epsilon}).
		\end{align*}		
		Then
		\begin{align*}
			\int_{\{r_\epsilon\leq R_\eps\}\cap M}D (-2\log (\epsilon+r_\epsilon^2))D(b_1x_1+b_2x_2)dx= -8b_2\alpha\sqrt{\epsilon}+4\pi b_2\sqrt{1+a^2}\sqrt{\epsilon}+o(\sqrt{\epsilon}).
		\end{align*}

		Finally, noting that
		\begin{equation}\label{eq:bubble-constant-gradient}
			\frac12\bigl(b_1(p)^2+b_2(p)^2\bigr)
			\int_{\{r_\epsilon\leq R_\eps\}\cap M}dx
			=
			O(\alpha^2\epsilon)
			=
			o(\sqrt{\epsilon}).
		\end{equation}
		we obtain
		\begin{align*}
			\frac12\int_{M\cap\{r_\epsilon\leq R_\eps\}}
			|\nabla_g\phi_{\epsilon,\delta}|^2\,dv_g
			=&8\pi\log\alpha-4\pi
			-8\pi\log\bigl(a+\sqrt{1+a^2}\bigr)
			+\frac{4\pi a}{\sqrt{1+a^2}}\\
			&+\frac{16a}{\alpha}+\frac{8\pi}{\alpha^2}
			-8\bigl(b_2(p)+\kappa(p)\bigr)R_\eps\\
			&+\left(
			4\pi\sqrt{1+a^2}\,b_2(p)
			+\frac{2\pi\kappa(p)(2a^2+3)}{\sqrt{1+a^2}}
			\right)\sqrt{\epsilon}
			+o(\sqrt{\epsilon}).
		\end{align*}
		This proves \eqref{eq:bubble-gradient}.
	\end{proof}
	
	\
	
	\begin{lem}\label{lem:neck-gradient}
		As $\epsilon\to0$, we have
		\begin{align}
			\frac12\int_{M\cap\{R_\eps<r_\epsilon<2R_\eps\}}
			|\nabla_g\phi_{\epsilon,\delta}|^2\,dv_g
			=8\pi\log2-\frac{8a}{\alpha}-\frac{6\pi}{\alpha^2}-8\bigl(b_2(p)+\kappa(p)\bigr)R_\eps
			+o(\sqrt{\epsilon}).
			\label{eq:neck-gradient}
		\end{align}
	\end{lem}
	\begin{proof}
		In the neck region, we write
		\begin{align*}
			\phi_{\epsilon,\delta}
			={}&\log\epsilon-2\log\bigl(\epsilon+r_\epsilon^2\bigr)
			+A(p)+b_1(p)x_1+b_2(p)x_2\\
			&+(1-\chi_\epsilon)
			\bigl[G(x,p)+4\log r-A(p)-b_1(p)x_1-b_2(p)x_2\bigr].
		\end{align*}
		By the expansion of $G(x,p)$,
		$$
		G(x,p)+4\log r-A(p)-b_1(p)x_1-b_2(p)x_2=O(r^2),
		$$
		it follows that
		$$
		D\phi_{\epsilon,\delta}
		=
		D\left(-2\log(\epsilon+r_\epsilon^2)
		+A(p)+b_1(p)x_1+b_2(p)x_2\right)
		+O(R_\eps).
		$$
		Moreover,
		$$
		\left|
		D\left(-2\log(\epsilon+r_\epsilon^2)
		+A(p)+b_1(p)x_1+b_2(p)x_2\right)
		\right|
		\leq\frac{C}{R_\eps}.
		$$
		We obtain
		\begin{align}
			&\frac12\int_{M\cap\{R_\eps<r_\epsilon<2R_\eps\}}
			|\nabla_g\phi_{\epsilon,\delta}|^2\,dv_g \notag\\
			&=\frac12\int_{M\cap\{R_\eps<r_\epsilon<2R_\eps\}}
			|D\phi_{\epsilon,\delta}|^2\,dx +o(\sqrt{\epsilon})\notag\\
			&=
			\frac12\int_{M\cap\{R_\eps<r_\epsilon<2R_\eps\}}
			\left|
			D\left(-2\log(\epsilon+r_\epsilon^2)
			+A(p)+b_1(p)x_1+b_2(p)x_2\right)
			\right|^2\,dx
			+o(\sqrt{\epsilon}).
			\label{eq:neck-energy-reduction}
		\end{align}

		Applying
		Lemma \ref{lem:bubble-gradient} to the regions
		$r_\epsilon\leq2R_\eps$ and $r_\epsilon\leq R_\eps$, and
		subtracting the resulting expansions, we obtain \eqref{eq:neck-gradient}.
	\end{proof}

	\

	\begin{lem}\label{lem:exterior-gradient}
		For every fixed  small constant $\delta>0$, as $\epsilon\to0$, we have
		\begin{align}
			\frac12\int_{M\cap\{r_\epsilon\geq2R_\eps\}}
			|\nabla_g\phi_{\epsilon,\delta}|^2\,dv_g
			={}&-8\pi\log(2R_\eps)+2\pi A(p)
			-\frac{8a}{\alpha}-\frac{2\pi}{\alpha^2} \notag\\
			&+16\bigl(b_2(p)+\kappa(p)\bigr)R_\eps
			+O(\delta\sqrt{\epsilon})
			+o(\sqrt{\epsilon}).
			\label{eq:exterior-gradient}
		\end{align}
		
	\end{lem}
	\begin{proof}
		
		\noindent\textbf{Step 1:} Compute $\frac12\int_{M\cap\{r\geq2R_\eps\}}|\nabla_gG|^2\,dv_g$.
		
		\

		Integrating by parts, we get
		\begin{align*}
			\frac12\int_{M\cap\{r\geq2R_\eps\}}|\nabla_gG|^2\,dv_g &=-\frac12\int_{M\cap\{r\geq2R_\eps\}}\Delta_gG\cdot G\,dv_g +\frac12\int_{M\cap\{r=2R_\eps\}}G(-\partial_rG)\,ds_g\\
			&=\frac{2\pi}{|M|}\int_{M\cap\{r<2R_\eps\}}G\,dv_g
			+\frac12\int_{M\cap\{r=2R_\eps\}}G(-\partial_rG)\,ds_g.
		\end{align*}
		Since the divergence theorem yields
		$$
		\int_{M\cap\{r=2R_\eps\}}(-\partial_rG)\,ds_g
		=
		4\pi-\frac{4\pi}{|M|}
		\operatorname{Vol}_g\bigl(M\cap\{r<2R_\eps\}\bigr),
		$$
		then
		\begin{align}
			\frac12\int_{M\cap\{r\geq2R_\eps\}}|\nabla_gG|^2\,dv_g
			={}&-8\pi\log(2R_\eps)+2\pi A(p) \notag\\
			&+\frac12\int_{M\cap\{r=2R_\eps\}}
			\bigl[G+4\log(2R_\eps)-A(p)\bigr](-\partial_rG)\,ds_g \notag\\
			&+\frac{2\pi}{|M|}\int_{M\cap\{r<2R_\eps\}}
			\bigl[G+4\log(2R_\eps)-A(p)\bigr]\,dv_g.
			\label{eq:green-energy-exterior-reduction}
		\end{align}
		
		For $r<2R_\eps$,
		$$
		G+4\log(2R_\eps)-A(p)
		=
		-4\log\frac{r}{2R_\eps}+b_1(p)x_1+b_2(p)x_2+O(r^2).
		$$
		Since
		$$
		\int_0^{2R_\eps}\left|\log\frac{r}{2R_\eps}\right|r\,dr=O(R_\eps^2),
		$$
		we have
		$$
		\int_{M\cap\{r<2R_\eps\}}
		\bigl[G+4\log(2R_\eps)-A(p)\bigr]\,dv_g
		=
		O(R_\eps^2)=o(\sqrt{\epsilon}).
		$$
		
		On $r=2R_\eps$, we have
		$$
		G+4\log(2R_\eps)-A(p)
		=
		2R_\eps\bigl(b_1(p)\cos\theta+b_2(p)\sin\theta\bigr)+O(R_\eps^2),
		$$
		$$
		-\partial_rG
		=
		\frac{2}{R_\eps}-b_1(p)\cos\theta-b_2(p)\sin\theta+O(R_\eps),
		\quad
		ds_g=2R_\eps\bigl(1+O(R_\eps^2)\bigr)\,d\theta.
		$$
		Let $\theta_+$ and $\theta_-$ be defined so that the two endpoints of the circular arc $r=2R_\epsilon$ correspond to the polar angles $\theta_+$ and $\pi-\theta_-$, respectively. Then by the proof of Lemma \ref{lem:04}, we have
		$$
		\theta_+=\kappa(p)R_\eps+O(R_\eps^2),
		\qquad
		\theta_-=\kappa(p)R_\eps+O(R_\eps^2),
		$$
		and
		\begin{align*}
			\int_{\theta_+}^{\pi-\theta_-}\cos\theta\,d\theta
			=\sin\theta_- -\sin\theta_+=O(R_\eps^2),\
			\int_{\theta_+}^{\pi-\theta_-}\sin\theta\,d\theta
			=\cos\theta_++\cos\theta_-=2+O(R_\eps^2).
		\end{align*}
		Therefore,
		\begin{align*}
			&\frac12\int_{M\cap\{r=2R_\eps\}}
			\left(G+4\log(2R_\eps)-A(p)\right)(-\partial_rG)\,ds_g\\
			&=
			4R_\eps\int_{\theta_+}^{\pi-\theta_-}
			\bigl(b_1(p)\cos\theta+b_2(p)\sin\theta\bigr)\,d\theta
			+O(R_\eps^2)=
			8b_2(p)R_\eps+O(R_\eps^2).
		\end{align*}
		Substituting these estimates into
		\eqref{eq:green-energy-exterior-reduction}, and using
		$b_2(p)=-2\kappa(p)$ and $R_\eps^2=o(\sqrt{\epsilon})$, we obtain
		\begin{equation}\label{eq:complete-green-energy-exterior-final}
			\frac12\int_{M\cap\{r\geq2R_\eps\}}|\nabla_gG|^2\,dv_g
			=
			-8\pi\log(2R_\eps)+2\pi A(p)
			+16\bigl(b_2(p)+\kappa(p)\bigr)R_\eps
			+o(\sqrt{\epsilon}).
		\end{equation}

		\
		
		\noindent\textbf{Step 2:} Compute $\frac12\int_{M\cap\{r_\epsilon\geq2R_\eps,\ r\leq\delta\}}
		|\nabla_g\phi_{\epsilon,\delta}|^2\,dv_g-\frac12\int_{M\cap\{2R_\eps\leq r\leq\delta\}}
		|\nabla_gG|^2\,dv_g $.
		
		\

		In $M\cap\{r_\epsilon\geq2R_\eps,\ r\leq\delta\}$, we have
		\begin{align*}
			\nabla_g\phi_{\epsilon,\delta}
			=&
			\nabla_g\bigl[-2\log(\epsilon+r_\epsilon^2)\bigr]
			+
			\nabla_g\bigl(G(x,p)+4\log r\bigr),\\
			\nabla_gG
			=&
			\nabla_g(-4\log r)
			+
			\nabla_g\bigl(G(x,p)+4\log r\bigr).
		\end{align*}
		Then
		\begin{align*}
			&\frac12\int_{M\cap\{r_\epsilon\geq2R_\eps,\ r\leq\delta\}}
			|\nabla_g\phi_{\epsilon,\delta}|^2\,dv_g
			-
			\frac12\int_{M\cap\{2R_\eps\leq r\leq\delta\}}
			|\nabla_gG|^2\,dv_g \notag\\
			&=
			\frac12\int_{M\cap\{r_\epsilon\geq2R_\eps,\ r\leq\delta\}}
			\left|\nabla_g\bigl[-2\log(\epsilon+r_\epsilon^2)\bigr]\right|^2\,dv_g
			-
			\frac12\int_{M\cap\{2R_\eps\leq r\leq\delta\}}
			|\nabla_g(-4\log r)|^2\,dv_g \notag\\
			&\quad+
			\int_{M\cap\{r_\epsilon\geq2R_\eps,\ r\leq\delta\}}
			\left\langle
			\nabla_g\bigl[-2\log(\epsilon+r_\epsilon^2)\bigr],
			\nabla_g\bigl(G(x,p)+4\log r\bigr)
			\right\rangle_g\,dv_g \notag\\
			&\quad-
			\int_{M\cap\{2R_\eps\leq r\leq\delta\}}
			\left\langle
			\nabla_g(-4\log r),
			\nabla_g\bigl(G(x,p)+4\log r\bigr)
			\right\rangle_g\,dv_g \notag\\
			&\quad+
			\frac12\int_{M\cap\{r_\epsilon\geq2R_\eps,\ r\leq\delta\}}
			\left|\nabla_g\bigl(G(x,p)+4\log r\bigr)\right|^2\,dv_g -
			\frac12\int_{M\cap\{2R_\eps\leq r\leq\delta\}}
			\left|\nabla_g\bigl(G(x,p)+4\log r\bigr)\right|^2\,dv_g\\
			&:=\mathbf{I_1}+\mathbf{I_2}+\mathbf{I_3}.
		\end{align*}
		
		For $\mathbf{I_1}$, noting that
		\begin{align*}
			\int_{\{r_\epsilon\leq \delta\}}|g^{-2}-r^{-2}||\frac{\partial}{\partial\theta}(-2\log(\epsilon+r_\epsilon^2))|^2 dv_g&= \int_{\{r_\epsilon\leq \delta\}}\left|\frac{\Pi}{3}+O(r)\right|\left| \frac{4a\sqrt{\epsilon}r\cos\theta}{r_\epsilon^2+\epsilon}\right|^2 dv_g\\&\leq C\int_{\{r_\epsilon\leq \delta\}}\left(\left| \frac{4a\sqrt{\epsilon}r\cos\theta}{r_\epsilon^2+\epsilon}\right|^2 \right)dx\\
			&\leq C\epsilon\int_{\{r_\epsilon\leq \delta\}}\frac{1}{r_\epsilon^2+\epsilon}dx\leq C\epsilon\log\alpha=o(\sqrt{\epsilon})
		\end{align*} and
		$$
		\left|D\bigl[-2\log(\epsilon+r_\epsilon^2)\bigr]\right|^2
		-|D(-4\log r)|^2=O(\frac{\sqrt{\epsilon}x_2+\epsilon}{r^4}),\ for\  r_\epsilon\geq2R_\eps,\ 2R_\eps\leq  r\leq\delta,
		$$
		we have
		\begin{align*}
			\mathbf{I_1}
			&=\frac12\int_{M\cap\{r_\epsilon\geq2R_\eps,\ r\leq\delta\}}
			\left|D\bigl[-2\log(\epsilon+r_\epsilon^2)\bigr]\right|^2\,dv_g
			-\frac12\int_{M\cap\{2R_\eps\leq r\leq\delta\}}
			|D(-4\log r)|^2\,dv_g +o(\sqrt{\epsilon})\notag\\
			&=
			\frac12\int_{M\cap\{r_\epsilon\geq2R_\eps,\ 2R_\eps\leq  r\leq\delta\}}
			\left(
			\left|D\bigl[-2\log(\epsilon+r_\epsilon^2)\bigr]\right|^2
			-|D(-4\log r)|^2
			\right)dv_g +o(\sqrt{\epsilon})\notag\\
			&\quad+
			\frac12\int_{M\cap\{r_\epsilon\geq2R_\eps,\ r<2R_\eps\}}
			\left|D\bigl[-2\log(\epsilon+r_\epsilon^2)\bigr]\right|^2\,dv_g -
			\frac12\int_{M\cap\{r\geq2R_\eps,\ r_\epsilon<2R_\eps\}}
			|D(-4\log r)|^2\,dv_g\\
			&=
			\frac12\int_{M\cap\{r_\epsilon\geq2R_\eps,\ 2R_\eps\leq  r\leq\delta\}}
			\left(
			\left|D\bigl[-2\log(\epsilon+r_\epsilon^2)\bigr]\right|^2
			-|D(-4\log r)|^2
			\right)dx +o(\sqrt{\epsilon})+O(\delta\sqrt{\epsilon})\\
			&\quad+
			\frac12\int_{M\cap\{r_\epsilon\geq2R_\eps,\ r<2R_\eps\}}
			\left|D\bigl[-2\log(\epsilon+r_\epsilon^2)\bigr]\right|^2\,dx -
			\frac12\int_{M\cap\{r\geq2R_\eps,\ r_\epsilon<2R_\eps\}}
			|D(-4\log r)|^2\,dx.
		\end{align*}

		Since
		\begin{align*}
			\{r_\epsilon\geq2R_\eps,\ r<2R_\eps, \ 0\leq x_2\leq f(x_1)\}\subset \{\sqrt{4R_\epsilon^2-(x_2+a\sqrt{\epsilon})^2}\leq |x_1|\leq 2R_\eps, \ 0\leq x_2\leq f(x_1)\}
		\end{align*} and $$2R_\eps-\sqrt{4R_\epsilon^2-(x_2+a\sqrt{\epsilon})^2}=O(\frac{\epsilon}{R_\epsilon})=o(\sqrt{\epsilon}),$$
		we get
		\begin{align*}
			&\int_{\{r_\epsilon\geq2R_\eps,\ r<2R_\eps, \ 0\leq x_2\leq f(x_1)\}}
			\left|D\bigl[-2\log(\epsilon+r_\epsilon^2)\bigr]\right|^2\,dx=o(\sqrt{\epsilon}).
		\end{align*}
		Similarly,
		\begin{align*}
			\int_{\{r\geq2R_\eps,\ r_\epsilon<2R_\eps, \ 0\leq x_2\leq f(x_1)\}}
			|D(-4\log r)|^2\,dx=o(\sqrt{\epsilon}).
		\end{align*}
		
		Noting that
		\begin{align*}
			&\int_{\{r_\epsilon\geq2R_\eps,\ 2R_\eps\leq  r\leq\delta,\ 0\leq x_2\leq f(x_1)\}}
			\left|
			\left|D\bigl[-2\log(\epsilon+r_\epsilon^2)\bigr]\right|^2
			-|D(-4\log r)|^2
			\right|dx\\& \leq C\int_{R_\epsilon}^\delta\int_0^{f(x_1)}\frac{\sqrt{\epsilon}x_2+\epsilon}{r^4}dx_2dx_1=O(\delta\sqrt{\epsilon})+o(\sqrt{\epsilon}),
		\end{align*}
		we have
		\begin{align*}
			\mathbf{I_1}
			&=\frac12\int_{\{r_\epsilon\geq2R_\eps,\ r\leq\delta,\ x_2\geq 0\}}
			\left|D\bigl[-2\log(\epsilon+r_\epsilon^2)\bigr]\right|^2\,dx
			-\frac12\int_{\{2R_\eps\leq r\leq\delta, \ x_2\geq 0\}}
			|D(-4\log r)|^2\,dx \\&\quad +o(\sqrt{\epsilon})+O(\delta\sqrt{\epsilon})\\
			&=8\int_{r_\epsilon\geq2R_\eps,\ r\leq\delta,\ x_2\geq 0}
			\frac{r_\epsilon^2}{(\epsilon+r_\epsilon^2)^2}\,dx
			-
			8\int_{2R_\eps\leq r\leq\delta,\ x_2\geq 0}
			\frac1{r^2}\,dx +
			O(\delta\sqrt{\epsilon})
			+o(\sqrt{\epsilon}).
		\end{align*}
		
		On one hand, by \eqref{equation:09} and \eqref{equation:10}, we have
		\begin{align*}
			&8\int_{r_\epsilon\leq \alpha\sqrt{\epsilon},\ x_2\geq 0}
			\frac{r_\epsilon^2}{(\epsilon+r_\epsilon^2)^2}\,dx\\ &=8\int_{y\in B^+_\alpha(0),\ y_2\geq a}
			\frac{|y|^2}{(1+|y|^2)^2}\,dy \\
			&=8\pi\log \alpha+\frac{8\pi}{\alpha^2}+16\frac{a}{\alpha}-4\pi-8\pi\log (\sqrt{1+a^2}+a)+4\pi\frac{a}{\sqrt{1+a^2}}+O(\frac{1}{\alpha^3}),
		\end{align*} which implies
		\begin{align*}
			8\int_{2\alpha\sqrt{\epsilon}\leq r_\epsilon\leq\delta,\ x_2\geq 0}
			\frac{r_\epsilon^2}{(\epsilon+r_\epsilon^2)^2}\,dx&=\left(8\pi\log s+\frac{8\pi}{s^2}+16\frac{a}{s}\right)\bigg|_{s=2\alpha}^{s=\frac{\delta}{\sqrt{\epsilon}}}+O(\frac{1}{\alpha^3})\\
			&=8\pi\log \frac{\delta}{2R_\epsilon}+8\pi(\frac{\epsilon}{\delta^2}-\frac{1}{4\alpha^2})+16a(\frac{\sqrt{\epsilon}}{\delta}-\frac{1}{2\alpha}) +O(\frac{1}{\alpha^3})\\
			&=8\pi\log \frac{\delta}{2R_\epsilon}-\frac{2\pi}{\alpha^2}+16a\frac{\sqrt{\epsilon}}{\delta}-\frac{8a}{\alpha} +o(\sqrt{\epsilon}).
		\end{align*}
		
		On the other hand, since
		\begin{align*}
			8\int_{ r_\epsilon\geq \delta,\ r\leq\delta ,\ x_2\geq 0}
			\frac{r_\epsilon^2}{(\epsilon+r_\epsilon^2)^2}\,dx&= 8\int_{r_\epsilon\geq \delta,\ r\leq\delta ,\ x_2\geq 0}
			\frac{1}{r^2}\,dx+O(\frac{\sqrt{\epsilon}\delta+\epsilon}{\delta^4})\int_{r_\epsilon\geq \delta,\ r\leq\delta ,\ x_2\geq 0}
			\,dx\\
			&= 8\int_{|y|\geq \frac{\delta}{\sqrt{\epsilon}},\ y_1^2+(y_2-a)^2\leq\frac{\delta^2}{\epsilon} ,\ y_2\geq a}
			\frac{1}{y_1^2+(y_2-a)^2}\,dy+O(\frac{\sqrt{\epsilon}\delta+\epsilon}{\delta^4})\cdot O(\delta\sqrt{\epsilon})\\
			&=8\int_0^\pi d\theta\int_{-a\sin\theta+\sqrt{\frac{\delta^2}{\epsilon}-a^2\cos^2\theta}}^{\frac{\delta}{\sqrt{\epsilon}}}\frac{1}{t}dt+ o(\sqrt{\epsilon})\\
			&=16a\frac{\sqrt{\epsilon}}{\delta}+ o(\sqrt{\epsilon}),
		\end{align*}
		where $y_1=t\cos\theta$, $y_2=a+t\sin\theta$, then
		\begin{align*}
			8\int_{r_\epsilon\geq2R_\eps,\ r\leq\delta,\ x_2\geq 0}
			\frac{r_\epsilon^2}{(\epsilon+r_\epsilon^2)^2}\,dx= 8\pi\log \frac{\delta}{2R_\epsilon}-\frac{2\pi}{\alpha^2}+32a\frac{\sqrt{\epsilon}}{\delta}-\frac{8a}{\alpha} +o(\sqrt{\epsilon}).
		\end{align*}
		
		Noting the fact that $$8\int_{2R_\eps\leq r\leq\delta,\ x_2\geq 0}
		\frac1{r^2}\,dx= 8\pi\log \frac{\delta}{2R_\epsilon},$$
		thus,
		\begin{align*}
			\mathbf{I_1}=
			-\frac{8a}{\alpha}-\frac{2\pi}{\alpha^2}+32a\frac{\sqrt{\epsilon}}{\delta} +
			O(\delta\sqrt{\epsilon})
			+o(\sqrt{\epsilon}).
		\end{align*}
		
		\
		
		For $\mathbf{I_2}$, noting that
		\begin{align*}
			&\int_{M\cap\{2R_\eps\leq r\leq\delta\}}
			\left\langle
			\nabla_g(-4\log r),
			\nabla_g\bigl(G(x,p)+4\log r\bigr)
			\right\rangle_g\,dv_g\\&=\int_{M\cap\{2R_\eps\leq r\leq\delta\}}
			D(-4\log r)D\bigl(G(x,p)+4\log r\bigr)
			\,dv_g
		\end{align*}
		and
		\begin{align*}
			&\int_{M\cap\{r_\epsilon\geq2R_\eps,\ r\leq\delta\}}
			\left\langle
			\nabla_g\bigl[-2\log(\epsilon+r_\epsilon^2)\bigr],
			\nabla_g\bigl(G(x,p)+4\log r\bigr)
			\right\rangle_g\,dv_g \\
			&=\int_{M\cap\{r_\epsilon\geq2R_\eps,\ r\leq\delta\}}
			D\bigl[-2\log(\epsilon+r_\epsilon^2)\bigr]D\bigl(G(x,p)+4\log r\bigr)
			\,dv_g\\&\quad +\int_{M\cap\{r_\epsilon\geq2R_\eps,\ r\leq\delta\}}
			(g^{-2}-r^{-2})\frac{\partial}{\partial\theta}\bigl[-2\log(\epsilon+r_\epsilon^2)\bigr] \cdot\frac{\partial}{\partial\theta}\bigl(G(x,p)+4\log r\bigr)
			\,dv_g\\
			&=\int_{M\cap\{r_\epsilon\geq2R_\eps,\ r\leq\delta\}}
			D\bigl[-2\log(\epsilon+r_\epsilon^2)\bigr]D\bigl(G(x,p)+4\log r\bigr)
			\,dv_g+O(\delta\sqrt{\epsilon}),
		\end{align*}
		we have
		\begin{align*}
			\mathbf{I_2}&=\int_{M\cap\{r_\epsilon\geq2R_\eps,\ r\leq\delta\}}
			D\bigl[-2\log(\epsilon+r_\epsilon^2)\bigr]D\bigl(G(x,p)+4\log r\bigr)
			\,dv_g\\&\quad -\int_{M\cap\{2R_\eps\leq r\leq\delta\}}
			D(-4\log r)D\bigl(G(x,p)+4\log r\bigr)
			\,dv_g+O(\delta\sqrt{\epsilon}).
		\end{align*}
		Following the same argument as in the treatment of $\mathbf{I_1}$, we get
		\begin{align*}
			\mathbf{I_2}&=\int_{M\cap\{r_\epsilon\geq2R_\eps,\ r\leq\delta\}}
			D\bigl[-2\log(\epsilon+r_\epsilon^2)\bigr]D\bigl(G(x,p)+4\log r\bigr)
			\,dx\\&\quad -\int_{M\cap\{2R_\eps\leq r\leq\delta\}}
			D(-4\log r)D\bigl(G(x,p)+4\log r\bigr)
			\,dx+O(\delta\sqrt{\epsilon})+o(\sqrt{\epsilon})\\
			&=\int_{\{r_\epsilon\geq2R_\eps,\ r\leq\delta,\ x_2\geq 0\}}
			D\bigl[-2\log(\epsilon+r_\epsilon^2)\bigr]D\bigl(G(x,p)+4\log r\bigr)
			\,dx\\&\quad -\int_{\{2R_\eps\leq r\leq\delta,\ x_2\geq 0\}}
			D(-4\log r)D\bigl(G(x,p)+4\log r\bigr)
			\,dx+O(\delta\sqrt{\epsilon})+o(\sqrt{\epsilon})\\
			&=\int_{\{r_\epsilon\geq2R_\eps,\ 2R_\eps\leq r\leq\delta,\ x_2\geq 0\}}
			\left(D\bigl[-2\log(\epsilon+r_\epsilon^2)\bigr]-D(-4\log r) \right)D\bigl(G(x,p)+4\log r\bigr)
			\,dx\\&\quad +\int_{\{r_\epsilon\geq2R_\eps,\ r\leq 2R_\eps,\ x_2\geq 0\}}
			D\bigl[-2\log(\epsilon+r_\epsilon^2)\bigr]D\bigl(G(x,p)+4\log r\bigr)
			\,dx\\&\quad -\int_{\{r\geq 2R_\eps, \ r_\epsilon\leq 2R_\epsilon,\ x_2\geq 0\}}
			D(-4\log r)D\bigl(G(x,p)+4\log r\bigr)
			\,dx+O(\delta\sqrt{\epsilon})+o(\sqrt{\epsilon}).
		\end{align*}
		Noting that $$\int_{\{r_\epsilon\geq2R_\eps,\ 2R_\eps\leq r\leq\delta,\ x_2\geq 0\}}
		\frac{x_1}{\epsilon+r_\epsilon^2}
		\,dx=\int_{\{r_\epsilon\geq2R_\eps,\ 2R_\eps\leq r\leq\delta,\ x_2\geq 0\}}
		\frac{x_1}{r^2}
		\,dx=0$$
		and
		$$\left|D\bigl[-2\log(\epsilon+r_\epsilon^2)\bigr]-D(-4\log r) \right|\leq C\frac{\epsilon+r\sqrt{\epsilon}}{r^3},$$ we have
		\begin{align*}
			&\int_{\{r_\epsilon\geq2R_\eps,\ 2R_\eps\leq r\leq\delta,\ x_2\geq 0\}}
			\left(D\bigl[-2\log(\epsilon+r_\epsilon^2)\bigr]-D(-4\log r) \right)D\bigl(G(x,p)+4\log r\bigr)
			\,dx\\
			&= \int_{\{r_\epsilon\geq2R_\eps,\ 2R_\eps\leq r\leq\delta,\ x_2\geq 0\}}
			\left(\frac{-4x_1}{\epsilon+r_\epsilon^2}-\frac{-4x_1}{r^2}\right)(b_1+O(r))
			\,dx\\
			&\quad + \int_{\{r_\epsilon\geq2R_\eps,\ 2R_\eps\leq r\leq\delta,\ x_2\geq 0\}}
			\left( \frac{-4(x_2+a\sqrt{\epsilon})}{\epsilon+r_\epsilon^2}-\frac{-4x_2}{r^2}\right)(b_2+O(r))
			\,dx\\
			&=b_2 \int_{\{r_\epsilon\geq2R_\eps,\ 2R_\eps\leq r\leq\delta,\ x_2\geq 0\}}
			\left( \frac{-4(x_2+a\sqrt{\epsilon})}{\epsilon+r_\epsilon^2}-\frac{-4x_2}{r^2}\right)
			\,dx+O(\delta\sqrt{\epsilon})+o(\sqrt{\epsilon}).
		\end{align*}

		Similarly, we have
		\begin{align*}
			&\int_{\{r_\epsilon\geq2R_\eps,\ r\leq 2R_\eps,\ x_2\geq 0\}}
			D\bigl[-2\log(\epsilon+r_\epsilon^2)\bigr]D\bigl(G(x,p)+4\log r\bigr)
			\,dx\\&= \int_{\{r_\epsilon\geq2R_\eps,\ r\leq 2R_\eps,\ x_2\geq 0\}}
			\left(\frac{-4x_1}{\epsilon+r_\epsilon^2}(b_1+O(r))+ \frac{-4(x_2+a\sqrt{\epsilon})}{\epsilon+r_\epsilon^2}(b_2+O(r))\right)dx\\&= b_2\int_{\{r_\epsilon\geq2R_\eps,\ r\leq 2R_\eps,\ x_2\geq 0\}}
			\frac{-4(x_2+a\sqrt{\epsilon})}{\epsilon+r_\epsilon^2}dx+ o(\sqrt{\epsilon})
		\end{align*}
		and
		\begin{align*}
			-\int_{\{r\geq 2R_\eps, \ r_\epsilon\leq 2R_\epsilon,\ x_2\geq 0\}}
			D(-4\log r)D\bigl(G(x,p)+4\log r\bigr)
			\,dx
			=-b_2\int_{\{r\geq2R_\eps,\ r_\epsilon\leq 2R_\eps,\ x_2\geq 0\}}
			\frac{-4x_2}{r^2}dx+ o(\sqrt{\epsilon}).
		\end{align*}
		Thus,
		\begin{align*}
			\mathbf{I_2}&= b_2\int_{\{r_\epsilon\geq2R_\eps,\ r\leq \delta,\ x_2\geq 0\}}
			\frac{-4(x_2+a\sqrt{\epsilon})}{\epsilon+r_\epsilon^2}dx- b_2\int_{\{2R_\eps\leq  r\leq \delta,\ x_2\geq 0\}}
			\frac{-4x_2}{r^2}dx+O(\delta\sqrt{\epsilon})+ o(\sqrt{\epsilon}).
		\end{align*}
		
		Setting $x_1=\sqrt{\epsilon}t\cos\theta$,   $x_2=\sqrt{\epsilon}t\sin\theta$, a direct computation yields
		\begin{align*}
			&-4 b_2\int_{\{r_\epsilon\geq2R_\eps,\ r\leq \delta,\ x_2\geq 0\}}
			\frac{x_2+a\sqrt{\epsilon}}{\epsilon+r_\epsilon^2}dx\\&=-4b_2\sqrt{\epsilon}\int_0^\pi \int^{\frac{\delta}{\sqrt{\epsilon}}}_{-a\sin\theta+\sqrt{4\alpha^2-a^2\cos^2\theta}} \frac{t(a+t\sin\theta)}{1+t^2+2at\sin\theta+a^2}dtd\theta\\&=-4b_2\sqrt{\epsilon}\int_0^\pi \int^{\frac{\delta}{\sqrt{\epsilon}}}_{-a\sin\theta+\sqrt{4\alpha^2-a^2\cos^2\theta}} \left(\sin\theta+\frac{a}{t}\right)\left(1-\frac{2a\sin\theta}{t}+O(t^{-2})\right)dtd\theta\\&=-4b_2\sqrt{\epsilon}\int_0^\pi \int^{\frac{\delta}{\sqrt{\epsilon}}}_{-a\sin\theta+\sqrt{4\alpha^2-a^2\cos^2\theta}} \left(\sin\theta+\frac{a(1-2\sin^2\theta)}{t}+O(t^{-2})\right)dtd\theta\\
			&=-8b_2\delta+16b_2\alpha\sqrt{\epsilon}-2a\pi b_2\sqrt{\epsilon}+o(\sqrt{\epsilon}).
		\end{align*}
		Then we get
		\begin{align*}
			\mathbf{I_2}&= - b_2\int_{\{2R_\eps\leq  r\leq \delta,\ x_2\geq 0\}}
			\frac{-4x_2}{r^2}dx-8b_2\delta+16b_2\alpha\sqrt{\epsilon}-2a\pi b_2\sqrt{\epsilon}+O(\delta\sqrt{\epsilon})+ o(\sqrt{\epsilon})\\
			&=-2a\pi b_2\sqrt{\epsilon}+O(\delta\sqrt{\epsilon})+ o(\sqrt{\epsilon}).
		\end{align*}
		
		For $\mathbf{I_3}$, we have
		\begin{align*}
			\mathbf{I_3}&=\frac12\int_{M\cap\{r_\epsilon\geq 2R_\eps,\ r\leq 2R_\eps\}}
			\left|\nabla_g\bigl(G(x,p)+4\log r\bigr)\right|^2\,dv_g \\&\quad -
			\frac12\int_{M\cap\{r\geq 2R_\eps,\  r_\epsilon\leq 2R_\eps\}}
			\left|\nabla_g\bigl(G(x,p)+4\log r\bigr)\right|^2\,dv_g\\
			&=\frac12\int_{M\cap\{r_\epsilon\geq 2R_\eps,\ r\leq 2R_\eps\}}
			\left|D\bigl(G(x,p)+4\log r\bigr)\right|^2\,dv_g \\&\quad -
			\frac12\int_{M\cap\{r\geq 2R_\eps,\  r_\epsilon\leq 2R_\eps\}}
			\left|D\bigl(G(x,p)+4\log r\bigr)\right|^2\,dv_g+o(\sqrt{\epsilon})=o(\sqrt{\epsilon}).
		\end{align*}
		
		Thus,
		\begin{align*}
			&\frac12\int_{M\cap\{r_\epsilon\geq2R_\eps,\ r\leq\delta\}}
			|\nabla_g\phi_{\epsilon,\delta}|^2\,dv_g-\frac12\int_{M\cap\{2R_\eps\leq r\leq\delta\}}
			|\nabla_gG|^2\,dv_g \\
			&=-\frac{8a}{\alpha}-\frac{2\pi}{\alpha^2}+32a\frac{\sqrt{\epsilon}}{\delta}-2a\pi b_2\sqrt{\epsilon} +
			O(\delta\sqrt{\epsilon})
			+o(\sqrt{\epsilon}).
		\end{align*}
		
		\
		
		\noindent\textbf{Step 3:} Compute $\frac12\int_{M\cap\{ r\geq\delta\}}
		|\nabla_g\phi_{\epsilon,\delta}|^2\,dv_g-\frac12\int_{M\cap\{r\geq\delta\}}
		|\nabla_gG|^2\,dv_g $.
		
		\
		
		In this case, we have
		\begin{align*}
			&\frac12\int_{M\cap\{ r\geq\delta\}}
			|\nabla_g\phi_{\epsilon,\delta}|^2\,dv_g-\frac12\int_{M\cap\{r\geq\delta\}}
			|\nabla_gG|^2\,dv_g\\
			&=\int_{M\cap\{ r\geq\delta\}}
			\nabla_g G \nabla_g\left(\eta_\delta (4\log r-2\log (\epsilon+r_\epsilon^2))\right)\,dv_g+ \frac12\int_{M\cap\{ r\geq\delta\}}\left|\nabla_g\left(\eta_\delta (4\log r-2\log (\epsilon+r_\epsilon^2))\right)\right|^2dv_g.
		\end{align*}
		
		On one hand, noting that $$4\log r-2\log (\epsilon+r_\epsilon^2)=O(\frac{\sqrt{\epsilon}}{\delta})$$ and
		\begin{align*}
			\left|\nabla_g\left(4\log r-2\log (\epsilon+r_\epsilon^2)\right)\right|^2&= \left|D\left(4\log r-2\log (\epsilon+r_\epsilon^2)\right)\right|^2+(g^{-2}-r^{-2})\left|\frac{\partial}{\partial\theta}\left(-2\log (\epsilon+r_\epsilon^2)\right)\right|^2\\
			&=O(\frac{\epsilon}{\delta^2})=o(\sqrt{\epsilon}),
		\end{align*}
		we have
		\begin{equation}
			\int_{M\cap\{ r\geq\delta\}}\left|\nabla_g\left(\eta_\delta (4\log r-2\log (\epsilon+r_\epsilon^2))\right)\right|^2dv_g=o(\sqrt{\epsilon}).
		\end{equation}
		
		On the other hand, integrating by parts, we get
		\begin{align*}
			&\int_{M\cap\{ r\geq\delta\}}
			\nabla_g G \nabla_g\left(\eta_\delta (4\log r-2\log (\epsilon+r_\epsilon^2))\right)\,dv_g\\
			&=-\frac{4\pi}{|M|} \int_{M\cap\{ r\geq\delta\}}
			\eta_\delta (4\log r-2\log (\epsilon+r_\epsilon^2))\,dv_g +\int_{M\cap\{ r=\delta\}}
			(-\frac{\partial}{\partial r} G) \left(\eta_\delta (4\log r-2\log (\epsilon+r_\epsilon^2))\right)\,ds_g.
		\end{align*}
		Since $$2\log (\epsilon+r_\epsilon^2)-4\log r=\frac{4a\sqrt{\epsilon}x_2}{r^2}+O(\frac{\epsilon}{r^2}),\ for\ r\geq\delta,$$ we have
		\begin{align*}
			\int_{M\cap\{ r\geq\delta\}}
			\eta_\delta (4\log r-2\log (\epsilon+r_\epsilon^2))\,dv_g &=\int_{\{ \delta\leq r\leq 2\delta\}}
			\eta_\delta (4\log r-2\log (\epsilon+r_\epsilon^2))\,dv_g =O(\delta\sqrt{\epsilon})+o(\sqrt{\epsilon})
		\end{align*} and
		\begin{align*}
			&\int_{M\cap\{ r=\delta\}}
			(\frac{\partial}{\partial r} G) \left(\eta_\delta (-4\log r+2\log (\epsilon+r_\epsilon^2))\right)\,ds_g\\
			&= \int_{M\cap\{ r=\delta\}}\left(-\frac{4}{\delta}+b_1\cos\theta+b_2\sin\theta+O(\delta) \right) \left( \frac{4a\sqrt{\epsilon}\delta \sin\theta}{\delta^2}+O(\frac{\epsilon}{\delta^2})\right)(\delta+O(\delta^3))d\theta\\
			&=
			-4a\sqrt{\epsilon}\int_{\theta_+(\delta)}^{\pi-\theta_-(\delta)}
			\left[
			\frac{4\sin\theta}{\delta}
			-b_1(p)\sin\theta\cos\theta
			-b_2(p)\sin^2\theta
			+O(\delta)
			\right]d\theta\\
			&=
			-\frac{32a}{\delta}\sqrt{\epsilon}
			+
			2\pi a b_2(p)\sqrt{\epsilon}
			+
			O(\delta\sqrt{\epsilon})
			+
			o_\epsilon(\sqrt{\epsilon}),
		\end{align*} where $\theta_{\pm}(\delta)=\frac{\kappa(p)}{2}\delta+O(\delta^2)$ which are defined  so that the two endpoints of the circular arc $r=\delta$ correspond to the polar angles $\theta_+(\delta)$ and $\pi-\theta_-(\delta)$ respectively.
		
		Thus,
		\begin{align*}
			&\frac12\int_{M\cap\{ r\geq\delta\}}
			|\nabla_g\phi_{\epsilon,\delta}|^2\,dv_g-\frac12\int_{M\cap\{r\geq\delta\}}
			|\nabla_gG|^2\,dv_g\\
			&=-\frac{32a}{\delta}\sqrt{\epsilon}
			+
			2\pi a b_2(p)\sqrt{\epsilon}
			+
			O(\delta\sqrt{\epsilon})
			+
			o_\epsilon(\sqrt{\epsilon}).
		\end{align*}
		
		Combining these together, we arrive at
		\begin{align*}
			&\frac12\int_{M\cap\{r_\epsilon\geq2R_\eps\}}
			|\nabla_g\phi_{\epsilon,\delta}|^2\,dv_g\\&
			=
			-8\pi\log(2R_\eps)+2\pi A(p)
			-\frac{8a}{\alpha}-\frac{2\pi}{\alpha^2}
			+16\bigl(b_2(p)+\kappa(p)\bigr)R_\eps
			+O(\delta\sqrt{\epsilon})
			+o_\epsilon(\sqrt{\epsilon}).
		\end{align*}
		
	\end{proof}

	By Lemma \ref{lem:bubble-gradient}-Lemma \ref{lem:exterior-gradient}, we get
	\begin{align}
		\frac12\int_M
		|\nabla_g\phi_{\epsilon,\delta}|^2\,dv_g
		&=
		-4\pi\log\epsilon
		-4\pi
		-8\pi\log\bigl(a+\sqrt{1+a^2}\bigr)
		+\frac{4\pi a}{\sqrt{1+a^2}}
		+2\pi A(p) \notag\\
		&\quad
		+
		\left[
		4\pi\sqrt{1+a^2}\,b_2(p)
		+
		\frac{2\pi(2a^2+3)}
		{\sqrt{1+a^2}}\kappa(p)
		\right]\sqrt{\epsilon}
		+
		O(\delta\sqrt{\epsilon})
		+
		o_\epsilon(\sqrt{\epsilon}).
		\label{eq:total-gradient-expansion}
	\end{align}

	\
	
	\begin{lem}\label{lem:bubble-neck-integral}
		As $\epsilon\to0$, we have
		\begin{equation}\label{eq:bubble-neck-integral}
			\int_{M}
			\phi_{\epsilon,\delta}\,dv_g
			=|M|\log\epsilon+O(\delta\sqrt{\epsilon})+o(\sqrt{\epsilon}).
		\end{equation}
	\end{lem}
	
	\begin{proof}
		Firstly, noting that $$|\phi_{\epsilon,\delta}|\leq -2\log (\epsilon+ r_\epsilon^2)+\log\epsilon+C,\ \ for\ \ r_\epsilon\leq 2\alpha\sqrt{\epsilon},$$
		we have
		\begin{align*}
			\int_{M\cap\{r_\epsilon<2R_\eps\}}
			|\phi_{\epsilon,\delta}|\,dv_g &\leq \int_{\{r_\epsilon<2R_\eps,\ x_2\geq 0\}}
			(-2\log (\epsilon+ r_\epsilon^2)+\log\epsilon+C)\,dx +o(\sqrt\epsilon)\\
			&=-\pi\epsilon\left( 4\alpha^2\log\epsilon+(1+4\alpha^2)\log(1+4\alpha^2)-4\alpha^2\right)+2\pi\alpha^2\epsilon\log\epsilon +C\alpha^2\epsilon+o(\sqrt\epsilon) \\
			&=-4\pi \alpha^2\epsilon\log (\alpha^2\epsilon)+2\pi\alpha^2\epsilon\log\epsilon+o(\sqrt{\epsilon})\\
			&=-4\pi \alpha^2\epsilon\log (\alpha^2\sqrt{\epsilon})+o(\sqrt{\epsilon})=o(\sqrt{\epsilon}).
		\end{align*}
		
		Secondly,
		\begin{align*}
			\int_{M\cap\{r_\epsilon\geq 2R_\eps\}}
			\phi_{\epsilon,\delta}\,dv_g &=|M|\log\epsilon-\int_{M\cap\{r_\epsilon\leq 2R_\eps\}}
			(G+\log\epsilon)\,dv_g +\int_{M\cap\{r_\epsilon\geq 2R_\eps\}}
			\eta_\delta(4\log r-2\log(\epsilon+r_\epsilon^2))\,dv_g \\
			&=|M|\log\epsilon-\int_{M\cap\{r_\epsilon\leq 2R_\eps\}}
			(G+\log\epsilon+4\log r-2\log(\epsilon+r_\epsilon^2))\,dv_g \\&\quad +\int_{M\cap\{r\leq 2\delta\}}
			\eta_\delta(4\log r-2\log(\epsilon+r_\epsilon^2))\,dv_g\\
			&=|M|\log\epsilon +\int_{M\cap\{r\leq 2\delta\}}
			\eta_\delta(4\log r-2\log(\epsilon+r_\epsilon^2))\,dv_g+o(\sqrt{\epsilon})\\
			&=|M|\log\epsilon +\int_{M\cap\{100\sqrt\epsilon\leq r\leq 2\delta\}}
			\eta_\delta(4\log r-2\log(\epsilon+r_\epsilon^2))\,dv_g+o(\sqrt{\epsilon}),
		\end{align*} where the last equality follows from the fact that $$\int_{M\cap\{ r\leq 100\sqrt\epsilon\}}
		\eta_\delta(4\log r-2\log(\epsilon+r_\epsilon^2))\,dv_g=O(\epsilon\log\epsilon)=o(\sqrt{\epsilon}).$$
		
		Noting that
		\begin{align*}
			|4\log r-2\log(\epsilon+r_\epsilon^2)|&=2\log\left(1+\frac{2a\sqrt{\epsilon}x_2+(a^2+1)\epsilon}{r^2}\right)\\
			&\leq 2\frac{2a\sqrt{\epsilon}x_2+(a^2+1)\epsilon}{r^2},
		\end{align*} for $ x_2\geq 0$, we get
		\begin{align*}
			\int_{M\cap\{100\sqrt\epsilon\leq r\leq 2\delta\}}
			|4\log r-2\log(\epsilon+r_\epsilon^2)|\,dv_g&\leq C\int_{M\cap\{100\sqrt\epsilon\leq r\leq 2\delta\}}(\frac{\sqrt{\epsilon}}{r}+\frac{\epsilon}{r^2})dv_g\\ &=O(\delta\sqrt{\epsilon})+o(\sqrt{\epsilon}).
		\end{align*}
		
	\end{proof}

	\
	
	\begin{lem}\label{lem:bubble-exponential-integral}
		As $\epsilon\to0$, we have
		\begin{align*}
			\log
			\int_MK(x)e^{\phi_{\epsilon,\delta}}\,dv_g
			&=
			A(p)
			+
			\log K(p)
			+
			\log \frac{\pi (\sqrt{1+a^2}-a)}{2\sqrt{1+a^2}}\\&\quad
			+
			\left[
			\sqrt{1+a^2}
			\left(
			\frac{K_2(p)}{K(p)}+b_2(p)
			\right)
			-
			\frac{\kappa(p)}
			{2\bigl(\sqrt{1+a^2}-a\bigr)}
			\right]\sqrt{\epsilon}
			+
			o(\sqrt{\epsilon}).
		\end{align*}
	\end{lem}
	
	\begin{proof}
		\noindent\textbf{Step 1:} Compute $\int_{M\cap \{r_\epsilon\leq R_\epsilon\}}K(x)e^{\phi_{\epsilon,\delta}}\,dv_g$.
		
		\
		
		Since
		\begin{align}\label{equation:12}
			K(x)e^{\phi_{\epsilon,\delta}}&=\epsilon K(p)e^{A(p)}\frac{1}{(\epsilon+r_\epsilon^2)^2}e^{(b_1+\frac{K_1(p)}{K(p)})x_1+ (b_2+\frac{K_2(p)}{K(p)})x_2+O(r^2)}\notag\\
			&=\epsilon K(p)e^{A(p)}\frac{1+(b_1+\frac{K_1(p)}{K(p)})x_1+ (b_2+\frac{K_2(p)}{K(p)})x_2}{(\epsilon+r_\epsilon^2)^2}+O(1)
		\end{align} and
		\begin{align*}
			\int_{\{r_\epsilon\leq R_\epsilon,\ 0\leq x_2\leq f(x_1)\}}\frac{1}{r_\epsilon}dv_g&\leq C\int_0^{R_\epsilon}\int_0^{f(x_1)} \frac{1}{r_\epsilon}dv_g=O(R_\epsilon^2),
		\end{align*}
		we have
		\begin{align*}
			&\int_{M\cap \{r_\epsilon\leq R_\epsilon\}}K(x)e^{\phi_{\epsilon,\delta}}\,dv_g \\&= \epsilon K(p)e^{A(p)} \int_{M\cap \{r_\epsilon\leq R_\epsilon\}}\frac{1+(b_1+\frac{K_1(p)}{K(p)})x_1+ (b_2+\frac{K_2(p)}{K(p)})x_2}{(\epsilon+r_\epsilon^2)^2}dx+O(R_\epsilon^2)\\
			&=  \epsilon K(p)e^{A(p)} \left(\int_{ \{r_\epsilon\leq R_\epsilon,\ x_2\geq 0\}}\frac{(b_1+\frac{K_1(p)}{K(p)})x_1+ (b_2+\frac{K_2(p)}{K(p)})x_2}{(\epsilon+r_\epsilon^2)^2}dx +\int_{ M\cap \{r_\epsilon\leq R_\epsilon\}} \frac{1}{(\epsilon+r_\epsilon^2)^2}dx\right)+o(\sqrt{\epsilon})\\
			&=  \epsilon K(p)e^{A(p)} \left(\int_{ \{r_\epsilon\leq R_\epsilon,\ x_2\geq 0\}}\frac{ (b_2+\frac{K_2(p)}{K(p)})x_2}{(\epsilon+r_\epsilon^2)^2}dx + \int_{ M\cap \{r_\epsilon\leq R_\epsilon\}} \frac{1}{(\epsilon+r_\epsilon^2)^2}dx\right)+o(\sqrt{\epsilon}).
		\end{align*}
		Set $x_1=\sqrt{\epsilon}y_1$, $x_2+a\sqrt{\epsilon}=\sqrt{\epsilon}y_2$. On one hand, by \eqref{equation:07}, we have
		\begin{align}\label{equation:13}
			\int_{|y|\leq \alpha,\ y_2\geq a}\frac{1}{(1+|y|^2)^2}dy
			&=\int_{|y|\leq \alpha,\ y_2\geq 0}\frac{1}{(1+|y|^2)^2}dy -\int_{|y|\leq \alpha,\ 0\leq y_2\leq a}\frac{1}{(1+|y|^2)^2}dy\notag \\
			&=\frac{\pi}{2}\frac{\alpha^2}{\alpha^2+1}-\frac{\pi}{2}\frac{a}{\sqrt{a^2+1}}+O(\frac{1}{\alpha^3})\notag\\&= \frac{\pi}{2}\left(1-\frac{a}{\sqrt{a^2+1}}\right)-\frac{\pi}{2\alpha^2}+o(\sqrt{\epsilon}).
		\end{align}
		Noting that
		\begin{align*}
			\int_{|y|\leq \alpha,\ a\leq y_2\leq a+d_\epsilon(y_1)}\frac{1}{(1+|y|^2)^2}dy
			&=\int_{|y|\leq \alpha,\ a\leq y_2\leq a+d_\epsilon(y_1)}\frac{1}{(1+a^2+y_1^2)^2}dy+o(\sqrt{\epsilon})\\
			&=2\int_0^{\sqrt{\alpha^2-a^2}}\frac{\frac{1}{2}\kappa \sqrt{\epsilon}y_1^2+O(\epsilon y_1^3)}{(1+a^2+y_1^2)^2}dy_1+o(\sqrt{\epsilon})\\
			&=\frac{\pi}{4}\frac{1}{\sqrt{a^2+1}}\kappa\sqrt{\epsilon}+o(\sqrt{\epsilon}),
		\end{align*}
		then
		\begin{align*}
			\epsilon\int_{ M\cap \{r_\epsilon\leq R_\epsilon\}}\frac{ 1}{(\epsilon+r_\epsilon^2)^2}dx &=\int_{|y|\leq \alpha,\ y_2\geq a}\frac{1}{(1+|y|^2)^2}dy-\int_{|y|\leq \alpha,\ a\leq y_2\leq a+d_\epsilon(y_1)}\frac{1}{(1+|y|^2)^2}dy \\
			&=\frac{\pi}{2}\left(1-\frac{a}{\sqrt{a^2+1}}\right)-\frac{\pi}{2\alpha^2}-\frac{\pi}{4}\frac{1}{\sqrt{a^2+1}}\kappa\sqrt{\epsilon}+o(\sqrt{\epsilon}).
		\end{align*}

		On the other hand, there holds
		\begin{align*}
			\epsilon\int_{  \{r_\epsilon\leq R_\epsilon,\ x_2\geq 0\}}\frac{ x_2}{(\epsilon+r_\epsilon^2)^2}dx&=\sqrt{\epsilon}\int_{|y|\leq \alpha,\ y_2\geq a}\frac{y_2-a}{(1+|y|^2)^2}dy\\
			&=\sqrt{\epsilon}\int_{|y|\leq \alpha,\ y_2\geq a}\frac{y_2}{(1+|y|^2)^2}dy-a\sqrt{\epsilon}\frac{\pi}{2}\left(1-\frac{a}{\sqrt{a^2+1}}\right)+o(\sqrt{\epsilon}).
		\end{align*}
		
		A direct computation yields
		\begin{align*}
			\int_{|y|\leq \alpha,\ y_2\geq a}\frac{y_2}{(1+|y|^2)^2}dy&=\int_{|y|\leq \alpha,\ y_2\geq 0}\frac{y_2}{(1+|y|^2)^2}dy -\int_{|y|\leq \alpha,\ 0\leq y_2\leq a}\frac{y_2}{(1+|y|^2)^2}dy
			\\&=\arctan\alpha-\frac{\alpha}{\alpha^2+1}-\int_{|y_1|\leq \alpha,\ 0\leq y_2\leq a}\frac{y_2}{(1+|y|^2)^2}dy+O(\alpha^{-3})\\
			&=-\frac{\alpha}{\alpha^2+1}+\frac{1}{\sqrt{1+a^2}}\arctan\frac{\alpha}{\sqrt{1+a^2}}+O(\alpha^{-3})\\
			&=\frac{\pi}{2}\frac{1}{\sqrt{1+a^2}}+O(\alpha^{-1}).
		\end{align*}
		
		Then
		\begin{align*}
			\epsilon\int_{  \{r_\epsilon\leq R_\epsilon,\ x_2\geq 0\}}\frac{ x_2}{(\epsilon+r_\epsilon^2)^2}dx=\frac{\pi \sqrt{\epsilon}}{2}(\sqrt{a^2+1}-a)+o(\sqrt{\epsilon}).
		\end{align*}

		Thus,
		\begin{align}
			\int_{M\cap\{r_\epsilon\leq R_\eps\}}
			K(x)e^{\phi_{\epsilon,\delta}}\,dv_g &=
			K(p)e^{A(p)}
			\Bigg\{
			\frac{\pi}{2}
			\left(
			1-\frac{a}{\sqrt{1+a^2}}
			\right)
			-\frac{\pi}{2\alpha^2} \notag\\
			&\quad+
			\frac{\pi}{2}
			\left[
			\bigl(\sqrt{1+a^2}-a\bigr)
			\left(
			\frac{K_2(p)}{K(p)}+b_2(p)
			\right)
			-\frac{\kappa(p)}{2\sqrt{1+a^2}}
			\right]\sqrt{\epsilon}
			+o(\sqrt{\epsilon})
			\Bigg\}.
		\end{align}
		
		\
		
		\noindent\textbf{Step 2:} Compute $\int_{M\cap\{R_\epsilon\leq r_\epsilon\leq 2R_\eps\}}
		K(x)e^{\phi_{\epsilon,\delta}}\,dv_g$.
		
		\
		
		Since for $R_\epsilon\leq r_\epsilon\leq 2R_\eps$, we have
		\begin{equation}\label{equation:14}
			\phi_{\epsilon,\delta}
			=
			-2\log\bigl(\epsilon+r_\epsilon^2\bigr)
			+A(p)+b_1(p)x_1+b_2(p)x_2
			+\log\epsilon+O(R_\eps^2).
		\end{equation}
		Then by Step 1 of this lemma, we get
		\begin{align*}
			\int_{M\cap\{R_\epsilon\leq r_\epsilon\leq 2R_\eps\}}
			K(x)e^{\phi_{\epsilon,\delta}}\,dv_g&=\int_{M\cap\{R_\epsilon\leq r_\epsilon\leq 2R_\eps\}}
			K(x)e^{-2\log(\epsilon+r_\epsilon^2)
				+A(p)+b_1(p)x_1+b_2(p)x_2
				+\log\epsilon}\,dv_g+O(R_\eps^2)\\
			&=K(p)e^{A(p)}(-\frac{\pi}{2})(\frac{1}{4\alpha^2}-\frac{1}{\alpha^2})+o(\sqrt{\epsilon})\\
			&=K(p)e^{A(p)}\frac{3\pi}{8\alpha^2}+o(\sqrt{\epsilon}).
		\end{align*}
		
		\
		
		\noindent\textbf{Step 3:} Compute $\int_{M\cap\{ r_\epsilon\geq 2R_\eps\}}
		K(x)e^{\phi_{\epsilon,\delta}}\,dv_g$.
		
		\
		
		On one hand, it is easy to see that
		\begin{align*}
			\left|\int_{M\cap\{ r\geq \delta\}}
			K(x)e^{\phi_{\epsilon,\delta}}\,dv_g\right|&\leq C(\delta)\epsilon=o(\sqrt{\epsilon}).
		\end{align*}
		
		On the other hand, noting that for $ r\leq \delta$, there holds $$
		\phi_{\epsilon,\delta}
		=
		-2\log\bigl(\epsilon+r_\epsilon^2\bigr)
		+A(p)+b_1(p)x_1+b_2(p)x_2
		+\log\epsilon+O(r^2).
		$$
		Then
		\begin{align*}
			\int_{M\cap\{r_\epsilon\geq 2R_\epsilon, \ r\leq \delta\}}
			K(x)e^{\phi_{\epsilon,\delta}}\,dv_g&=\epsilon K(p)e^{A(p)}\int_{M\cap\{r_\epsilon\geq 2R_\epsilon, \ r\leq \delta\}}\frac{1+O(r)}{(\epsilon+r_\epsilon^2)^2}dx+o(\sqrt{\epsilon})\\
			&=\epsilon K(p)e^{A(p)}\int_{M\cap\{r_\epsilon\geq 2R_\epsilon, \ r\leq \delta\}}\frac{1}{(\epsilon+r_\epsilon^2)^2}dx+o(\sqrt{\epsilon})\\
			&=\epsilon K(p)e^{A(p)}\int_{M\cap\{r_\epsilon\geq 2R_\epsilon \}}\frac{1}{(\epsilon+r_\epsilon^2)^2}dx+o(\sqrt{\epsilon})\\
			&=\epsilon K(p)e^{A(p)}\int_{\{r_\epsilon\geq 2R_\epsilon,\ x_2\geq 0\}}\frac{1}{(\epsilon+r_\epsilon^2)^2}dx+o(\sqrt{\epsilon}).
		\end{align*}
		where we used the fact that
		\begin{align*}
			\epsilon \int_{\{r_\epsilon\geq 2R_\epsilon, \ 0\leq x_2\leq f(x_1)\}}\frac{1}{(\epsilon+r_\epsilon^2)^2}dx&\leq \epsilon \int_{\{R_\epsilon\leq r,\ 0\leq x_2\leq f(x_1)\}}\frac{1}{(\epsilon+r_\epsilon^2)^2}dx =o(\sqrt{\epsilon}).
		\end{align*}
		By \eqref{equation:13}, we get
		\begin{align*}
			\epsilon K(p)e^{A(p)}\int_{\{r_\epsilon\geq 2R_\epsilon,\ x_2\geq 0\}}\frac{1}{(\epsilon+r_\epsilon^2)^2}dx& = K(p)e^{A(p)}\int_{|y|\geq 2\alpha,\ y_2\geq a}\frac{1}{(1+|y|^2)^2}dy\\
			&=K(p)e^{A(p)} \frac{\pi}{8\alpha^2}+o(\sqrt{\epsilon}).
		\end{align*}
		
		Then
		\begin{align*}
			\int_MK(x)e^{\phi_{\epsilon,\delta}}\,dv_g
			&=
			\frac{\pi}{2}K(p)e^{A(p)}
			\left(
			1-\frac{a}{\sqrt{1+a^2}}
			\right) \notag\\
			&\quad
			+
			\frac{\pi}{2}K(p)e^{A(p)}
			\left[
			\bigl(\sqrt{1+a^2}-a\bigr)
			\left(
			\frac{K_2(p)}{K(p)}+b_2(p)
			\right)
			-
			\frac{\kappa(p)}{2\sqrt{1+a^2}}
			\right]\sqrt{\epsilon}
			+
			o(\sqrt{\epsilon})
			.
		\end{align*}

	\end{proof}
	
	\

	\begin{lem}\label{lem:bubble-boundary-exponential-integral}
		As $\epsilon\to0$, we have
		\begin{align*}
			&\log
			\int_{\partial M}
			h(x)e^{\phi_{\epsilon,\delta}/2}\,ds_g \\
			&=
			\frac{A(p)}{2}
			+
			\log h(p)
			+
			\log\frac{\pi}{\sqrt{1+a^2}} \notag\\
			&\quad
			+
			\Bigg[
			-\frac{a\kappa(p)}{2}
			+
			\frac{\sqrt{1+a^2}}
			{\pi e^{A(p)/2}h(p)}
			\left(
			\int_{\partial M\cap\{r\geq\delta\}}
			h(x)e^{G(x,p)/2}\,ds_g
			-
			\frac{2e^{A(p)/2}h(p)}{\delta}
			\right)
			\Bigg]\sqrt{\epsilon} \notag\\
			&\quad
			+
			O(\delta\sqrt{\epsilon})
			+
			o(\sqrt{\epsilon}).
		\end{align*}
	\end{lem}
	\begin{proof}
		\noindent\textbf{Step 1:} Compute $\int_{\partial M\cap\{r_\epsilon\leq R_\eps\}}
		h(x)e^{\phi_{\epsilon,\delta}/2}\,ds_g$.
		
		\
		
		Since
		\begin{align*}
			h(x)e^{\frac{1}{2}\phi_{\epsilon,\delta}}&=\sqrt{\epsilon} h(p)e^{\frac{1}{2}A(p)}\frac{1+(\frac{1}{2}b_1+\frac{h_1(p)}{h(p)})x_1+ O(x^2_1)}{\epsilon+r_\epsilon^2},\ \ x\in \partial M\cap\{r_\epsilon\leq R_\eps\}
		\end{align*} and
		\begin{align*}
			ds_g=\left|\frac{\partial}{\partial x_1}+f'(x_1)\frac{\partial}{\partial x_2}\right|_gdx_1=(1+O(x_1^2))dx_1,
		\end{align*}
		we have
		\begin{align*}
			&\int_{\partial M\cap\{r_\epsilon\leq R_\eps\}}
			h(x)e^{\phi_{\epsilon,\delta}/2}\,ds_g = \sqrt{\epsilon} h(p)e^{\frac{1}{2}A(p)}\int_{\partial M\cap\{r_\epsilon\leq R_\eps\}}\frac{1+(\frac{1}{2}b_1+\frac{h_1(p)}{h(p)})x_1}{\epsilon+r_\epsilon^2}dx_1+o(\sqrt{\epsilon}).
		\end{align*}
		
		Let $(l_\epsilon^+,f(l_\epsilon^+))$, $(l_\epsilon^-,f(l_\epsilon^-))$ be the two endpoints of
		$\partial M\cap\{r_\epsilon= R_\eps\}$ where $l_\epsilon^-<0<l_\epsilon^+$. Then
		$$
		(l_\epsilon^\pm)^2
		+
		\bigl(f(l_\epsilon^\pm)+a\sqrt{\epsilon}\bigr)^2
		=
		\alpha^2\epsilon,
		$$ which implies that $$\sqrt{\alpha^2-a^2}\sqrt{\epsilon}-|l_\epsilon^{\pm}|=\frac{1}{2}a\kappa \sqrt{\alpha^2-a^2}\epsilon+O((\alpha\sqrt{\epsilon})^3).$$
		
		Since
		\begin{align*}
			\sqrt{\epsilon}\left|\int_{\partial M\cap\{r_\epsilon\leq R_\eps\}}\left(\frac{x_1}{\epsilon+r_\epsilon^2}- \frac{x_1}{\epsilon+x_1^2+(\frac{1}{2}\kappa x_1^2)^2+a^2\epsilon}\right)dx_1\right| &\leq C\sqrt{\epsilon} \int_{\partial M\cap\{r_\epsilon\leq R_\eps\}}\frac{x_1^6+\sqrt{\epsilon}|x_1^3|}{(\epsilon+x_1^2)^2}dx_1\\&=o(\sqrt{\epsilon})
		\end{align*} and
		\begin{align*}
			\sqrt{\epsilon}\int_{\partial M\cap\{r_\epsilon\leq R_\eps\}} \frac{x_1}{\epsilon+x_1^2+(\frac{1}{2}\kappa x_1^2)^2+a^2\epsilon}dx_1 &=\sqrt{\epsilon}\int_{l^-_\epsilon}^{l^+_\epsilon} \frac{x_1}{\epsilon+x_1^2+(\frac{1}{2}\kappa x_1^2)^2+a^2\epsilon}dx_1=o(\sqrt{\epsilon}),
		\end{align*}where we used the symmetric property of the integration, thus we get
		\begin{align*}
			\int_{\partial M\cap\{r_\epsilon\leq R_\eps\}}
			h(x)e^{\phi_{\epsilon,\delta}/2}\,ds_g &= \sqrt{\epsilon} h(p)e^{\frac{1}{2}A(p)}\int_{\partial M\cap\{r_\epsilon\leq R_\eps\}}\frac{1}{\epsilon+r_\epsilon^2}dx_1+o(\sqrt{\epsilon})\\
			&= \sqrt{\epsilon} h(p)e^{\frac{1}{2}A(p)}\int_{\partial M\cap\{|x_1|\leq \sqrt{\alpha^2-a^2}\sqrt{\epsilon}\}}\frac{1}{\epsilon+r_\epsilon^2}dx_1+o(\sqrt{\epsilon})
		\end{align*} where the last equality follows from
		\begin{align*}
			\sqrt{\epsilon}\int_{\partial M\cap\{|x_1|\leq \sqrt{\alpha^2-a^2}\sqrt{\epsilon},\ r_\epsilon\geq R_\epsilon\}}\frac{1}{\epsilon+r_\epsilon^2}dx_1\leq C\frac{\sqrt{\epsilon}}{R_\epsilon^2}(\sqrt{\alpha^2-a^2}\sqrt{\epsilon}-|l_\epsilon^{\pm}|)=O(\sqrt{\epsilon}\alpha^{-1})=o(\sqrt{\epsilon})
		\end{align*}
		
		Set $x_1=\sqrt{\epsilon}t$, then
		\begin{align*}
			&\sqrt{\epsilon} \int_{\partial M\cap\{|x_1|\leq \sqrt{\alpha^2-a^2}\sqrt{\epsilon}\}}\frac{1}{\epsilon+r_\epsilon^2}dx_1\\
			&=\int_{|t|\leq \sqrt{\alpha^2-a^2}}\frac{1}{a^2+1+t^2+a\kappa \sqrt{\epsilon}t^2+\epsilon\cdot (O(t^3)+O(t^4))}dt\\
			&=\int_{|t|\leq \sqrt{\alpha^2-a^2}}\left[\frac{1}{1+a^2+t^2}
			-
			a\kappa\sqrt{\epsilon}
			\frac{t^2}{(1+a^2+t^2)^2}+
			O\left(
			\epsilon\frac{|t|^3+|t|^4}{(1+a^2+t^2)^2}
			+
			\epsilon\frac{t^4}{(1+a^2+t^2)^3}
			\right)\right]dt\\
			&=\frac{2}{\sqrt{1+a^2}}\arctan \frac{\sqrt{\alpha^2-a^2}}{\sqrt{1+a^2}}-\frac{a\kappa}{\sqrt{1+a^2}}\arctan \frac{\sqrt{\alpha^2-a^2}}{\sqrt{1+a^2}}\sqrt{\epsilon}+o(\sqrt{\epsilon})\\
			&=\frac{\pi}{\sqrt{1+a^2}}
			-\frac{2}{\alpha}
			-\frac{\pi a\kappa}{2\sqrt{1+a^2}}
			\sqrt{\epsilon}+o(\sqrt{\epsilon}).
		\end{align*}
		Then
		\begin{align*}
			&\int_{\partial M\cap\{r_\epsilon\leq R_\eps\}}
			h(x)e^{\phi_{\epsilon,\delta}/2}\,ds_g =
			e^{A(p)/2}h(p)
			\left[
			\frac{\pi}{\sqrt{1+a^2}}
			-\frac{2}{\alpha}
			-\frac{\pi a\kappa(p)}{2\sqrt{1+a^2}}
			\sqrt{\epsilon}
			+
			o(\sqrt{\epsilon})
			\right].
		\end{align*}
		
		\
		
		\noindent\textbf{Step 2:} Compute $\int_{\partial M\cap\{R_\epsilon\leq r_\epsilon\leq 2R_\eps\}}
		h(x)e^{\phi_{\epsilon,\delta}/2}\,ds_g$.
		
		\
		
		By \eqref{equation:14} and Step 1, we get
		$$\int_{\partial M\cap\{R_\epsilon\leq r_\epsilon\leq 2R_\eps\}}
		h(x)e^{\phi_{\epsilon,\delta}/2}\,ds_g=e^{A(p)/2}h(p)\frac{1}{\alpha}+o(\sqrt{\epsilon}).$$

		\
		
		\noindent\textbf{Step 3:} Compute $\int_{\partial M\cap\{ r_\epsilon\geq 2R_\eps\}}
		h(x)e^{\phi_{\epsilon,\delta}/2}\,ds_g$.
		
		\
		
		On one hand, noting that
		\begin{align*}
			\int_{\partial M\cap\{ \delta\leq r\leq 2\delta\}}
			h(x)e^{\phi_{\epsilon,\delta}/2}\,ds_g&=\sqrt{\epsilon}\int_{\partial M\cap\{ \delta\leq r\leq 2\delta\}}h(x)e^{\frac{1}{2}G}\left(1+O(\frac{\sqrt{\epsilon}}{\delta})+o(\sqrt{\epsilon})\right)ds_g\\
			&=\sqrt{\epsilon}\int_{\partial M\cap\{\delta\leq r\leq 2\delta\}}h(x)e^{\frac{1}{2}G}ds_g+o(\sqrt{\epsilon}),
		\end{align*}
		we have
		\begin{align*}
			\int_{\partial M\cap\{ r\geq\delta\}}
			h(x)e^{\phi_{\epsilon,\delta}/2}\,ds_g&=\sqrt{\epsilon}\int_{\partial M\cap\{ r\geq \delta\}}h(x)e^{\frac{1}{2}G}ds_g+o(\sqrt{\epsilon}).
		\end{align*}
		
		On the other hand, we get
		\begin{align*}
			&\int_{\partial M\cap\{ r_\epsilon\geq 2R_\epsilon,\  r\leq \delta\}}
			h(x)e^{\phi_{\epsilon,\delta}/2}\,ds_g\\
			&=\sqrt{\epsilon} h(p)e^{\frac{1}{2}A(p)}\int_{\partial M\cap\{r_\epsilon\geq 2R_\epsilon,\  r\leq \delta\}}\frac{1+(\frac{1}{2}b_1+\frac{h_1(p)}{h(p)})x_1+O(x_1^2)}{\epsilon+r_\epsilon^2}(1+O(x_1^2))dx_1\\
			&=\sqrt{\epsilon} h(p)e^{\frac{1}{2}A(p)}\int_{\partial M\cap\{r_\epsilon\geq 2R_\epsilon,\  r\leq \delta\}}\frac{1+(\frac{1}{2}b_1+\frac{h_1(p)}{h(p)})x_1}{\epsilon+r_\epsilon^2}dx_1+O(\delta\sqrt{\epsilon}).
		\end{align*}
		
		Let $(j_\epsilon^+,f(j_\epsilon^+))$, $(j_\epsilon^-,f(j_\epsilon^-))$ be the two endpoints of
		$\partial M\cap\{r_\epsilon= 2R_\eps\}$ where $j_\epsilon^-<0<j_\epsilon^+$. Let $(j_\delta^+,f(j_\delta^+))$, $(j_\delta^-,f(j_\delta^-))$ be the two endpoints of
		$\partial M\cap\{r= \delta\}$ where $j_\delta^-<0<j_\delta^+$. Then
		$$
		(j_\epsilon^\pm)^2
		+
		\bigl(f(j_\epsilon^\pm)+a\sqrt{\epsilon}\bigr)^2
		=
		4R_\epsilon^2,\ \ (j_\delta^\pm)^2
		+
		\bigl(f(j_\delta^\pm)\bigr)^2
		=
		\delta^2,
		$$ which implies that $$\sqrt{4\alpha^2-a^2}\sqrt{\epsilon}-|j_\epsilon^{\pm}|=\frac{1}{2}a\kappa \sqrt{4\alpha^2-a^2}\epsilon+O((\alpha\sqrt{\epsilon})^3),\ \ \delta-|j_\delta^\pm|=O(\delta^3).$$
		
		Since
		\begin{align*}
			&\sqrt{\epsilon}\left|\int_{\partial M\cap\{r_\epsilon\geq 2R_\eps,\ r\leq\delta\}}\left(\frac{x_1}{\epsilon+r_\epsilon^2}- \frac{x_1}{\epsilon+x_1^2+(\frac{1}{2}\kappa x_1^2)^2+a^2\epsilon}\right)dx_1\right| \\
			&\leq C\sqrt{\epsilon} \int_{\partial M\cap\{r_\epsilon\geq 2R_\eps,\ r\leq\delta\}}\frac{x_1^6+\sqrt{\epsilon}|x_1^3|}{(\epsilon+x_1^2)^2}dx_1=O(\delta\sqrt{\epsilon})+o(\sqrt{\epsilon})
		\end{align*} and
		\begin{align*}
			&\sqrt{\epsilon}\left|\int_{\partial M\cap\{r_\epsilon\geq 2R_\eps,\ r\leq\delta\}} \frac{x_1}{\epsilon+x_1^2+(\frac{1}{2}\kappa x_1^2)^2+a^2\epsilon}dx_1 \right|\\
			&\leq \sqrt{\epsilon}\left(\left|\int_{|j^-_\epsilon|}^{j^+_\epsilon} \frac{x_1}{\epsilon+x_1^2+(\frac{1}{2}\kappa x_1^2)^2+a^2\epsilon}dx_1\right|+ \left|\int_{|j^-_\delta|}^{j^+_\delta} \frac{x_1}{\epsilon+x_1^2+(\frac{1}{2}\kappa x_1^2)^2+a^2\epsilon}dx_1\right|\right)\\
			&=O(\delta\sqrt{\epsilon})+o(\sqrt{\epsilon}),
		\end{align*}where we used the symmetric property of the integration to get the first inequality, thus we get
		\begin{align*}
			\int_{\partial M\cap\{ r_\epsilon\geq 2R_\epsilon,\  r\leq \delta\}}
			h(x)e^{\phi_{\epsilon,\delta}/2}\,ds_g=\sqrt{\epsilon} h(p)e^{\frac{1}{2}A(p)}\int_{\partial M\cap\{r_\epsilon\geq 2R_\epsilon,\  r\leq \delta\}}\frac{1}{\epsilon+r_\epsilon^2}dx_1+O(\delta\sqrt{\epsilon})+o(\sqrt{\epsilon}).
		\end{align*}
		
		Setting $x_1=\sqrt{\epsilon}t$, we have
		\begin{align*}
			\sqrt{\epsilon} \int_{\partial M\cap\{j^+_\epsilon\leq x_1\leq j_\delta^+\}}\frac{1}{\epsilon+r_\epsilon^2}dx_1 &=\int_{\frac{j^+_\epsilon}{\sqrt{\epsilon}}}^{\frac{j^+_\delta}{\sqrt{\epsilon}}}\frac{dt}{a^2+1+t^2 +a\kappa\sqrt{\epsilon}t^2+\epsilon O(t^4)}\\ &=\int_{\frac{j^+_\epsilon}{\sqrt{\epsilon}}}^{\frac{j^+_\delta}{\sqrt{\epsilon}}}\left(\frac{1}{a^2+1+t^2} -\frac{a\kappa \sqrt{\epsilon}t^2}{(a^2+1+t^2)^2}+ \frac{ \epsilon O(t^4)}{(a^2+1+t^2)^2}\right)dt\\
			&=\int_{\frac{j^+_\epsilon}{\sqrt{\epsilon}}}^{\frac{j^+_\delta}{\sqrt{\epsilon}}}\frac{1}{a^2+1+t^2}dt+O(\delta\sqrt{\epsilon})+o(\sqrt{\epsilon})\\
			&=\frac{1}{2\alpha}-\frac{\sqrt{\epsilon}}{\delta}+O(\delta\sqrt{\epsilon})+o(\sqrt{\epsilon}).
		\end{align*}
		Similarly,
		\begin{align*}
			\sqrt{\epsilon} \int_{\partial M\cap\{j^-_\delta\leq x_1\leq j_\epsilon^-\}}\frac{1}{\epsilon+r_\epsilon^2}dx_1 =\frac{1}{2\alpha}-\frac{\sqrt{\epsilon}}{\delta}+O(\delta\sqrt{\epsilon})+o(\sqrt{\epsilon}).
		\end{align*}
		Thus,
		\begin{align*}
			\int_{\partial M\cap\{ r_\epsilon\geq 2R_\epsilon,\  r\leq \delta\}}
			h(x)e^{\phi_{\epsilon,\delta}/2}\,ds_g= h(p)e^{\frac{1}{2}A(p)}\left(\frac{1}{\alpha}-\frac{2\sqrt{\epsilon}}{\delta}\right)+O(\delta\sqrt{\epsilon})+o(\sqrt{\epsilon})
		\end{align*} and
		\begin{align*}
			&\int_{\partial M\cap\{ r_\epsilon\geq 2R_\epsilon\}}
			h(x)e^{\phi_{\epsilon,\delta}/2}\,ds_g= h(p)e^{\frac{1}{2}A(p)}\left(\frac{1}{\alpha}-\frac{2\sqrt{\epsilon}}{\delta}\right)+\sqrt{\epsilon}\int_{\partial M\cap\{ r\geq \delta\}}h(x)e^{\frac{1}{2}G}ds_g+O(\delta\sqrt{\epsilon})+o(\sqrt{\epsilon}).
		\end{align*}
		
		Therefore,
		\begin{align*}
			&\int_{\partial M}
			h(x)e^{\phi_{\epsilon,\delta}/2}\,ds_g \\
			&=
			\frac{\pi e^{A(p)/2}h(p)}{\sqrt{1+a^2}}
			+
			\Bigg[
			\int_{\partial M\cap\{r\geq\delta\}}
			h(x)e^{G(x,p)/2}\,ds_g
			-
			\frac{2e^{A(p)/2}h(p)}{\delta}
			-
			\frac{\pi a\kappa(p)e^{A(p)/2}h(p)}
			{2\sqrt{1+a^2}}
			\Bigg]\sqrt{\epsilon}\\
			&\quad
			+
			O(\delta\sqrt{\epsilon})
			+
			o_\epsilon(\sqrt{\epsilon}).
		\end{align*}

	\end{proof}
	
	\

	Now, we prove Proposition \ref{prop:01}.
	\begin{proof}[\textbf{Proof of Proposition \ref{prop:01}}]
		\begin{align*}
			&J_{\rho,\ \rho'}(\phi_{\epsilon,\delta})\\
			&=\frac{1}{2}\int_M|\nabla \phi_{\epsilon,\delta}|^2dv_g+\frac{4\pi}{|M|}\int_M \phi_{\epsilon,\delta} dv_g-2\rho\log\int_MK e^{\phi_{\epsilon,\delta}}dv_g-4\rho'\log\int_{\partial M}h e^{\frac{\phi_{\epsilon,\delta}}{2}}\,ds_g\\
			&=
			-4\pi
			-8\pi\log\bigl(a+\sqrt{1+a^2}\bigr)
			+\frac{4\pi a}{\sqrt{1+a^2}}
			+2\pi A(p) -2\rho\left[A(p)
			+
			\log K(p)
			+
			\log \frac{\pi (\sqrt{1+a^2}-a)}{2\sqrt{1+a^2}}\right] \\
			&\quad -4\rho'\left[\frac{A(p)}{2}
			+
			\log h(p)
			+
			\log\frac{\pi}{\sqrt{1+a^2}}\right]
			\\ &\quad
			+
			\sqrt{\epsilon} \Bigg\{
			4\pi\sqrt{1+a^2}\,b_2(p)
			+
			\frac{2\pi(2a^2+3)}
			{\sqrt{1+a^2}}\kappa(p)-2\rho \left[
			\sqrt{1+a^2}
			\left(
			\frac{K_2(p)}{K(p)}+b_2(p)
			\right)
			-
			\frac{\kappa(p)}
			{2\bigl(\sqrt{1+a^2}-a\bigr)}
			\right]
			\\
			&
			\quad -4\rho'
			\bigg[
			-\frac{a\kappa(p)}{2}
			+
			\frac{\sqrt{1+a^2}}
			{\pi e^{A(p)/2}h(p)}
			\left(
			\int_{\partial M\cap\{r\geq\delta\}}
			h(x)e^{G(x,p)/2}\,ds_g
			-
			\frac{2e^{A(p)/2}h(p)}{\delta}
			\right)
			\bigg]\Bigg\} \\
			&\quad
			+
			O(\delta\sqrt{\epsilon})
			+
			o(\sqrt{\epsilon}).
		\end{align*}
		By the definition of $a$, it is easy to see that
		\begin{align*}
			\frac{a}{\sqrt{1+a^2}}
			=&
			\frac{\rho'}{2\pi},\ \ \ \ \frac{a}{\sqrt{1+a^2}-a}
			=
			\frac{\rho'}{\rho},\ \
			\ \
			\sqrt{1+a^2}-a
			=
			\sqrt{\frac{\rho}{2\pi+\rho'}},
			\\
			a+\sqrt{1+a^2}
			=&
			\sqrt{\frac{2\pi+\rho'}{\rho}},
			\ \
			\frac{\pi}{\sqrt{1+a^2}}
			=
			\frac12\sqrt{\rho(2\pi+\rho')}.
		\end{align*}
		Then we get
		\begin{align*}
			&-4\pi
			-8\pi\log\bigl(a+\sqrt{1+a^2}\bigr)
			+\frac{4\pi a}{\sqrt{1+a^2}}
			+2\pi A(p) -2\rho\left[A(p)
			+
			\log K(p)
			+
			\log \frac{\pi (\sqrt{1+a^2}-a)}{2\sqrt{1+a^2}}\right] \\
			&\quad -4\rho'\left[\frac{A(p)}{2}
			+
			\log h(p)
			+
			\log\frac{\pi}{\sqrt{1+a^2}}\right]\\
			&=-4\pi
			+
			2\rho'
			+
			4\pi\log4
			-
			2(2\pi+\rho')\log(2\pi+\rho')
			-
			2\pi
			\left[
			A(p)
			+
			\frac{\rho}{\pi}\log K(p)
			+
			\frac{\rho'}{\pi}\log h^2(p)
			\right]
		\end{align*}
		and
		\begin{align*}
			&4\pi\sqrt{1+a^2}\,b_2(p)
			+
			\frac{2\pi(2a^2+3)}
			{\sqrt{1+a^2}}\kappa(p)
			-2\rho
			\left[
			\sqrt{1+a^2}
			\left(
			\frac{K_2(p)}{K(p)}+b_2(p)
			\right)
			-
			\frac{\kappa(p)}
			{2\bigl(\sqrt{1+a^2}-a\bigr)}
			\right]\\&\quad
			+
			2\rho'a\kappa(p)
			-
			\frac{4\rho'\sqrt{1+a^2}}
			{\pi e^{A(p)/2}h(p)}
			\left(
			\int_{\partial M\cap\{r\geq\delta\}}
			h(x)e^{G(x,p)/2}\,ds_g
			-
			\frac{2e^{A(p)/2}h(p)}{\delta}
			\right)\\
			&=
			-4\pi
			\sqrt{\frac{\rho}{2\pi+\rho'}}
			\Bigg[
			\frac{K_2(p)}{K(p)}
			-
			2\kappa(p)
			+
			\frac{2\rho'}
			{\pi\rho e^{A(p)/2}h(p)}
			\left(
			\int_{\partial M\cap\{r\geq\delta\}}
			h(x)e^{G(x,p)/2}\,ds_g
			-
			\frac{2e^{A(p)/2}h(p)}{\delta}
			\right)
			\Bigg],
		\end{align*} which implies the conclusion of the proposition.
		
	\end{proof}
	
	\
	
	\begin{lem}\label{lem:boundary-green-finite-part}
		Let $p\in\partial M$, let $r=d_g(x,p)$, and assume that
		$h\in C^2(\partial M)$ and $h>0$. Then the limit
		\begin{equation}\label{eq:boundary-green-finite-part}
			\lim_{\delta\downarrow0}
			\left[
			\int_{\partial M\cap\{r\geq\delta\}}
			h(x)e^{G(x,p)/2}\,ds_g
			-
			\frac{2e^{A(p)/2}h(p)}{\delta}
			\right]
		\end{equation}
		exists and is finite.
	\end{lem}
	
	\begin{proof}
		Fix $r_0>0$ sufficiently small so that geodesic normal coordinates
		centered at $p$ are defined in $\{r\leq r_0\}$ and
		$\partial M$ is given by
		$$
		x_2=f(x_1)
		=
		\frac{\kappa(p)}{2}x_1^2+O(|x_1|^3).
		$$
		For $0<s\leq r_0$, let $x_s^-<0<x_s^+$ be the $x_1$-coordinates of
		the two points in $\partial M\cap\{r=s\}$. Since
		$$
		s^2=(x_s^\pm)^2+f(x_s^\pm)^2,
		$$
		we have
		\begin{equation}\label{eq:boundary-intersection-expansion}
			x_s^+=s+O(s^3),
			\qquad
			x_s^-=-s+O(s^3).
		\end{equation}
		
		Denote $h_1=\frac{\partial }{\partial x_1}h(p)$ and $$H(x_1)dx_1:=h(x)e^{G(x,p)/2}\,ds_g-h(p)e^{\frac{1}{2}A(p)}\frac{1+(\frac{1}{2}b_1+\frac{h_1}{h(p)})x_1}{x_1^2}dx_1.$$
		Noting that
		\begin{align*}
			h(x)e^{G(x,p)/2}\,ds_g=h(p)e^{\frac{1}{2}A(p)}\frac{1+(\frac{1}{2}b_1+\frac{h_1}{h(p)})x_1+O(x_1^2)}{x_1^2}\left(1+O(x_1^2)\right)dx_1,\ \ \text{for }\delta\leq r\leq r_0,
		\end{align*}
		then $$|H(x_1)|\leq C(r_0).$$
		We have
		\begin{align*}
			&\int_{\partial M\cap\{\delta\leq r\leq r_0\}}
			h(x)e^{G(x,p)/2}\,ds_g
			-
			\frac{2e^{A(p)/2}h(p)}{\delta}\\ &=h(p)e^{\frac{1}{2}A(p)}\int_{\partial M\cap\{\delta\leq r\leq r_0\}}
			\frac{1+(\frac{1}{2}b_1+\frac{h_1}{h(p)})x_1}{x_1^2}dx_1
			-
			\frac{2e^{A(p)/2}h(p)}{\delta}+ \int_{\partial M\cap\{\delta\leq r\leq r_0\}}H(x_1)dx_1\\
			&=
			e^{A(p)/2}h(p)
			\left[
			\frac1{x_\delta^+}
			-
			\frac1{x_\delta^-}
			-
			\frac2\delta
			-
			\frac1{x_{r_0}^+}
			+
			\frac1{x_{r_0}^-}
			\right]+
			e^{A(p)/2}
			\left(
			h_1+\frac12b_1(p)h(p)
			\right)
			\log
			\frac{x_{r_0}^+|x_\delta^-|}
			{x_\delta^+|x_{r_0}^-|}\\
			&\quad +\int_{\partial M\cap\{\delta\leq r\leq r_0\}}H(x_1)dx_1
		\end{align*}
		
		By \eqref{eq:boundary-intersection-expansion},
		$$
		\frac1{x_\delta^+}
		-
		\frac1{x_\delta^-}
		-
		\frac2\delta
		=
		O(\delta),
		\qquad
		\log\frac{|x_\delta^-|}{x_\delta^+}
		=
		O(\delta^2).
		$$
		
		Hence,
		\begin{align*}
			&\lim_{\delta\downarrow0}
			\left[
			\int_{\partial M\cap\{\delta\leq r\leq r_0\}}
			h(x)e^{G(x,p)/2}\,ds_g
			-
			\frac{2e^{A(p)/2}h(p)}{\delta}
			\right]\\
			&=
			e^{A(p)/2}h(p)
			\left(
			-\frac1{x_{r_0}^+}
			+
			\frac1{x_{r_0}^-}
			\right)+
			e^{A(p)/2}
			\left(
			h_1+\frac12b_1(p)h(p)
			\right)
			\log\frac{x_{r_0}^+}{|x_{r_0}^-|}+\int_{\partial M\cap\{ r\leq r_0\}}H(x_1)dx_1.
		\end{align*}
		The right-hand side depends only on the fixed $r_0$ and is finite.
		Therefore, the above limit exists and is finite.
	\end{proof}

\

At the end of this section, we give the proof of Theorem \ref{thm:main-01} and Theorem \ref{thm:main-02}.
\begin{proof}[\textbf{Proof of Theorem \ref{thm:main-01} and Theorem \ref{thm:main-02}}]
We just need to show that the minimizer sequence $\{u_\epsilon\}$ given in Section \ref{sec:blow-up-analysis} does not blow up. Otherwise, on one hand, by Proposition \ref{prop:lower-bound}, we get
\begin{align*}
\inf_{u\in W^{1,2}(M)}J_{\rho,\rho'}(u)=\liminf_{\epsilon\to 0}J_{\rho_\epsilon,
			\rho'}(u_\epsilon)\geq& -4\pi+2\rho'+ 4\pi\log 4-2(2\pi+\rho')\log(2\pi+\rho')\notag\\&-2\pi \left(A(p) +\frac{\rho}{\pi}\log K(p)+\frac{\rho'}{\pi}\log h^2(p)\right).\end{align*}

On the other hand, by Proposition \ref{prop:01}, if we first take $\delta=\delta_0(M,\|K\|_{C^1(M)},\|h\|_{C^1(\partial M)})>0$ small, then for all $\epsilon$ sufficiently small, we have
\begin{align*}
J_{\rho,
			\rho'}(\phi_{\epsilon,\delta})<& -4\pi+2\rho'+ 4\pi\log 4-2(2\pi+\rho')\log(2\pi+\rho')\notag\\&-2\pi \left(A(p) +\frac{\rho}{\pi}\log K(p)+\frac{\rho'}{\pi}\log h^2(p)\right),\end{align*} which is a contradiction.
\end{proof}

	\

\section{A geometric application on the hemisphere}
	\label{sec:hemisphere-application}
	
\

In this section, we consider the case that $(M,g)$ is a spherical cap or a unit disk, where the integration constant $\Lambda(G,h)$ can be expressed by $(-\Delta_{\partial M})^{\frac{1}{2}}h(p)$. We also give some geometric applications to prescribed curvature problems when $M$ is topologically a disk. Theorem \ref{thm:main-03}-Theorem \ref{thm:main-06} are proved in this section.

\

We first give the proof of Theorem \ref{thm:main-03}.
	\begin{proof}[\textbf{Proof of Theorem \ref{thm:main-03}}]
		Assume now that
		$$
		M=\S^2\cap\{x_3\geq\cos\theta_0\},
		\qquad 0<\theta_0<\pi.
		$$
		Using the spherical parametrization
		$$
		x=(\sin\theta\cos\varphi,\sin\theta\sin\varphi,\cos\theta),\ 0\leq\theta\leq\theta_0,\ 0\leq\varphi<2\pi,
		$$
		then
		$$
		g=d\theta^2+\sin^2\theta\,d\varphi^2.
		$$
		Moreover,
		$$
		\begin{aligned}
			\partial M=\{\theta=\theta_0\},\
			|M|=2\pi(1-\cos\theta_0)
			=4\pi\sin^2\frac{\theta_0}{2},\
			ds_g=\sin\theta_0\,d\varphi.
		\end{aligned}
		$$
		The inward unit normal along $\partial M$ is $-\partial_\theta$. Since
		$$
		\nabla_{\frac{1}{\sin\theta_0}\partial_\varphi}
		\left(
		\frac{1}{\sin\theta_0}\partial_\varphi
		\right)
		=
		-\cot\theta_0\,\partial_\theta,
		$$
		the sign convention used above gives
		\begin{equation}\label{eq:spherical-cap-kappa}
			\kappa(p)=\cot\theta_0
			\qquad
			\mbox{for every }p\in\partial M.
		\end{equation}
		
		By a rotation, we may assume that
		$$
		p=(\sin\theta_0,0,\cos\theta_0).
		$$
		Introducing the conformal coordinate
		$$
		z
		=
		\frac{\tan(\theta/2)}{\tan(\theta_0/2)}e^{i\varphi},
		$$
		it maps $M$ onto the closed unit disk, sends $p$ to $1$, and satisfies
		\begin{equation}\label{eq:spherical-cap-conformal-metric}
			g
			=
			\frac{4\tan^2(\theta_0/2)}
			{\left(1+\tan^2(\theta_0/2)|z|^2\right)^2}
			|dz|^2.
		\end{equation}
		One can check that the normalized Green function $G(z,p)$ with pole at $p$ defined in \eqref{eq:intro-green-function} is
		\begin{align}
			G(z,p)
			={}&
			-4\log|z-1|
			+
			\frac{1}{\sin^2(\theta_0/2)}
			\log\left(1+\tan^2\frac{\theta_0}{2}|z|^2\right)
			\notag\\
			&-
			\frac{1}{\sin^2(\theta_0/2)}
			+
			\frac{\cos^2(\theta_0/2)}{\sin^4(\theta_0/2)}
			\log\frac{1}{\cos^2(\theta_0/2)}.
			\label{eq:spherical-cap-green-function}
		\end{align}

		On $\partial M$, formula \eqref{eq:spherical-cap-green-function} becomes
		\begin{align}
			G(e^{i\varphi},p)
			={}&
			-4\log|e^{i\varphi}-1|+
			\frac{
				\log\dfrac{1}{\cos^2(\theta_0/2)}
				-
				\sin^2(\theta_0/2)
			}
			{\sin^4(\theta_0/2)}.
			\label{eq:spherical-cap-green-boundary}
		\end{align}
		If $r=d_g(x,p)$ and
		$$
		x=(\sin\theta_0\cos\varphi,
		\sin\theta_0\sin\varphi,
		\cos\theta_0)\in\partial M,
		$$
		then
		\begin{equation}\label{eq:spherical-cap-distance-relation}
		\sin\frac r2
			=
			\sin\theta_0\left|\sin\frac\varphi2\right|,\ \cos r
		=
		\cos^2\theta_0+\sin^2\theta_0\cos\varphi.
		\end{equation}

		Comparing \eqref{eq:spherical-cap-green-boundary} with
		$G(x,p)=-4\log r+A(p)+O(r)$, by \eqref{eq:spherical-cap-distance-relation}, we obtain
		\begin{equation}\label{eq:spherical-cap-regular-part}
			A(p)
			=
			4\log(\sin\theta_0)
			+
			\frac{
				\log\dfrac{1}{\cos^2(\theta_0/2)}
				-
				\sin^2(\theta_0/2)
			}
			{\sin^4(\theta_0/2)}.
		\end{equation}
		In particular, $A(p)$ is a constant on $\partial M$. Furthermore,
		by \eqref{eq:spherical-cap-green-boundary} and
		\eqref{eq:spherical-cap-regular-part}, we have
		\begin{equation}\label{eq:spherical-cap-boundary-green-density}
			e^{G(e^{i\varphi},p)/2}\,ds_g
			=
			\frac{e^{A(p)/2}}
			{4\sin\theta_0\sin^2(\varphi/2)}\,d\varphi.
		\end{equation}
		
		We identify $h|_{\partial M}$ with the corresponding $2\pi$-periodic
		function $h(\varphi)$, so that $h(0)=h(p)$. For sufficiently small
		$\delta>0$, let $\varphi_\delta\in(0,\pi)$ be determined by
		\begin{equation}\label{eq:spherical-cap-cutoff-relation}
			\sin\frac\delta2
			=
			\sin\theta_0\sin\frac{\varphi_\delta}{2}.
		\end{equation}
		Then \eqref{eq:spherical-cap-distance-relation} gives
		$$
		\partial M\cap\{r\geq\delta\}
		=
		\left\{e^{i\varphi}:
		\varphi_\delta\leq\varphi\leq2\pi-\varphi_\delta
		\right\}.
		$$
		Using \eqref{eq:spherical-cap-boundary-green-density} and
		$$
		\int_{\varphi_\delta}^{2\pi-\varphi_\delta}
		\frac{d\varphi}{\sin^2(\varphi/2)}
		=
		4\cot\frac{\varphi_\delta}{2},
		$$
		we obtain
		\begin{align}
			&\int_{\partial M\cap\{r\geq\delta\}}
			h(x)e^{G(x,p)/2}\,ds_g
			-
			\frac{2e^{A(p)/2}h(p)}{\delta}
			\notag\\
			&\quad=
			\frac{e^{A(p)/2}}{4\sin\theta_0}
			\int_{\varphi_\delta}^{2\pi-\varphi_\delta}
			\frac{h(\varphi)-h(0)}{\sin^2(\varphi/2)}\,d\varphi
			+
			e^{A(p)/2}h(0)
			\left[
			\frac{1}{\sin(\delta/2)}
			\sqrt{
				1-
				\frac{\sin^2(\delta/2)}{\sin^2\theta_0}
			}
			-
			\frac2\delta
			\right]\notag\\
			&\quad=
			\frac{e^{A(p)/2}}{4\sin\theta_0}
			\int_{\varphi_\delta}^{2\pi-\varphi_\delta}
			\frac{h(\varphi)-h(0)}{\sin^2(\varphi/2)}\,d\varphi+O(\delta).
			\label{eq:spherical-cap-finite-part-reduction}
		\end{align}

Since the metric induced on $\partial M$ is	$	\sin^2\theta_0\,d\varphi^2	$, by Proposition A.2 in \cite{Riviere}, we have $$(-\Delta_{\partial M})^{1/2}h(p)=\frac{1}{\sin\theta_0}(-\Delta_{S^1})^{1/2}h(p)=\frac{1}{4\pi\sin\theta_0}
			\operatorname{PV}
			\int_0^{2\pi}
			\frac{h(0)-h(\varphi)}
			{\sin^2(\varphi/2)}\,d\varphi.$$

Then
\begin{align}\label{equation:21}
			\lim_{\delta\downarrow0}
			\left[
			\int_{\partial M\cap\{r\geq\delta\}}
			h(x)e^{G(x,p)/2}\,ds_g
			-
			\frac{2e^{A(p)/2}h(p)}{\delta}
			\right]&
			=
			-\frac{e^{A(p)/2}}{4\sin\theta_0}
			\operatorname{PV}
			\int_0^{2\pi}
			\frac{h(0)-h(\varphi)}
			{\sin^2(\varphi/2)}\,d\varphi\notag\\&=
			-\pi e^{A(p)/2}
			(-\Delta_{\partial M})^{1/2}h(p).
		\end{align}

		Since $A(p)$ is constant on $\partial M$, we may take $p$ to be a
		maximum point of
		$$
		\rho\log K+2\rho'\log h
		\qquad\mbox{on }\partial M.
		$$
		Substituting \eqref{eq:spherical-cap-kappa} and \eqref{equation:21}
		 into Theorem \ref{thm:main-01} and Theorem \ref{thm:main-02},
			we get the first and second conclusions of Theorem \ref{thm:main-03}.
		
		In particular, assume that $K$ and $h$ are positive constants. Then
		$$
		K_\nu(p)=0,
		\qquad
		(-\Delta_{\partial M})^{1/2}h=0,
		$$
		and the sufficient condition \eqref{equation:22} reduces to
		$$
		-2\cot\theta_0>0,
		$$ which yields the third conclusion.
		
	\end{proof}
	
	\

Let $M=\mathbb S^2_+=\S^2\cap\{x_3\geq0\}$ be the standard hemisphere,
equipped with the round metric $g$. Then $|M|=2\pi$,
$\Pi=1$ and $\kappa=0$. Following \cite{Cruz-Ruiz}, for $h>0$, we define
\begin{equation}\label{eq:joint-curvature-functional}
	\begin{aligned}
		\widehat J(u,\rho)
		={}&J_{\rho,\,2\pi-\rho}(u)+2\rho\log\rho
		+4(2\pi-\rho)\log(2\pi-\rho)+2\rho,
	\end{aligned}
\end{equation}
for $(u,\rho)\in W^{1,2}(M)\times[0,2\pi]$ with $\int_{\partial M}uds_g=0$, where we adopt the convention that $0\log 0\equiv 0$.

\begin{proof}[\textbf{Proof of Theorem \ref{thm:main-05}}]
	
Since the conclusion for $h\equiv 0$ follows immediately from Theorem \ref{thm:main-02}, we just prove the result for the case $h>0$.

For $0<t<1$, define
	$$
	\widehat J_t(u,\rho)
	=\widehat J(u,\rho)
	+\frac{t}{2}\int_M|\nabla_g u|^2\,dv_g.
	$$
	
	Since
	$$
	J_{\rho,\,2\pi-\rho}(u)
	=\frac{\rho}{2\pi}J_{2\pi,0}(u)
	+\left(1-\frac{\rho}{2\pi}\right)J_{0,2\pi}(u),
	$$ by the  Moser--Trudinger inequalities  \cite{ChangYang1988,LiLiu2005,Yang2006}
	$$
	J_{2\pi,0}(u)\geq-C,
	\qquad
	J_{0,2\pi}(u)\geq-C,
	\qquad u\in W^{1,2}(M),
	$$
	we get
	$$
	J_{\rho,\,2\pi-\rho}(u)\geq-C,\ \ u\in W^{1,2}(M),\ \rho\in[0,2\pi].
	$$
 Therefore,
	\begin{equation}\label{equation:18}
	\widehat J(u,\rho)\geq-C,
	\qquad
	\widehat J_t(u,\rho)
	\geq\frac{t}{2}\int_M|\nabla_g u|^2\,dv_g-C,\ \ u\in W^{1,2}(M),\ \rho\in[0,2\pi].
	\end{equation}
	 In particular, this implies the functional $\widehat J_t(u,\rho)$ is coercive for any fixed $t>0$ and the following infimum  exists		$$
	m:=\inf_{\substack{u\in W^{1,2}(M)\\0\leq\rho\leq2\pi}}
	\widehat J(u,\rho).
	$$
By the classical variational method, we can assume  $\widehat J_t(u,\rho)$ admits a joint global minimizer
	$(u_t,\rho_t)$. Then
	\begin{equation*}
		\begin{aligned}
			m&\leq\widehat J(u_t,\rho_t)
			\leq\widehat J_t(u_t,\rho_t)\leq\widehat J(v,\rho)
			+\frac{t}{2}\int_M|\nabla_gv|^2\,dv_g,\ \ \forall\  (v,\rho)\in W^{1,2}(M)\times[0,2\pi].
		\end{aligned}
	\end{equation*} That is $$m=\lim_{t\to 0}\widehat J(u_t,\rho_t)
			= \lim_{t\to 0}\widehat J_t(u_t,\rho_t).$$
	
	For fixed $t$, as $s\downarrow0$, we have
	$$
	\begin{aligned}
		\widehat J_t(u_t,s)-\widehat J_t(u_t,0)
		&=2s\log s+O(s),\\
		\widehat J_t(u_t,2\pi-s)-\widehat J_t(u_t,2\pi)
		&=4s\log s+O(s).
	\end{aligned}
	$$
	Thus minimality excludes both endpoints, and
	$0<\rho_t<2\pi$.
	
	The first variation in $u$ gives
	$$
	\begin{cases}
		\displaystyle
		-(1+t)\Delta_g u_t+2
		=2\rho_t\frac{Ke^{u_t}}{\int_MKe^{u_t}\,dv_g},
		&\text{in }M,\\[1.2em]
		\displaystyle
		(1+t)\partial_\nu u_t
		=2(2\pi-\rho_t)
		\frac{he^{u_t/2}}{\int_{\partial M}he^{u_t/2}\,ds_g},
		&\text{on }\partial M.
	\end{cases}
	$$
	Since the regularizing term is independent of $\rho$, the
	first variation in $\rho$ gives
	$\partial_\rho\widehat J(u_t,\rho_t)=0$. That is
	\begin{equation}\label{eq:joint-perturbed-compatibility}
		\frac{
			\left(\displaystyle\int_{\partial M}
			he^{u_t/2}\,ds_g\right)^2}
		{\displaystyle\int_MKe^{u_t}\,dv_g}
		=\frac{(2\pi-\rho_t)^2}{\rho_t}.
	\end{equation}
	Set
	$$
	w_t=u_t+
	\log\frac{\rho_t}
	{(1+t)\displaystyle\int_MKe^{u_t}\,dv_g}.
	$$
	Then \eqref{eq:joint-perturbed-compatibility} and the equation
	for $u_t$ yield
	\begin{equation}\label{eq:joint-perturbed-equation}
		\begin{cases}
			\displaystyle
			-\Delta_gw_t+\frac{2}{1+t}=2Ke^{w_t},
			&\text{in }M,\\[6pt]
			\displaystyle
			\partial_\nu w_t
			=\frac{2h}{\sqrt{1+t}}e^{w_t/2},
			&\text{on }\partial M.
		\end{cases}
	\end{equation}
	Moreover,
	\begin{equation}\label{eq:joint-perturbed-masses}
		\begin{aligned}
			\int_MKe^{w_t}\,dv_g=\frac{\rho_t}{1+t},\ \
			\int_{\partial M}\frac{h}{\sqrt{1+t}}
			e^{w_t/2}\,ds_g=\frac{2\pi-\rho_t}{1+t}.
		\end{aligned}
	\end{equation}

	Passing to a subsequence, assume that
	$\rho_t\to\rho_*\in[0,2\pi]$ as $t\downarrow0$. Then we have to consider the following two cases.

\

\noindent\textbf{Case 1.} The sequence $w_t$ is uniformly bounded above.

\

By standard elliptic estimates and
\eqref{eq:joint-perturbed-masses}, after passing to a subsequence,
we have
$$
w_t\longrightarrow w_*
\qquad\text{in }C^2(M).
$$
Passing to the limit in \eqref{eq:joint-perturbed-equation}
as $t\downarrow0$, we obtain
$$
\begin{cases}
	-\Delta_gw_*+2=2Ke^{w_*},
	&\text{in }M,\\
	\partial_\nu w_*=2he^{w_*/2},
	&\text{on }\partial M.
\end{cases}
$$
Thus $w_*$ is a smooth solution of the prescribed curvature
problem.

\

\noindent\textbf{Case 2.} The sequence $w_t$ blows up.

\

In this case, we will show that the infimum $m$ will be attained by $(u_*,\rho_*)\in W^{1,2}(M)\times (0,2\pi)$. We divide the remaining proof into the following steps.
	
\

	\noindent\textbf{Step 1:} The infimum $m$  satisfies $m<4\pi-8\pi\log\Phi(p)$.

\

	On the standard hemisphere, we have $A=4\log2-2$ and
	$\kappa=0$. Taking
	$$
	\rho_p=2\pi\left(
	1-\frac{h(p)}{\sqrt{K(p)+h(p)^2}}
	\right)\in(0,2\pi),
	$$
	and the test function $\phi_{\epsilon,\delta}$ centered at
	$p$ in \eqref{eq:test-function}, with parameters $(\rho_p,2\pi-\rho_p)$, by Proposition \ref{prop:01}
	and \eqref{equation:22} with
	$\theta_0=\pi/2$, we obtain
	$$
	\begin{aligned}
		\widehat J(\phi_{\epsilon,\delta},\rho_p)
		={}&4\pi-8\pi\log\Phi(p)-\frac{4\pi}{\Phi(p)\sqrt{K(p)}}
		\left[
		K_2(p)-2\Phi(p)(-\Delta_{\partial M})^{1/2}h(p)
		\right]\sqrt{\epsilon}\\
		&+O(\delta\sqrt{\epsilon})
		+o(\sqrt{\epsilon}).
	\end{aligned}
	$$
	By \eqref{eq:joint-boundary-condition}, first choosing
	$\delta>0$ sufficiently small and then $\epsilon>0$
	sufficiently small, we obtain
	\begin{equation}\label{eq:joint-strict-blow-up-level}
		m\leq\widehat J(\phi_{\epsilon,\delta},\rho_p)
		<4\pi-8\pi\log\Phi(p).
	\end{equation}
	
	\
	
\noindent\textbf{Step 2:} There exists $\rho_*\in (0,2\pi)$ such that $m=\inf_{u\in W^{1,2}(M)}\widehat J(u,\rho_*)$.

\

Passing to a subsequence, assume that
$$
\max_M w_t\longrightarrow+\infty,
\qquad
\rho_t\longrightarrow\rho_*\in[0,2\pi].
$$
Suppose $w_t$ blows up at $q$. Since
$$
\int_M2Ke^{w_t}\,dv_g
=\frac{2\rho_t}{1+t}\leq4\pi,
$$
it must be that $q\in\partial M$. Suppose $w_t(x_t)=\max_Mw_t=\lambda_t$. Setting $$U_t(x):=w_t (x_t+\delta_t x)+2\log\delta_t,\ \ where\ \ \delta_t=e^{-\frac{\lambda_t}{2}}.$$
By the argument in Section \ref{sec:blow-up-analysis}, we know $U_t$ converges strongly to $U$ in $C^2_{loc}(\R^2_+)$ where
$$
\begin{cases}
	-\Delta U=2K(q)e^U,
	&\text{in }\R^2_+,\\
	\partial_\nu U=2h(q)e^{U/2},
	&\text{on }\partial\R^2_+.
\end{cases}
$$
By \eqref{eq:joint-perturbed-masses},
we have
\begin{equation}\label{inequ:01}
	\int_{\R^2_+}K(q)e^U\,dx
	\leq\rho_*,\ \
	\int_{\R}h(q)e^{U(x_1,0)/2}\,dx_1
	\leq2\pi-\rho_*.
\end{equation}
The classification result in \cite{Li-Zhu} that
$$
\int_{\R^2_+}K(q)e^U\,dx
+\int_{\R}h(q)e^{U(x_1,0)/2}\,dx_1=2\pi,
$$
yields that both
inequalities in \eqref{inequ:01} are in fact equalities.

We now compute $\rho_*$. By \cite{Li-Zhu}, we have
$$
U(x)
=\log\frac{4\lambda^2}
{K(q)\left[
	\lambda^2+x_1^2+
	\left(x_2+\frac{h(q)\lambda}{\sqrt{K(q)}}\right)^2
	\right]^2},
\qquad \lambda=\frac{2\sqrt{K(q)}}{K(q)+h^2(q)}.
$$
Hence
\begin{equation}\label{eq:joint-limiting-mass}
	\rho_*=2\pi-\int_{\partial \R^2_+}h(q)e^{\frac{1}{2}U(x)}dx_1
	=2\pi\left(
	1-\frac{h(q)}{\sqrt{K(q)+h(q)^2}}
	\right)\in(0,2\pi).
\end{equation}

Noting that
$$
\frac{\partial}{\partial \rho}\widehat J(u_t,\rho_t)=\frac{\partial}{\partial \rho}\widehat J_t(u_t,\rho_t)=0,\ \frac{\partial^2}{\partial \rho^2}\widehat J(u_t,\rho_t)=\frac{4\pi+2\rho_t}{\rho_t(2\pi-\rho_t)}
$$
 and
$\rho_t\to\rho_*\in(0,2\pi)$, by Taylor's expansion, we get
$$
\widehat J(u_t,\rho_*)
=\widehat J(u_t,\rho_t)
+O(|\rho_t-\rho_*|^2)
\longrightarrow m,\ as \ t\to 0.
$$
Hence
$$
m\leq
\inf_{u\in W^{1,2}(M)}\widehat J(u,\rho_*)
\leq\widehat J(u_t,\rho_*)\longrightarrow m,
$$
and therefore
\begin{equation}\label{eq:joint-fixed-slice-infimum}
	\inf_{u\in W^{1,2}(M)}\widehat J(u,\rho_*)=m.
\end{equation}

\

\noindent\textbf{Step 3:} There exists
$u_*\in W^{1,2}(M)$ such that
$$
\widehat J(u_*,\rho_*)
=\inf_{u\in W^{1,2}(M)}\widehat J(u,\rho_*)=m.
$$

\

Let $u_\epsilon$
be a minimizer of
$
J_{\rho_*-\epsilon,\,2\pi-\rho_*},
$
normalized by \eqref{equation:03}.

\

\noindent\textbf{Claim:} Passing to a subsequence, there holds
$$
u_\epsilon\longrightarrow u_*
\qquad\text{strongly in }W^{1,2}(M).
$$
Then $u_*$
minimizes $J_{\rho_*,\,2\pi-\rho_*}$. Since
$\widehat J(\cdot,\rho_*)$ differs from
$J_{\rho_*,\,2\pi-\rho_*}$ only by a constant,
\eqref{eq:joint-fixed-slice-infimum} gives
$$
\widehat J(u_*,\rho_*)
=\inf_{u\in W^{1,2}(M)}\widehat J(u,\rho_*)=m.
$$

\

We prove the claim by a contradiction argument. Otherwise, $u_\epsilon$ blows up. By
Proposition \ref{prop:lower-bound} and the definition of $\widehat J(u,\rho_*)$,  noting that $A=4\log2-2$ for $M=\mathbb S^2_+$, we have
\begin{align*}
\widehat J(u,\rho_*)&=\lim_{\epsilon\to 0} J_{\rho_*-\epsilon,2\pi-\rho_*}(u)+2\rho_*\log\rho_*
		+4(2\pi-\rho_*)\log(2\pi-\rho_*)+2\rho_*\\
&\geq \lim_{\epsilon\to 0} J_{\rho_*-\epsilon,2\pi-\rho_*}(u_\epsilon)+2\rho_*\log\rho_*
		+4(2\pi-\rho_*)\log(2\pi-\rho_*)+2\rho_*\\
&\geq \min_{z\in\partial M}L(\rho_*,z)
\end{align*} where
$$
\begin{aligned}
	L(\rho,z)
	={}&4\pi-2(4\pi-\rho)\log(4\pi-\rho)
	+2\rho\log\rho\\
	&+4(2\pi-\rho)\log(2\pi-\rho)
	-2\rho\log K(z)-4(2\pi-\rho)\log h(z).
\end{aligned}
$$
Thus,
\begin{equation}\label{eq:joint-fixed-mass-threshold}
	m=\inf_{u\in W^{1,2}(M)}\widehat J(u,\rho_*)
	\geq\min_{z\in\partial M}L(\rho_*,z).
\end{equation}

For every fixed $z\in\partial M$, a direct computation gives
$$
\frac{\partial}{\partial\rho} L(\rho,z)
=2\log\frac{\rho(4\pi-\rho)h(z)^2}
{(2\pi-\rho)^2K(z)},\ \ \frac{\partial^2}{\partial \rho^2} L(\rho,z)>0.
$$
Thus the function $L(\cdot,z)$ has a unique minimum at
$$
\rho_z=2\pi\left(
1-\frac{h(z)}{\sqrt{K(z)+h(z)^2}}
\right).
$$
Therefore, by the choice of $p$, we have
\begin{equation}\label{eq:joint-optimal-blow-up-level}
	L(\rho,z)\geq L(\rho_z,z)=4\pi-8\pi\log\Phi(z)
	\geq4\pi-8\pi\log\Phi(p)
\end{equation}
for every $0<\rho<2\pi$ and $z\in\partial M$.
Then \eqref{eq:joint-fixed-mass-threshold}  gives
$$
m\geq4\pi-8\pi\log\Phi(p),
$$
 which is a contradiction to \eqref{eq:joint-strict-blow-up-level}.       Thus we proved the claim.

Consequently, there exists $u_*\in W^{1,2}(M)$ such that
\begin{equation}\label{eq:joint-limit-minimality}
	\widehat J(u_*,\rho_*)
	=\inf_{u\in W^{1,2}(M)}\widehat J(u,\rho_*)=m,
	\qquad 0<\rho_*<2\pi.
\end{equation}

\end{proof}

\

As a consequence of Theorem~\ref{thm:main-05},
we obtain the following existence criterion on the unit disk.

\begin{proof}[\textbf{Proof of Theorem \ref{thm:main-06}}]
	Let $\overline{\mathbb D}$ be the closed unit disk with the metric
	$$
	g=e^\varphi(dx_1^2+dx_2^2),
	\qquad
	\varphi(x)=2\log\frac{2}{1+|x|^2}.
	$$
	Consider the isometric map
	$$
	F(x_1,x_2)
	=\frac{1}{1+|x|^2}
	\left(2x_1,2x_2,1-|x|^2\right): (\overline{\mathbb D},g)\to (\S^2_+,g_{\S^2}).
	$$
	It is easy to see that
	$$
	F^{-1}(y_1,y_2,y_3)
	=\left(\frac{y_1}{1+y_3},\frac{y_2}{1+y_3}\right),\ \ \left\langle
	\frac{\partial F}{\partial x_i},
	\frac{\partial F}{\partial x_j}
	\right\rangle_{\R^3}
	=\frac{4}{(1+|x|^2)^2}\delta_{ij},
	\ \ i,j=1,2.
	$$

Denote $\nu$ and $\nu_g$  the outward unit normal vectors to
$\partial\mathbb D$ with respect to the Euclidean metric
and $g$ respectively. Since $\varphi=0$ on
$\partial\mathbb D$, we have
$$
\nu_g=e^{-\varphi/2}\nu=\nu
\qquad\text{on }\partial\mathbb D.
$$
Set
$$
\widetilde K=K\circ F^{-1},
\ \
\widetilde h=h\circ F^{-1},\ \ \widetilde\Phi
=\widetilde h+\sqrt{\widetilde h^2+\widetilde K}.
$$
Since $H=h$ on $\partial\mathbb D$, we have
$$
\widetilde\Phi\circ F=\Phi
\qquad\text{on }\partial\mathbb D.
$$
Hence $F(p)$ is a maximum point of $\widetilde\Phi$ on
$\partial M$ and
$$
\widetilde K_2(F(p))
=-\partial_{\nu_g}K(p)
=-\partial_\nu K(p).
$$

Recall that $H$ satisfies
$$
\begin{cases}
	\Delta H=0,
	&\text{in }\mathbb D,\\
	H=h,
	&\text{on }\partial\mathbb D,
\end{cases}
$$
where $\Delta$ is the Euclidean Laplacian.
By \cite[formula~(A.3) and Lemma~A.1]{CabreCsatoMas},
$$
(-\Delta_{\S^1})^{1/2}h=\partial_\nu H
\qquad\text{on }\partial\mathbb D.
$$

Differentiating $\Phi=H+\sqrt{H^2+K}$ in the Euclidean
outward normal direction, we obtain
$$
2\sqrt{H^2+K}\,\partial_\nu\Phi
=\partial_\nu K+2\Phi\,\partial_\nu H
\qquad\text{on }\partial\mathbb D.
$$
Recalling that
$$
\widetilde K_2(F(p))=-\partial_\nu K(p),
\ \
\widetilde\Phi \circ F=\Phi,\ \ \widetilde h\circ F=h
$$ and noting that $$
\left[
(-\Delta_{\partial M})^{1/2}\widetilde h
\right]\circ F
=(-\Delta_{\S^1})^{1/2}h
=\partial_\nu H
\qquad\text{on }\partial\mathbb D,
$$
we have
$$
\begin{aligned}
	\widetilde K_2(F(p))
	-2\widetilde\Phi(F(p))
	\left[
	(-\Delta_{\partial M})^{1/2}\widetilde h
	\right](F(p))
	&
	=-\partial_\nu K(p)-2\Phi(p)\partial_\nu H(p)\\
	&
	=-2\sqrt{h(p)^2+K(p)}\,\partial_\nu\Phi(p)>0.
\end{aligned}
$$
Since $F(p)$ is a maximum point of $\widetilde\Phi$ on
$\partial M$, all the assumptions of
Theorem~\ref{thm:main-05} are satisfied.
Then there exists a solution
$w\in C^\infty(\overline{\mathbb D})$ to
$$
\begin{cases}
	-\Delta_gw+2=2Ke^w,
	&\text{in }\mathbb D,\\
	\partial_{\nu_g}w=2he^{w/2},
	&\text{on }\partial\mathbb D,
\end{cases}
$$
where $g=e^\varphi(dx_1^2+dx_2^2)$.

Using the facts that
$$
-\Delta\varphi
=\frac{8}{(1+|x|^2)^2}
=2e^\varphi
\ \ \text{in }\mathbb D, \  \ \varphi=0,
\ \
\partial_\nu\varphi=-2
\ \ \text{on }\partial\mathbb D,
$$
it is easy to see that $u=w+\varphi$ satisfies
$$
\begin{cases}
	-\Delta u=2Ke^u,
	&\text{in }\mathbb D,\\
	\partial_\nu u+2=2he^{u/2},
	&\text{on }\partial\mathbb D,
\end{cases}
$$ which is the desired solution.
\end{proof}

\

	\section{Appendix}

	\begin{lem}\label{lem:03}		
		Let $(M,g)$ be a smooth Riemann surface, and let
		$\gamma:(-\epsilon_0,\epsilon_0)\to M$ be a $C^4$ curve
		parametrized by arc length. Set $p=\gamma(0)$, choose a smooth unit
		normal vector field $\nu(s)$ along $\gamma$, and define the signed
		geodesic curvature $\kappa(s)$ by
		$$
		\nabla_{\dot{\gamma}}\dot{\gamma}=\kappa(s)\nu(s).
		$$
		Let $e_1=\dot{\gamma}(0)$ and $e_2=\nu(0)$, and define geodesic normal
		coordinates centered at $p$ by
		$$
		(x_1,x_2)\longmapsto\exp_p(x_1e_1+x_2e_2).
		$$
		Then, for $|x_1|$ sufficiently small, $\gamma$ can be written as
		$x_2=f(x_1)$, where
		$$
		f(x_1)
		=
		\frac{\kappa(0)}{2}x_1^2
		+
		\frac{\kappa'(0)}{6}x_1^3
		+
		O(x_1^4).
		$$
		Here $\kappa'(0)=\left.\frac{d\kappa}{ds}\right|_{s=0}$.
	\end{lem}
	
	\begin{proof}
		Write $\gamma(s)=(x_1(s),x_2(s))$ in the above coordinates. Since
		$\gamma(0)=p$, $e_1=\dot{\gamma}(0)$, and $e_2=\nu(0)$, we have
		$$
		x_1(0)=x_2(0)=0,
		\qquad
		x_1'(0)=1,
		\qquad
		x_2'(0)=0.
		$$
		Moreover, geodesic normal coordinates satisfy
		$$
		g_{ij}(p)=\delta_{ij},
		\qquad
		\Gamma_{ij}^m(p)=0,
		$$ where $\Gamma_{ij}^m$ is the Christoffel symbols.
		
		Since
		$\dot{\gamma}=\sum_{i=1}^2x_i'\partial_{x_i}$, the local coordinate
		formula for the covariant derivative gives
		$$
		\left(\nabla_{\dot{\gamma}}\dot{\gamma}\right)^m
		=
		x_m''
		+
		\sum_{i,j=1}^2
		\Gamma_{ij}^m(\gamma(s))x_i'x_j'.
		$$
		Evaluating at $s=0$, all Christoffel symbols vanish, and hence
		$$
		\left(\nabla_{\dot{\gamma}}\dot{\gamma}\right)^m(p)
		=
		x_m''(0).
		$$
		Since
		$\nabla_{\dot{\gamma}}\dot{\gamma}(p)=\kappa(0)e_2$, it follows that
		$$
		x_1''(0)=0,
		\qquad
		x_2''(0)=\kappa(0).
		$$
		
		Metric compatibility and
		$\langle\dot{\gamma},\nu\rangle_g=0$ imply
		$\nabla_{\dot{\gamma}}\nu=-\kappa\dot{\gamma}$. Therefore,
		$$
		\nabla_{\dot{\gamma}}
		\left(\nabla_{\dot{\gamma}}\dot{\gamma}\right)
		=
		\kappa'\nu-\kappa^2\dot{\gamma}.
		$$
		At $s=0$, its two coordinate components are
		$-\kappa(0)^2$ and $\kappa'(0)$.
		
		On the other hand, differentiating the coordinate expression for
		$\nabla_{\dot{\gamma}}\dot{\gamma}$ and using
		$x_1'(0)=1$, $x_2'(0)=0$, and $\Gamma_{ij}^m(p)=0$, we obtain
		$$
		\left[
		\nabla_{\dot{\gamma}}
		\left(\nabla_{\dot{\gamma}}\dot{\gamma}\right)
		\right]^m(p)
		=
		x_m'''(0)
		+
		\partial_1\Gamma_{11}^m(p).
		$$
		The coordinate curve $t\mapsto(t,0)$ is a geodesic, so its geodesic
		equation gives $\Gamma_{11}^m(t,0)=0$. Consequently,
		$\partial_1\Gamma_{11}^m(p)=0$, and hence
		$$
		x_1'''(0)=-\kappa(0)^2,
		\qquad
		x_2'''(0)=\kappa'(0).
		$$
		
		Since $x_1'(0)=1$, the inverse function theorem allows us to write
		$x_2=f(x_1)$. Differentiating
		$x_2(s)=f(x_1(s))$ at $s=0$ gives
		$$
		f(0)=f'(0)=0,
		\qquad
		f''(0)=\kappa(0),
		\qquad
		f'''(0)=\kappa'(0).
		$$
		The conclusion follows from Taylor's theorem.
	\end{proof}
	
	\


\begin{thebibliography}{10}
\bibitem{Aubin}
T. Aubin, \emph{Meilleures constantes dans le th\'eor\`eme d'inclusion de Sobolev et un th\'eor\`eme de Fredholm non lin\'eaire pour la transformation conforme de la courbure scalaire}, J. Funct. Anal. 32(2), 148--174 (1979).

\bibitem{BattagliaLopezSoriano}
L. Battaglia and R. L\'opez-Soriano, \emph{A double mean field equation related to a curvature prescription problem}, J. Differential Equations 269(4), 2705--2740 (2020).

\bibitem{BattagliaLopezSoriano2025}
L. Battaglia and R. L\'opez-Soriano, \emph{A mean field problem approach for the double curvature prescription problem}, Commun. Contemp. Math. 27(8), 2550008, 22 pp. (2025).

\bibitem{Brendle}
S. Brendle, \emph{A family of curvature flows on surfaces with boundary}, Math. Z. 241(4), 829--869 (2002).

\bibitem{CabreCsatoMas}
X. Cabr\'e, G. Csat\'o and A. Mas, \emph{Periodic solutions to integro-differential equations: variational formulation, symmetry, and regularity}, Comm. Partial Differential Equations 50(1--2), 162--210 (2025).

\bibitem{ChangLiu1996}
K. Chang and J. Liu, \emph{A prescribing geodesic curvature problem}, Math. Z. 223(2), 343--365 (1996).

\bibitem{Chang1993}
S. Chang, M. Gursky and P. Yang, \emph{The scalar curvature equation on 2- and 3-spheres}, Calc. Var. Partial Differ. Equ. 1(2), 205--229 (1993).

\bibitem{ChangYang1987}
S. Chang and P. Yang, \emph{Prescribing Gaussian curvature on $\mathbb{S}^2$}, Acta Math. 159(3--4), 215--259 (1987).

\bibitem{ChangYang1988}
S. Chang and P. Yang, \emph{Conformal deformation of metrics on $S^2$}, J. Differential Geom. 27(2), 259--296 (1988).

\bibitem{Chen-Ding1987}
W. Chen and W. Ding, \emph{Scalar curvatures on $\mathbb{S}^2$}, Trans. Amer. Math. Soc. 303(1), 365--382 (1987).

\bibitem{Chen-Ding1988}
W. Chen and W. Ding, \emph{A problem concerning the scalar curvature on $\mathbb{S}^2$}, Kexue Tongbao (English Ed.) 33(7), 533--537 (1988).

\bibitem{Chen-Li}
W. Chen and C. Li, \emph{Classification of solutions of some nonlinear elliptic equations}, Duke Math. J. 63(3), 615--622 (1991).

\bibitem{Chen-Li1995}
W. Chen and C. Li, \emph{A necessary and sufficient condition for the Nirenberg problem}, Commun. Pure Appl. Math. 48(6), 657--667 (1995).

\bibitem{Cruz-Ruiz}
S. Cruz-Bl\'azquez and D. Ruiz, \emph{Prescribing Gaussian and geodesic curvature on disks}, Adv. Nonlinear Stud. 18(3), 453--468 (2018).

\bibitem{Riviere}
F. Da Lio, L. Martinazzi and T. Rivi\`ere, \emph{Blow-up analysis of a nonlocal Liouville-type equation}, Anal. PDE 8(7), 1757--1805 (2015).

\bibitem{DingJostLiWang}
W. Ding, J. Jost, J. Li and G. Wang, \emph{The differential equation $\Delta u=8\pi-8\pi he^u$ on a compact Riemann surface}, Asian J. Math. 1(2), 230--248 (1997).

\bibitem{Glutsyuk}
 A. Glutsyuk, \emph{On curves with the Poritsky property}, J. Fixed Point Theory Appl. 24(2), Paper No. 35, 60 pp, (2022).

\bibitem{GuoHu1995}
K. Guo and S. Hu, \emph{Conformal deformation of metrics on subdomains of surfaces}, J. Geom. Anal. 5(3), 395--410 (1995).

\bibitem{Hu-Bartsch}
Z. Hu and T. Bartsch, \emph{The Morse property of limit functions appearing in mean field equations on surfaces with boundary}, J. Geom. Anal. 34(7), Paper No. 220 (2024).

\bibitem{Jost-Zhou-Zhu}
J. Jost, C. Zhou and M. Zhu, \emph{The qualitative boundary behavior of blow-up solutions of the super-Liouville equations}, J. Math. Pures Appl. (9) 101(5), 689--715 (2014).

\bibitem{Kazdan-Warner}
J. Kazdan and F. Warner, \emph{Curvature functions for compact 2-manifolds}, Ann. Math. 99(1), 14--47 (1974).

\bibitem{LiLiu2005}
Y. Li and P. Liu, \emph{A Moser-Trudinger inequality on the boundary of a compact Riemann surface}, Math. Z. 250(2), 363--386 (2005).

\bibitem{Li-Zhu}
Y. Li and M. Zhu, \emph{Uniqueness theorems through the method of moving spheres}, Duke Math. J. 80(2), 383--417 (1995).

\bibitem{LiuLi1993}
J. Liu and P. Li, \emph{Nirenberg's problem on the 2-dimensional hemi-sphere}, Internat. J. Math. 4(6), 927--939 (1993).

\bibitem{LiuHuang2005}
P. Liu and W. Huang, \emph{On prescribing geodesic curvature on $D^2$}, Nonlinear Anal. 60(3), 465--473 (2005).

\bibitem{Lopez-Soriano-Malchiodi-Ruiz}
R. L\'opez-Soriano, A. Malchiodi and D. Ruiz, \emph{Conformal metrics with prescribed Gaussian and geodesic curvatures}, Ann. Sci. \'Ec. Norm. Sup\'er. (4) 55(5), 1289--1328 (2022).

\bibitem{LopezSorianoRuiz2016}
R. L\'opez-Soriano and D. Ruiz, \emph{Prescribing the Gaussian curvature in a subdomain of $\mathbb S^2$ with Neumann boundary condition}, J. Geom. Anal. 26(1), 630--644 (2016).

\bibitem{Ruiz2024}
D. Ruiz, \emph{Conformal metrics of the disk with prescribed Gaussian and geodesic curvatures}, Math. Ann. 390(1), 1049--1075 (2024).

\bibitem{Struwe2005}
M. Struwe, \emph{A flow approach to Nirenberg's problem}, Duke Math. J. 128(1), 19--64 (2005).

\bibitem{Struwe2024}
M. Struwe, \emph{The prescribed curvature flow on the disc}, Ars Inven. Anal., Paper No. 5, 59 pp. (2024).

\bibitem{Wang2001}
G. Wang, \emph{Nirenberg's problem on domains in the 2-sphere}, J. Geom. Anal. 11(4), 717--726 (2001).


\bibitem{Yang2006}
Y. Yang, \emph{Extremal functions for Moser-Trudinger inequalities on 2-dimensional compact Riemannian manifolds with boundary}, Internat. J. Math. 17(3), 313--330 (2006).

\bibitem{Zhang2015}
H. Zhang, \emph{Prescribing the boundary geodesic curvature on a compact scalar-flat Riemann surface via a flow method}, Pacific J. Math. 273(2), 307--330 (2015).

\end{thebibliography}
	
	\providecommand{\bysame}{\leavevmode\hbox to3em{\hrulefill}\thinspace}
	\providecommand{\MR}{\relax\ifhmode\unskip\space\fi MR }
	\providecommand{\MRhref}[2]{%
		\href{http://www.ams.org/mathscinet-getitem?mr=#1}{#2}
	}
	\providecommand{\href}[2]{#2}

\end{document}